\documentclass[reqno,a4paper, 11pt]{amsart}
\usepackage{comment}
\usepackage[a4paper=true,pdfpagelabels]{hyperref}
\usepackage{graphicx}

\usepackage[ansinew]{inputenc}
\usepackage{amsfonts,epsfig}
\usepackage{latexsym}
\usepackage{mathabx}
\usepackage{amsmath}
\usepackage{amssymb}
\usepackage{mathrsfs}

\usepackage{color}

\newtheorem{theorem}{Theorem}
\newtheorem{lemma}[theorem]{Lemma}
\newtheorem{corollary}[theorem]{Corollary}

\newtheorem{proposition}[theorem]{Proposition}

\newtheorem{lettertheorem}{Theorem}
\newtheorem{letterlemma}[lettertheorem]{Lemma}

\theoremstyle{definition}

\theoremstyle{remark}

\numberwithin{equation}{section}

\newcommand{\B}{\mathcal{B}}
\newcommand{\D}{\mathbb{D}}
\newcommand{\DD}{\widehat{\mathcal{D}}}
\newcommand{\Dd}{\widecheck{\mathcal{D}}}

\newcommand{\DDD}{\mathcal{D}}

\newcommand{\N}{\mathbb{N}}

\newcommand{\R}{\mathbb{R}}

\newcommand{\C}{\mathbb{C}}

\newcommand{\conz}{\overline{z}}

\renewcommand{\phi}{\varphi}

\newcommand{\T}{\mathbb{T}}

\newcommand{\whw}{\widehat{\omega}}
\newcommand{\whv}{\widehat{\nu}}
\newcommand{\whm}{\widehat{\mu}}
\newcommand{\whe}{\widehat{\eta}}
\newcommand{\veps}{\varepsilon}

\newcommand{\be}{\begin{equation}}
\newcommand{\ee}{\end{equation}}

     \def\om{\omega}      
       \def\t{\theta}       
                  \def\z{\zeta}

\renewcommand{\H}{\mathcal{H}}

\allowdisplaybreaks

\begin{document}
\title{Ces\`aro operator induced by a Bergman kernel}

\thanks{The first author is  supported by "Ministerio de
		Ciencia e Innovaci\'{o}n\rq\rq , Spain (Project PID2022-138342NB-I00) and the second author is supported by Finnish Cultural Foundation, North Karelia Regional fund and Academy of Finland 356029.}


\author[O. Blasco]{Oscar Blasco}
\address{Departamento de An\'alisis Matem\'atico. Universitat de Valencia. Dr Moliner 60, Burjassot, 46100, Valencia, Spain}
\email{oscar.blasco@uv.es}

\author[A. Pennanen]{Atte Pennanen}
\address{University of Eastern Finland, P.O.Box 111, 80101 Joensuu, Finland}
\email{atte.pennanen@uef.fi}

\begin{abstract}
Let $\mu$ be a positive Borel measure on $[0,1)$ and $\om$ a radial weight. In this paper we consider the Ces\`aro-type operator $C_{\mu,\om}$ induced by the reproducing kernel $B^{\om}$ of the weighted Bergman space $A^2_{\om}$, given by
\begin{equation*}
    C_{\mu,\om}(f)(z)=\int_{0}^{1}f(tz)B^{\om}_t(z)\,d\mu(t), \quad z \in \D,
\end{equation*}
for functions $f$ analytic in $\D$. Under the assumption that $\om$ satisfies a natural doubling property, we study the boundedness of $C_{\mu,\om}$ acting on several spaces of analytic functions, including Hardy spaces $H^p$ and weighted Bergman spaces $A^p_{\nu}$. For $0<p,q<\infty$ and a two-sided doubling weight $\nu$, we completely characterize when $C_{\mu,\om}: H^p \to H^q$ and $C_{\mu,\om}: A^p_{\nu} \to A^q_{\nu}$ are bounded in terms of the interplay of tail integrals or moments of the inducing weights and the measure $\mu$. Many of the results obtained are new even in the setting of standard weights or when the Bergman reproducing kernel is replaced by the Cauchy kernel. In addition, we consider $C_{\mu,\om}$ acting on $H^{\infty}$, Korenblum spaces and weighted Hardy spaces.
\end{abstract}

\maketitle

\section{Introduction}

The classical Ces\`aro operator has a rich history, having its roots in the Ces\`aro summation method shown in the late 1800s \cite{Cesaro}. This summation can be considered as an operator on the sequence spaces or on spaces of analytic function by action on the coefficients. The boundedness of this operator between $\ell^p$-spaces for $1<p<\infty$, follows from the results of Hardy \cite{Hardy} and Landau \cite{Landau}. As far as we are aware, formally, this summation was first considered as an operator by Brown, Halmos and Shields, who considered certain operator theoretic properties on sequence and function spaces \cite{BHS}. Since then, there have been numerous studies on the properties of the Ces\`aro operator and its generalizations due to its many interesting properties. For a recent enlightening survey on the subject, see \cite{Ross}.

In this paper we are interested in Ces\`aro-type operators acting on spaces of analytic functions. Let $\H(\D)$ denote the space of analytic functions in the unit disc $\D$ of the complex plane~$\C$.  For $0<p<\infty$, we denote the $L^p$-mean of $f \in \H(\D)$ at the radius $r \in [0,1)$ by
\begin{equation*}
    M_p^p(r,f)=\frac{1}{2\pi}\int_{0}^{2\pi}|f(re^{i\t})|^p\,d\t,
\end{equation*}
and $M_\infty(r,f)=\sup_{|z|=r}|f(z)|$. We define the Hardy space $H^p$ as the functions $f \in \H(\D)$ satisfying
\begin{equation*}
    \|f\|_{H^p}=\sup_{0\leq r<1}M_p(r,f)<\infty.
\end{equation*}
Moreover, $H^{\infty}$ denotes the space of bounded analytic functions. For information on Hardy spaces, see for instance \cite{Duren,Garnett,Koosis}.

We say that a function $\om$ is a weight if $\om: \D \to [0,\infty)$ is integrable. A weight $\om$ is radial if $\om(z)=\om(|z|)$ for every $z \in \D$. For $0<p<\infty$ and a weight $\om$, we define the weighted Lebesgue space $L^p_{\om}$ as the measurable functions $f$ satisfying
\begin{equation*}
    \|f\|_{L^p_{\om}}=\int_{\D}|f(z)|^p\om(z)\,dA(z)<\infty.
\end{equation*}
The corresponding weighted Bergman space $A^p_{\om}$ consists of analytic functions in $L^p_{\om}$, so $A^p_{\om}=L^p_{\om} \cap \H(\D).$
When the weight is standard, meaning the weight is given by $\om_{\alpha}(z)=(\alpha+1)(1-|z|^2)^{\alpha}$ for $\alpha>-1$, we denote $A^p_{\om_{\alpha}}=A^p_{\alpha}$. For more information on Bergman spaces, see for instance \cite{DurenSchuster,Heden}.

The classical Ces\`aro operator $C$ acting on analytic functions can be defined by the action on the coefficients. Given $f \in \H(\D)$, defined by $f(z)=\sum_{n=0}^{\infty}a_nz^n$, we have
\begin{equation*}
    C(f)(z)=\sum_{k=0}^{\infty}\frac{1}{k+1}\sum_{n=0}^{k}a_n z^k=\int_{0}^{1}\frac{f(tz)}{1-tz}\,dt, \quad z \in \D.
\end{equation*}
The boundedness of this operator on Hardy and Bergman spaces has been achieved due to the work of several authors, {see for instance  \cite{Andersen,Miao,Siskakis1,Siskakis2} and the references within. In \cite{Andersen}, Andersen considered a weighted variant and, perhaps surprisingly, showed that the $L^p$-means of this weighted Ces\`aro operator of $f$ at the radius $r$ are bounded by the $L^p$-means of $f$ at the radius $r$, which directly yields the boundedness for Hardy spaces, Bergman spaces and mixed norm spaces.

Very recently, for a positive Borel measure $\mu$ on $[0,1)$, the authors in \cite{GalanoGirelaMerchCesaro} considered the following operator
\begin{equation}\label{Eq: Cmu}
    C_{\mu}(f)(z)=\int_{0}^{1}\frac{f(t z)}{1-t z}\,d\mu(t), \quad z \in \D.
\end{equation}
This can be seen as a generalization of the usual Ces\`aro operator, which is just the special case when $d\mu(t)=\,dt$. From now on, we denote $\mu \in M^+([0,1))$ when $\mu$ is a positive Borel measure on $[0,1)$. The authors in \cite{GalanoGirelaMerchCesaro} proved that the boundedness of this operator acting on $H^p$ and Bergman spaces $A^q_{\alpha}$ for $p\ge 1$ and $q>1$ is equivalent to $\mu$ being a classical Carleson measure. They also considered the operator acting on $\B$, $BMOA$ and $H^\infty$. Since then, this operator induced by a measure has seen considerable amount of study, see for instance \cite{Bao,Lin,Xie}, and even further generalizations, for example weighted versions \cite{BlascoMas,GalaSisZhao}, and complex Borel measures on $[0,1)$ \cite{BlascoHardy} and on the disc \cite{BlascoDirichlet,GalanoGirelaMerchCesaro2}.

When the point evaluations are bounded on $A^2_{\om}$, the Riesz representation theorem guarantees the existence of the reproducing kernels $B^{\om}_z \in A^2_{\om}$, which satisfy
\begin{equation*}
    f(z)=\int_{\D}f(\z)B^{\om}_z(\z)\om(\z)\,dA(\z), \quad z \in \D, \quad f \in A^2_{\om}.
\end{equation*}
This is certainly true for any radial weight. For standard weights, these reproducing kernels have a concrete form given by
\begin{equation*}
    B^{\om_{\alpha}}_z(\z)=\frac{1}{(1-\conz \z)^{2+\alpha}}, \quad z,\z \in \D,
\end{equation*}
which is zero-free and easy to work with. On the other hand, for a general weight the situation is completely different. For any orthonormal basis $\{e_n\}$ of $A^2_{\om}$, we only have the representation
\begin{equation*}
    B^{\om}_z(\z)=\sum_{n=0}^{\infty}e_n(\z)\overline{e_n(z)}, \quad z,\z \in \D.
\end{equation*}
In the case of radial weights, choosing the basis given by orthonormal monomials, we obtain the representation
\begin{equation*}
    B^{\om}_z(\z)=\sum_{n=0}^{\infty}\frac{(\overline{z}\z)^n}{2\om_{2n+1}}, \quad z,\z \in \D,
\end{equation*}
where $\om_x=\int_{0}^{1}r^x\om(r)\,dr$ denotes the moments of $\om$. For simplicity, we denote $B^{\om}_1(z)=B^{\om}(z)$ for every $z \in \D$, which is still certainly well defined. In general, one cannot say much more about the kernel, and therefore we are forced to work with this infinite series.  In fact, it is known that for a general radial weight, the kernel may have zeroes, which is not true for the standard weights \cite{PeralaJGEA}.

The classes of weights we consider are defined as follows. For a radial weight $\om$, let $\whw(z)=\int_{|z|}^{1}\om(r)\,dr$ for every $z \in \D$. We say that a radial weight $\om \in \DD$ if there exists $C=C(\om)>1$ such that
\begin{equation} \label{eq:DoublingDef}
    \whw(r) \leq C\whw\left(\frac{1+r}{2}\right), \quad 0\leq r<1.
\end{equation}
We say that 
 $\om \in \Dd$ if there exist $C'=C'(\om)>1$ and $K=K(\om)>1$ such that
\begin{equation}\label{eq:ReverseDoublingDef}
    \whw(r) \geq C'\whw\left(1-\frac{1-r}{K}\right), \quad 0\leq r<1.
\end{equation}
The intersection of these classes is denoted by $\DDD = \DD \cap \Dd$. We emphasize that weights in these classes are not necessarily differentiable, continuous, monotonic or even strictly positive, which is a vast difference to the standard weights. These weights were introduced by Pel\'aez and R\"atty\"a, originating from the study of small Bergman spaces in \cite{PR2014}, see also \cite{PR2015,PR2016}. These classes of weights arise from fundamental questions in the operator theory of weighted Bergman spaces, see \cite{PR2021,PRWW}.

In the setting of doubling weights, $L^p$-estimates for these kernels have been established in \cite{PR2016}. Very recently, the authors in \cite{PRWW} obtained sharp off-diagonal pointwise estimates for the kernels, which are crucial in some arguments we employ. Moreover, the results in \cite{PRWW} say that for the kernel to have a similar maximal growth as the standard kernels, the inducing weight has to belong to $\DD$.

In this paper we are interested in the Ces\`aro-type operator $C_{\mu,\om}$, induced by a Bergman kernel $B^{\om}$ and $\mu \in M^+([0,1))$, given by,
\begin{equation}\label{Eq: Cmuomega}
    C_{\mu,\om}(f)(z)=\int_{0}^{1}f(tz)B^{\om}_t(z)\,d\mu(t)=\sum_{n=0}^{\infty}\mu_n \sum_{k=0}^{n}\frac{a_k}{2\om_{2(n-k)+1}}z^n, \quad z \in \D, \quad f\in \H(\D).
\end{equation}
The generalization of Ces\`aro operator given by \eqref{Eq: Cmu} can be considered as an limit case of $C_{\mu,\om}$, as the classical operator is induced by the Cauchy kernel, which always has, possibly ever so slightly, less growth than any Bergman kernel. We will consider the case when $\om \in \DD$, which allows us sufficient control over the behavior of the kernel. We note that recently 
in \cite{MasMerchanRosa} the authors also considered Ces\`aro operators with kernels induced by doubling weights and a weight instead of a measure. Their generalization was slightly different as their case $\om \equiv 1$ corresponds to the classical Ces\`aro operator, while this is not the case for any weight and measure in \eqref{Eq: Cmuomega}. This small "issue" will be discussed later.

This paper is going to achieve several goals. First of all, we will show that the Carleson-type characterizations shown for example in \cite{GalanoGirelaMerchCesaro,GalaSisZhao} can actually be understood by the interplay between the inducing measure $\mu$ and the weight $\om$. Second of all, we will characterize the boundedness of $C_{\mu, \om}: H^p \to H^q$ and $C_{\mu, \om}:A^p_{\nu} \to A^q_{\nu}$ when $\nu \in \DDD$ for every $0<p,q<\infty$. In addition, the case $q>1$ is considered in the two-weight setting $A^p_{\nu} \to A^q_{\eta}$ for $\nu,\eta \in \DDD$. In the case of Hardy spaces, the situation of $p\neq q$ is essentially new, with only some partial results in the literature. For the weighted Bergman spaces, the case $q<p$ in the case of standard weights in the unit ball was shown recently by authors in the preprint \cite{Pan}. We emphasize that our results were obtained independently, and for a more general class of weights. When $p\leq q$, in \cite{BlascoMas,GalaSisZhao,Pan} many results were achieved in the standard weight setting. We emphasize here that the proof of our result in this case is still new and relies on several technical tools due to the generality of our setting. Moreover, we study the boundedness of $C_{\mu,\om}$ acting on $H^{\infty}$, Korenblum spaces and weighted Hardy spaces obtaining several novel results.

Next we define the fundamental function $F_{\mu,\om}$, whose behavior is closely related to that of $C_{\mu,\om}$. For a radial weight $\om$ and a positive Borel measure $\mu$ on $[0,1)$, define $F_{\mu,\om}$ by
\begin{equation*}
    F_{\mu,\om}(z)=\int_0^1 B^\om_t(z)d\mu(t)= \sum_{n=0}^{\infty}\frac{\mu_n}{2\om_{2n+1}}z^n, \quad z \in \D.
\end{equation*}
Note that $F_{\mu,\om}=C_{\mu,\om}(1)$. Aside from giving characterizations which are easier to compute explicitly, we show that the boundedness of $C_{\mu,\om}$ can be characterized by the fact that (a fractional derivative of) the function $F_{\mu,\om}$ belongs to a certain space. These connections have been studied for example in \cite{BlascoHardy, BlascoDirichlet, BlascoMas,GalanoGirelaMerchCesaro2}.

Therefore, we also need to define the one-weight fractional derivative $D^{\nu}$. For $f \in \mathcal{H}(\D)$ given by $f(z)=\sum_{k=0}^{\infty}a_kz^k$, it is defined by $D^{\nu}(f)(z)=\sum_{k=0}^{\infty}\frac{a_k}{\nu_{2k+1}}z^k$. The one-weight fractional integral $I^{\nu}$ is defined in an analogous manner. The study of fractional derivatives has a long history. The interest to this topic was possibly started by Liouville in 1832 \cite{Liouville}. Since then, it has been studied by many authors, including Hardy and Littlewood \cite{HL1}. More recently, Zhu introduced a different way to view fractional derivatives using Bergman kernels \cite{ZhuFrac}. Per\"al\"a was perhaps the first to consider fractional derivatives induced by weights in \cite{PeralaFrac}. He in fact studied the two-weight fractional derivative operator $R^{\om,\nu}$. Here we define $R^{\om,\nu}(f)(z)=\sum_{k=0}^{\infty}\frac{\om_{2k+1}}{\nu_{2k+1}}a_kz^k$. This notation is also needed in our study. For more information and recent studies on fractional derivatives, see for instance \cite{Bellavita,Moreno,PelaezDeLaRosa,PRW}. For simplicity, when the weight is the standard weight $\om_{\alpha-1}$ for some $\alpha>0$, we denote $D^{\om_{\alpha-1}}=D^{\alpha}$. We also need to have fractional derivatives of powers of weights to state our results. For $\beta>0$, we define $D^{\nu^{\beta}}$ by $D^{\nu^{\beta}}(f)(z)=\sum_{n=0}^{\infty}\frac{a_n}{\nu_{2n+1}^{\beta}}z^n$.

To state our result, we introduce the following notation. As with weights, for $\mu\in M^+([0,1))$, we denote $\widehat{\mu}(r)=\mu([r,1))=\int_{r}^{1}d\mu$ and $\mu_x=\int_{0}^{1}r^x\,d\mu(r)$ for every $x\ge 0$. The Carleson square $S(a)$ induced by a point $a \in \D \setminus \{0\}$ is defined as $S(a)=\{z \in \D: |z|\geq |a|, \, |\arg z-\arg a|\leq \frac{1}{2}(1-|a|)\}$. For $a,b>0$ we say that $a \lesssim b$ or $b \gtrsim a$ if there exists $C>0$ such that $a \leq Cb$. Moreover, if $a \lesssim b \lesssim a$, we write $a \asymp b$ and say that $a$ and $b$ are comparable. In particular, for a radial weight $\om$, we have that {$\om(S(r))\asymp \whw(r)(1-r)$} for $0\le r<1$. 

We also need to introduce some additional spaces of analytic functions. For $\alpha>0$, the Korenblum space $X_{\alpha}$ is defined by functions $f \in \H(\D)$ satisfying
\begin{equation*}
    \|f\|_{X_{\alpha}}=\sup_{0\leq r<1}(1-r)^{\alpha}M_{\infty}(r,f)<\infty.
\end{equation*}
In a similar manner, for $0<\gamma,p<\infty$, the weighted Hardy space $H^p_{\gamma}$ consists of functions $f \in \H(\D)$ satisfying
\begin{equation*}
    \|f\|_{H^p_{\gamma}}=(1-r)^{\gamma}M_p(r,f)<\infty.
\end{equation*}
The Bloch space $\B$ is defined as functions $f \in \H(\D)$ satisfying
\begin{equation*}
    \|f\|_{\B}=|f(0)|+\sup_{z\in\D}|f'(z)|(1-|z|^2)<\infty.
\end{equation*}
More information on Bloch space can be found in \cite{Bloch,Zhu}. The space $BMOA$ contains functions $f \in H^2$ with bounded mean oscillation on the boundary, and it can be equipped with the norm
\begin{equation*}
    \|f\|_{BMOA}=|f(0)|+\sup_{a\in \D}\|f \circ \phi_a-f(a)\|_{H^2}<\infty.
\end{equation*}
Here $\phi_a$ is the self-inverse automorphism of the unit disc. For a reference, see for instance \cite{GirelaBMOA,Zhu}.
Our first result is a complete characterization of measures for which $C_{\mu,\om}: H^p \to H^q$ is bounded for every $0<p,q<\infty$. The proofs rely on sharp pointwise and $L^p$-estimates of the kernels with many technical steps in the way.

\begin{theorem}\label{MainTheorem: Hardy}
Let $0<p,q<\infty$, $\om \in \DD$, and let $\mu \in M^+([0,1))$. Then the following holds:
\begin{enumerate}
    \item[\textup{(i)}] If $p\leq q$, the following conditions are equivalent:
    \begin{enumerate}
        \item[\textup{(a)}] $C_{\mu,\om}:H^p \to H^q$ is bounded;
        \item[\textup{(b)}]
        \begin{equation*}
            \|\mu^{(p,q)}_{\om}\|_{L^{\infty}}=\sup_{0\leq r<1}\frac{\whm(r)}{\whw(r)(1-r)^{1+\frac{1}{p}-\frac{1}{q}}}<\infty;
        \end{equation*}
        \item[\textup{(c)}]
            \begin{equation*}
                 \|\mu^{(p,q)}_{\om}\|_{\ell^{\infty}}=\sup_{n\in \N_0}\frac{\mu_n(n+1)^{1+\frac{1}{p}-\frac{1}{q}}}{\om_n}<\infty;
            \end{equation*}
        \item[\textup{(d)}] For every fixed $\beta>0$, we have $D^{\beta+\frac{1}{p}-\frac{1}{q}}F_{\mu,\om} \in X_{\beta}$.
    \end{enumerate}
    \begin{equation*}
    \end{equation*}
    Moreover, we have $$\|C_{\mu,\om}\|_{H^p \to H^q} \asymp \|\mu^{(p,q)}_{\om}\|_{L^{\infty}} \asymp \|\mu^{(p,q)}_{\om}\|_{\ell^{\infty}}\asymp \|D^{\beta+\frac{1}{p}-\frac{1}{q}}F_{\mu,\om}\|_{X_{\beta}}.$$
    \item[\textup{(ii)}] If $q<p$, the following conditions are equivalent:
    \begin{enumerate}
        \item[\textup{(a)}] $C_{\mu,\om}: H^p \to H^q$ is bounded;
        \item[\textup{(b)}] \begin{equation*}
        \|\mu_{\om}\|_{L^{\frac{qp}{p-q}}}^\frac{qp}{p-q}=\int_{0}^{1}\left(\frac{\whm(r)}{\whw(r)(1-r)}\right)^{\frac{qp}{p-q}}\,dr +\whm(0)^{\frac{qp}{p-q}}<\infty;
        \end{equation*}
        \item[\textup{(c)}]
        \begin{equation*}
         \|\mu_\om\|_{\ell^{\frac{qp}{p-q}}}^\frac{qp}{p-q}=\sum_{n=0}^{\infty}\left(\frac{\mu_n (n+1)}{\om_{n}}\right)^{\frac{qp}{p-q}}\frac{1}{(n+1)^2}<\infty;
        \end{equation*}
        \item[\textup{(d)}] $F_{\mu,\om} \in H^{\frac{qp}{p-q}}.$
    \end{enumerate}
Moreover, we have $$\|C_{\mu,\om}\|_{H^p \to H^q} \asymp \|\mu_{\om}\|_{\ell^{\frac{qp}{p-q}}} \asymp \|\mu_{\om}\|_{L^{\frac{qp}{p-q}}} \asymp \|F_{\mu,\om}\|_{H^{\frac{qp}{p-q}}}.$$
\end{enumerate}
\end{theorem}
In the case of Bergman spaces, we obtain the boundedness of $C_{\mu,\om}: A^p_{\nu} \to A^q_{\nu}$ when the weight inducing the space satisfies $\nu\in\DDD$. Results of this sort were achieved in  \cite{GalaSisZhao},\cite{Pan} and as a corollary in \cite[Corollary 6.13]{BlascoMas}. Our results are proved with a different method, based on arguments with decomposition norms, fractional derivatives, Carleson measures, and pointwise and $L^p$-estimates for the Bergman kernels.
\begin{theorem} \label{MainTheorem: Bergman}
     Let $0<p,q<\infty$, $\om \in \DD$, $\nu \in \DDD$, and let $\mu \in M^+([0,1))$. Then the following holds:
     \begin{enumerate}
        \item[\textup{(i)}] If $p\leq q$, then the following conditions are equivalent:
        \begin{enumerate}
            \item[\textup{(a)}] $C_{\mu,\om}: A^p_{\nu} \to A^q_{\nu}$ is bounded;
            \item[\textup{(b)}] \begin{equation*}
                 \|\mu^{(p,q)}_{\om,\nu}\|_{L^\infty}=\sup_{0\leq r<1}\frac{\widehat{\mu}(r)}{\whw(r)\whv(r)^{\frac{1}{p}-\frac{1}{q}}(1-r)^{1+\frac{1}{p}-\frac{1}{q}}}<\infty;
            \end{equation*}
            \item[\textup{(c)}]
            \begin{equation*}
                 \|\mu^{(p,q)}_{\om,\nu}\|_{\ell^\infty}=\sup_{n \in \N_0}\frac{\mu_n(n+1)^{1+\frac{1}{p}-\frac{1}{q}}}{\om_n\nu_n^{\frac{1}{p}-\frac{1}{q}}};
            \end{equation*}
            \item[\textup{(d)}] For every fixed $\beta>0$, we have $D^{\beta+\frac{1}{p}-\frac{1}{q}}D^{\nu^{\frac{1}{p}-\frac{1}{q}}}F_{\mu,\om} \in X_{\beta}$.
        \end{enumerate}
    Moreover, we have $$\|C_{\mu,\om}\|_{A^p_{\nu} \to A^q_{\nu}}\asymp \|\mu^{(p,q)}_{\om,\nu}\|_{L^\infty} \asymp \|\mu^{(p,q)}_{\om,\nu}\|_{\ell^\infty} \asymp \|D^{\beta+\frac{1}{p}-\frac{1}{q}}D^{\nu^{\frac{1}{p}-\frac{1}{q}}}F_{\mu,\om}\|_{X_{\beta}}.$$
    \item[\textup{(ii)}] If $q<p$, then the following conditions are equivalent:
    \begin{enumerate}
         \item[\textup{(a)}] $C_{\mu,\om}: A^p_{\nu} \to A^q_{\nu}$ is bounded;
        \item[\textup{(b)}]
        \begin{equation*}
        {\|\mu_{\om}\|^\frac{qp}{p-q}_{L^\frac{qp}{p-q}_{\widehat\nu}}}=\int_0^1 \left(\frac{\whm(r)}{\whw(r) (1-r)}\right)^{\frac{qp}{p-q}}\whv(r)\,dr+\whm(0)^\frac{qp}{p-q}<\infty.
        \end{equation*}
        \item[\textup{(c)}]
        \begin{equation*}
            {\|\mu_{\om}\|^\frac{qp}{p-q}_{\ell_{\frac{\nu_n}{(n+1)^2}}^\frac{qp}{p-q}}}=\sum_{n=0}^{\infty}\left(\frac{\mu_{n} (n+1)}{\om_{n}}\right)^{\frac{qp}{p-q}}\frac{\nu_{n}}{(n+1)^2}<\infty;
        \end{equation*}
        \item[\textup{(d)}] $F_{\mu,\om} \in A^{\frac{qp}{p-q}}_\nu.$
    \end{enumerate}
    \begin{equation*}
    \end{equation*}
    Moreover, we have $\|C_{\mu,\om}\|_{A^p_{\nu} \to A^q_{\nu}} \asymp {\|\mu_{\om}\|_{L^\frac{qp}{p-q}_{\widehat\nu}}} \asymp \   {\|\mu_{\om}\|_{\ell_{\frac{\nu_n}{(n+1)^2}}^\frac{qp}{p-q}}}\asymp \|F_{\mu,\om}\|_{A^{\frac{qp}{p-q}}_{\nu}}$.
    \end{enumerate}
\end{theorem}

In fact, our results go much further than the setting shown above. For $q>1$ and $\eta \in \DDD$, we characterize the boundedness of $C_{\mu,\om}: A^p_{\nu} \to A^q_{\eta}$. Even for $1\leq p\leq q$, and if we restrict to standard weights, our results contain the case $A^p_{\alpha} \to A^q_{\beta}$ for any $\alpha,\beta>-1$, unlike in the previously mentioned paper \cite{BlascoMas} which restrict the parameters of the weight.

We note that in both of these cases, the Carleson-type condition is not the right one when $q<p$, and one instead runs into an $L^p$-type condition, which is certainly natural. When considering the Ces\`aro-operator induced by the Cauchy kernel, at least in the case $p\leq q$ one seems to obtain a classical Carleson measure condition for $\mu$, but with general weights this is no longer precisely the case, and we have to consider the interplay between the tail integrals or moments of the weights and the measure.

We also study the action of the operator $C_{\mu,\om}$ on Korenblum spaces $X_\beta$ and weighted Hardy spaces $H^p_\gamma$ for $\beta, \gamma>0$. The main results in this direction can be stated as follows:
\begin{theorem}\label{MainTheorem: Korenbloom}  Let $1<p<\infty$, $0<\gamma<\infty$, $0\leq \beta<\infty$, $\om\in \DD$, and let $\mu\in M^+([0,1))$. 
Then the following statements hold:
\begin{enumerate}
\item [\textup{(i)}]$C_{\mu, \om}: X_\beta \to H^p$ is bounded if and only if 
\begin{equation*}
    \int_0^1 \left(\frac{\whm(r) }{\whw(r)(1-r)^{(\beta+1)}}\right)^p\,dr <\infty.
\end{equation*}
\item [\textup{(ii)}] $C_{\mu, \om}: X_\beta \to X_\gamma$ is bounded if and only if 
\begin{equation*}
\sup_{n\in\N_0}\frac{\mu_n}{\om_n (n+1)^{\gamma-\beta-1}}<\infty.
\end{equation*}
  \item[\textup{(iii)}]  $C_{\mu, \om}: X_\beta \to H^p_\gamma$ is bounded if and only if 
\begin{equation*}
\sup_{n\in\N_0}\frac{\mu_n}{\om_n (n+1)^{\gamma-\beta+\frac{1}{p}-1}}<\infty.
\end{equation*}
\end{enumerate}
Moreover, the operator norm in each situation is comparable to the corresponding characterizing condition.
\end{theorem}

Our final main result is the full characterization of when $C_{\mu,\om}$ maps $H^{\infty}$ to $BMOA$ and $\B$.
\begin{theorem}\label{Thm: HinftytoBMOABloch}
Let $\om \in \DD$ and let $\mu \in M^+([0,1))$. Then the following statements are equivalent:
\begin{enumerate}
\item[\textup{(i)}] $C_{\mu,\om}: H^{\infty} \to BMOA$ is bounded;
\item[\textup{(ii)}] $C_{\mu,\om}: H^{\infty} \to \B$ is bounded;
\item[\textup{(iii)}]
\begin{equation*}
    \|\mu_{\om}\|_{L^\infty}=\sup_{0\leq r<1}\frac{\widehat{\mu}(r)}{\whw(r)(1-r)}<\infty;
\end{equation*}
\item[\textup{(iv)}]
\begin{equation*}
    \|\mu_{\om}\|_{\ell^{\infty}}=\sup_{n\in \N_0}\frac{\mu_n(n+1)}{\om_n}<\infty.
\end{equation*}
\end{enumerate}
Moreover, we have $$\|C_{\mu,\om}\|_{H^{\infty} \to \B} \asymp \|C_{\mu,\om}\|_{H^{\infty} \to BMOA} \asymp \|\mu_{\om}\|_{L^\infty} \asymp \|\mu_{\om}\|_{\ell^{\infty}}.$$
\end{theorem}

The corresponding characterization matches with the one that is equivalent to the boundedness of $C_{\mu,\om}: H^p \to H^p$ and $C_{\mu,\om}: A^p_{\nu} \to A^p_{\nu}$, which is to be expected due to results shown in \cite{BlascoHardy,GalanoGirelaMerchCesaro}.

For $0<\beta<\infty$, denote $C_{\mu,\beta}(f)$ as the weighted Ces\`aro operator given by
\begin{equation*}
    C_{\mu,\beta}(f)(z)=\int_{0}^{1}\frac{f(tz)}{(1-tz)^{\beta}}\,d\mu(t), \quad z \in \D, \quad f \in \H(\D).
\end{equation*}
We note that any weighted Ces\`aro operator with the kernel of the form $(1-tz)^{-(1+\beta)}$ for any $\beta>0$ can be written in the form of \eqref{Eq: Cmuomega}. Even though the situation of when the kernel is replaced by $(1-tz)^{-\beta}$ for $0<\beta\leq 1$ is not covered by our study, some of the proofs of our results certainly extend to this situation. In this setting, one has to "replace" the tail integral by $\whw(r)=(1-r)^{\beta-1}$ and the moment as $\om_n=(n+1)^{1-\beta}$, respectively. This is obviously not true for any weight. Theorem~\ref{MainTheorem: Hardy} and Theorem~\ref{MainTheorem: Bergman} extend to this setting, which vastly generalizes the results found in the literature, see for example \cite{BlascoMas,GalaSisZhao, Pan}. The corresponding statements are given at the end of Section~\ref{Sec: Hardy} and Section~\ref{Sec:Bergman}, respectively. However, as this work is mainly concerned on $C_{\mu,\om}$, we leave the details to the interested reader.

Rest of this paper will be organized as follows. Section~\ref{Sec:Prelim} is dedicated to covering the necessary prerequisites and some lemmas on the classes of weights, which allow us to proceed with the main results. In Section~\ref{Sec: Hardy} we prove Theorem~\ref{MainTheorem: Hardy} by considering several different cases depending on the values of $p$ and $q$. Section~\ref{Sec:Bergman} begins with some additional auxiliary results and definitions needed and concludes with the proof of Theorem~\ref{MainTheorem: Bergman}, while also including the proofs for the two-weight setting $A^p_{\nu} \to A^q_{\eta}$ when $q>1$. In addition, we prove the boundedness of $C_{\mu,\om}:A^p_{\nu} \to A^q_{\nu}$ when $0<p\leq q\leq 1$ with the use of a perhaps new analog of classical embedding theorem of Hardy and Littlewood in the setting of weighted Bergman spaces, which is certainly of independent interest. Finally, Section~\ref{Sec: Other} is dedicated to studying $C_{\mu,\om}$ on some other spaces of analytic functions, proving the different cases in Theorem ~\ref{MainTheorem: Korenbloom} and Theorem~\ref{Thm: HinftytoBMOABloch}.

\section{Preliminaries}\label{Sec:Prelim}

We begin with the following fundamental characterization for weights in the class $\DD$, which can be found for example in \cite[Lemma 2.1]{SS2014}.

\begin{letterlemma}\label{DoublingLemma}
Let $\om$ be a radial weight. Then the following are equivalent:
   \begin{enumerate}
       \item[\textup{(i)}] $\om \in \DD$;
        \item[\textup{(ii)}] There exist constants $C=C(\om)>0$ and $\beta=\beta(\om)>0$ such that
        \begin{equation*}
            \frac{\whw(r)}{(1-r)^{\beta}} \leq \frac{\whw(t)}{(1-t)^{\beta}}, \quad 0\leq r \leq t<1;
        \end{equation*}
        \item[\textup{(iii)}] There exists $\lambda=\lambda(\om)>0$ such that
        \begin{equation*}
            \int_{\D}\frac{\om(z)}{|1-\overline{a}z|^{\lambda+1}}dA(z) \asymp \frac{\whw(a)}{(1-|a|)^{\lambda}},\quad a \in \D;
        \end{equation*}
        \item[\textup{(iv)}] The asymptotic equality $\om_x \asymp \whw\left(1-\frac{1}{x}\right)$ holds for for every $x \ge 1$.
        \item[\textup{(v)}] The moments satisfy $\om_n \asymp \om_{2n}$ for every $n \in \N$.
   \end{enumerate}
\end{letterlemma}
We also need the characterization for the reverse doubling class $\Dd$ found in \cite[Lemma B]{PelaezDeLaRosa}.
\begin{letterlemma}\label{ReverseDoubling}
 Let $\om$ be a radial weight. Then the following are equivalent:
   \begin{enumerate}
       \item[\textup{(i)}] $\om \in \Dd$;
        \item[\textup{(ii)}] There exist constants $C=C(\om)>0$ and $\alpha=\alpha(\om)>0$ such that
        \begin{equation*}
            \frac{\whw(t)}{(1-t)^{\alpha}} \leq \frac{\whw(r)}{(1-r)^{\alpha}}, \quad 0\leq r \leq t<1.
        \end{equation*}
        \item[\textup{(iii)}] There exists $K=K(\om)>1$  and $C=C(\om)>0$ such that
        \begin{equation*}
            \int_{r}^{1- \frac{1-r}{K}}\om(t)\,dt\ge C\whw(r) , \quad 0\le r<1.
        \end{equation*}
       
   \end{enumerate}
\end{letterlemma} 
The following lemma on functions defined by the moments of these weights is fundamental and needed in several places.
\begin{lemma}\label{Twoweightsumestimate}
Let $a,b \geq 0$, $c>0$, and let $\om,\nu \in \DD$. Then
\begin{equation*}
    \sum_{n=0}^{\infty}\frac{(n+1)^{c-1}}{\om_{2n+1}^a\nu_{2n+1}^b}s^n \asymp \frac{1}{\whw(s)^a\whv(s)^b(1-s)^c}, \quad 0\leq s<1.
\end{equation*}
\end{lemma}
\begin{proof}
    Split the sum at $\lfloor (1-s)^{-1}\rfloor$ and apply Lemma~\ref{DoublingLemma}. The proof is elementary and the details are omitted, see also \cite[Lemma 4]{PR2025} for a similar result.
\end{proof}
Our next result is a simple and a very general technique which yields a necessary condition for the boundedness of $C_{\mu,\om}: X \to Y$ for some very general spaces of analytic functions $X$ and $Y$. Moreover, we see in the later sections that this necessary condition is in fact also sufficient in many situations. This certainly does not come as a surprise due to results in \cite{GalanoGirelaMerchCesaro}. However, in this paper we show that there are in fact several types of spaces where this sort of characterization is not the correct one.
\begin{proposition}
 Let $f_a^{\gamma}(z)=\left(\frac{1-a}{1-a z}\right)^{\gamma}$ for $a \in [0,1)$ and $\gamma>0$. Let {$\omega \in \DD$, $\mu \in M^+([0,1))$} and let $X,Y$ be spaces of analytic functions such that $f_a^{\gamma} \in X$ for some $\gamma>0$, and the point evaluations $\delta_a: Y \to \C$ are bounded in $Y$ for every $a \in [0,1)$. If $C_{\mu,\om}:X \to Y$ is bounded, we have
\begin{equation*}
    \|C_{\mu,\om}\|_{X \to Y} \gtrsim \sup_{0 \leq a<1}\frac{\whm(a)}{{\whw(a)}(1-|a|)\|f^{\gamma}_a\|_X\|\delta_a\|_{Y \to \C}}
\end{equation*}
\end{proposition}

\begin{proof}
By the hypothesis, we obtain
\begin{equation*}
\begin{split}
\|f_a^{\gamma}\|_{X}\|C_{\mu,\om}\|_{X \to Y} &\geq \|C_{\mu,\om}(f_a^{\gamma})\|_Y\\
&\geq \frac{1}{\|\delta_a\|_{Y \to \C}}\left|\int_{0}^{1}f_a^{\gamma}(ta)B^{\om}_t(a)\,d\mu(t)\right|\\
&\geq\frac{1}{\|\delta_a\|_{Y \to \C}}\int_{a}^{1}f_a^{\gamma}(ta)B^{\om}_t(a)\,d\mu(t)\\
&\asymp \frac{1}{\|\delta_a\|_{Y \to \C}}\int_{a}^{1}\frac{f_a^{\gamma}(ta)}{\whw(ta)(1-ta)}\,d\mu(t)\\
&\asymp \frac{(1-a)^\gamma}{\|\delta_a\|_{Y \to \C}}\int_{a}^{1}\frac{d\mu(t)}{\whw(ta)(1-ta)^{\gamma+1}}\\
&{\asymp} \frac{\whm(a)}{{\whw(a)}(1-|a|)\|\delta_a\|_{Y \to \C} }, \quad 0\leq a<1,
\end{split}
\end{equation*}
where the last asymptotic follows from Lemma~\ref{DoublingLemma}(ii), as now $\whw(ta) \asymp \whw(a)$ when $a\leq t<1$, which yields the claim.
\end{proof}
Note that it is easy to verify that for Hardy spaces $H^p$, we have $\|\delta_a\|_{H^p \to \C} \asymp \frac{1}{(1-a)^{\frac{1}{p}}}$, and for the weighted Bergman spaces induced by $\om \in \DD$, we have $\|\delta_a\|_{A^p_{\om} \to \C} \asymp \frac{1}{(\whw(a)(1-a))^{\frac{1}{p}}}$. This can be seen by testing with the functions $f_a^{\gamma}$ and the known pointwise estimates for $H^p$ and $A^p_{\omega}$, see, for instance, \cite[Lemma 5.1.1]{Pavlovic2004} and an argument similar to \cite[(3.23),(3.24)]{PR2014}. Here $\gamma$ needs to be fixed large enough depending on $p$ and $\omega$. By applying to Lemma~\ref{DoublingLemma}(iii), we note that $\|f_a^{\gamma}\|_{A^p_{\om}} \asymp (\whw(a)(1-a))^{\frac{1}{p}}$ whenever $\om\in \DD$. This immediately yields the following important corollary.

\begin{corollary}\label{Corollary: TestConditionpleqq}
Let $0<p\leq q<\infty$, $\om,\nu,\eta \in \DD$, and let $\mu \in M^+([0,1))$. Then the following estimates hold:
\begin{enumerate}
    \item[\textup{(i)}]
    \begin{equation*}
        \|C_{\mu,\om}\|_{H^p \to H^q} \gtrsim \sup_{0\leq r<1}\frac{\whm(r)}{\whw(r)(1-r)^{1+\frac{1}{p}-\frac{1}{q}}}
    \end{equation*}
    \item[\textup{(ii)}]
    \begin{equation*}
        \|C_{\mu,\om}\|_{A^p_{\nu} \to A^q_{\eta}} \gtrsim \sup_{0\leq r<1}\frac{\widehat{\mu}(r)\eta(S(r))^{\frac{1}{q}}}{\om(S(r))\nu(S(r))^{\frac{1}{p}}}.
    \end{equation*}
\end{enumerate}
\end{corollary}
Due to Corollary~\ref{Corollary: TestConditionpleqq}, the following simple lemma provides us a valuable technical tool. This can be considered as an extension to a well-known result on classical Carleson measures, see for instance \cite[Lemma 2]{GalanoGirelaMerchCesaro}.
Let us reformulate the condition
$$\sup_{0\leq r<1}\frac{\widehat{\mu}(r)\eta(S(r))^{\frac{1}{q}}}{\om(S(r))\nu(S(r))^{\frac{1}{p}}}\asymp\sup_{0\leq r<1}\frac{\widehat{\mu}(r)\widehat\eta(r))^{\frac{1}{q}}}{\widehat\om(r)\widehat\nu(r)^{\frac{1}{p}}(1-r)^{1+\frac{1}{p}-\frac{1}{q}}}$$ in terms of moments of the weights involved in the formula.

\begin{lemma}\label{Momentvstail}
Let $0\leq a,b,c<\infty$, $-1<d<\infty$, and $\om,\nu,\eta \in \DD$, and let $\mu \in M^+([0,1))$. Then the following statements are equivalent:
\begin{enumerate}
    \item[\textup{(i)}] $\|\mu\|_{l^{\infty}}=\sup_{n\in\N_0}\frac{\mu_n \eta_n^c (n+1)^d}{\om_n^a\nu_n^b}<\infty;$
    \item[\textup{(ii)}] $\|\mu\|_{L^{\infty}}=\sup_{0\leq r<1}\frac{\whm(r)\widehat{\eta}(r)^c}{\whw(r)^a\whv(r)^b(1-r)^d}<\infty.$
\end{enumerate}
Moreover, $\|\mu\|_{l^{\infty}} \asymp \|\mu\|_{L^{\infty}}$, where the constants of comparison depend on all of the fixed parameters and the weights.
\end{lemma}
\begin{proof}
Assume first that $\|\mu\|_{l^{\infty}}<\infty$. For any $0\leq r<1$ there exists $n \in \N_0$ such that $1-\frac{1}{n+1}\leq r< 1-\frac{1}{n+2}$. Thus Lemma~\ref{DoublingLemma} yields
\begin{equation*}
\frac{\whm(r)\widehat{\eta}(r)^c}{\whw(r)^a\whv(r)^b(1-r)^d} \lesssim \frac{\whm\left(1-\frac{1}{n+1}\right) \eta_n^c (n+1)^d}{\om_n^a\nu_n^b}\lesssim \frac{\mu_n \eta_n^c (n+1)^d}{\om_n^a\nu_n^b} \leq \|\mu\|_{l^{\infty}}.
\end{equation*}
Assume next that $\|\mu\|_{L^{\infty}}<\infty$. As $\eta \in \DD$, by Lemma~\ref{DoublingLemma} we obtain
\begin{equation*}
\begin{split}
\mu_n \eta_n^c&=\eta_n^c\int_{0}^{1}t^n\,d\mu(t) \\
&\asymp\whe\left(1-\frac{1}{n}\right)^c\left(\int_{0}^{1-\frac{1}{n}}+\int_{1-\frac{1}{n}}^{1}\right)t^n\,d\mu(t)\\
&=I_1(n)+I_2(n), \quad n \in \N.
\end{split}
\end{equation*}
To deal with $I_2(n)$, the assumption together with the hypothesis that $\om,\nu \in \DD$ yields
\begin{equation*}
\begin{split}
I_2(n)
&\leq \whe\left(1-\frac{1}{n}\right)^c\whm\left(1-\frac{1}{n}\right)\\
&\leq \|\mu\|_{L^{\infty}} \frac{\whw\left(1-\frac{1}{n}\right)^a \whv\left(1-\frac{1}{n}\right)^b }{n^d}\\
&\asymp \|\mu\|_{L^{\infty}}\frac{\om_n^a \nu_n^b}{n^d}, \quad n \in \N.
\end{split}
\end{equation*}
For $I_1(n)$, Fubini's theorem and the assumption yields
\begin{equation*}
\begin{split}
I_1(n)&=\whe\left(1-\frac{1}{n}\right)^c \int_{0}^{1-\frac{1}{n}}n\int_{0}^{t}s^{n-1}\,ds\,d\mu(t)\\
&\leq n\int_{0}^{1-\frac{1}{n}}\int_{0}^{t}s^{n-1}\whe(s)^c\,ds \,d\mu(t)\\
&= n\int_{0}^{1-\frac{1}{n}}\int_{s}^{1-\frac{1}{n}}\,d\mu(t) s^{n-1} \whe(s)^c\,ds \\
&\leq n\int_{0}^{1-\frac{1}{n}}s^{n-1}\whm(s) \whe(s)^c\,ds\\
&\lesssim n\int_{0}^{1-\frac{1}{n}}s^{n-1}\whw(s)^a\whv(s)^b(1-s)^d\,ds
\end{split}
\end{equation*}
As $\om,\nu \in \DD$, by Lemma~\ref{DoublingLemma}, there exists constants $x=x(a,\om)>0$ and $y=y(\nu,b)>0$ such that the functions $r \mapsto \frac{\whw(r)^a}{(1-r)^{x}}$ and $r \mapsto \frac{\whv(r)^b}{(1-r)^{y}}$ are almost increasing with respect to $r$. For the remaining integral this yields
\begin{equation*}
\begin{split}
\int_{0}^{1-\frac{1}{n}}s^{n-1}\whw(s)^a\whv(s)^b(1-s)^d\,ds
    &\lesssim \whw\left(1-\frac{1}{n}\right)^a\whv\left(1-\frac{1}{n}\right)^b n^{x+y}\int_{0}^{1-\frac{1}{n}}s^{n-1}(1-s)^{d+x+y}\,ds\\
    &\leq \whw\left(1-\frac{1}{n}\right)^a\whv\left(1-\frac{1}{n}\right)^b n^{x+y}\int_{0}^{1}s^{n-1}(1-s)^{d+x+y}\,ds\\
    &\asymp \whw\left(1-\frac{1}{n}\right)^a\whv\left(1-\frac{1}{n}\right)^b n^{x+y}n^{-(d+x+y+1)}\\
    &= \|\mu\|_{L^{\infty}}\whw\left(1-\frac{1}{n}\right)^a\whv\left(1-\frac{1}{n}\right)^bn^{-(d+1)}, \quad n \in \N.
\end{split}
\end{equation*}
By Lemma~\ref{DoublingLemma}, the desired conclusion follows as the case $n=0$ is trivial.
\end{proof}
The following proposition helps in connecting the characterization of boundedness of $C_{\mu,\om}$ to properties of $F_{\mu,\om}.$
\begin{proposition}\label{Prop: Fmuom sup norm space}
    Let $0<a<\infty$, $\om \in \DD$, and let $\mu \in M^+([0,1))$. Then the following statements are equivalent:
    \begin{enumerate}
        \item[\textup{(i)}] $\|\mu\|_{\ell^{\infty}}=\sup_{n \in \N_0}\frac{\mu_n(n+1)^{a}}{\om_n}<\infty;$
        \item[\textup{(ii)}] $\|\mu\|_{L^{\infty}}=\sup_{0\leq r<1}\frac{\whm(r)}{\whw(r)(1-r)^a}<\infty;$
        \item[\textup{(iii)}] For every fixed $\beta>0$ satisfying $\beta+a>1$, we have $D^{\beta+a-1}F_{\mu,\om} \in X_{\beta}$.
    \end{enumerate}
    Moreover, $\|\mu\|_{\ell^{\infty}} \asymp \|\mu\|_{L^{\infty}} \asymp \|D^{\beta+a-1}F_{\mu,\om}\|_{X_{\beta}}$.
\end{proposition}
\begin{proof}
    The fact that \textup{(i)} and \textup{(ii)} are equivalent follows from Lemma~\ref{Momentvstail}. Thus, assume \textup{(ii)}. For $\alpha>0$, we note that given a standard weight $\om_{\alpha-1}$, we have $(\om_{\alpha-1})_n \asymp \frac{1}{(n+1)^{\alpha}}$. Therefore, we obtain by Lemma~\ref{DoublingLemma} that
    \begin{equation*}
    \begin{split}
    (1-|z|)^{\beta}|D^{\beta+a-1}F_{\mu,\om}(z)|
    &\lesssim (1-|z|)^{\beta}\sum_{n=0}^{\infty}\frac{\mu_n(n+1)^{\beta+a-1}}{\om_n}|z|^n\\
    &\leq (1-|z|)^{\beta}\|\mu\|_{\ell^{\infty}}\sum_{n=0}^{\infty}(n+1)^{\beta-1}|z|^n\\
    &\asymp \|\mu\|_{\ell^{\infty}}, \quad z\in \D.
    \end{split}
    \end{equation*}
    The other direction follows by letting $r_N=1-\frac{1}{N+1}$ for every $N \in \N_0$, which by Lemma~\ref{DoublingLemma} yields
    \begin{equation*}
    \begin{split}
    \|D^{\beta+a-1}F_{\mu,\om}\|_{X_\beta}
    &\gtrsim \frac{1}{(N+1)^{\beta}}\sum_{n=0}^{\infty}\frac{\mu_n(n+1)^{\beta+a-1}}{\om_n}r_N^n\\
    &\gtrsim  \frac{\mu_N}{(N+1)^{\beta}} \sum_{n=0}^{N}\frac{(n+1)^{\beta+a-1}}{\om_n}\\
    &\asymp \frac{\mu_N}{(N+1)^{\beta}}\frac{(N+1)^{\beta+a}}{\om_N}, \quad N \in \N_0.
    \end{split}
    \end{equation*}
    from where the statement directly follows. This concludes the proof.
\end{proof}

For our main results, we need the following technical auxiliary result.
\begin{lemma} \label{mainlemma0}Let $0<p,\beta<\infty$, $\om\in\DD$, and let $\mu \in M^+([0,1))$. Then
$$\int_0^1 \left(\int_0^x \frac{d\mu(t)}{\whw(rt)(1-rt)^\beta}\right)^p dx\lesssim \int_0^1 \left(\frac{\whm(s)}{\whw(rs)(1-rs)^\beta}\right)^pds+\whm(0)^p, \quad 0\leq r\leq 1.$$
\end{lemma}
\begin{proof}
By Lemma~\ref{DoublingLemma}, it is easy to see that
\be \label{int}
\int_0^t \frac{ds}{\whw(rs)(1-rs)^{\beta+1}}+1\asymp \frac{1}{\whw(rt)(1-rt)^\beta}, \quad 0\leq r\leq 1, \quad 0\leq t<1.
\ee
Therefore, for $0<x<1$ and $0\leq r\leq 1$, Fubini's theorem yields
$$\int_0^x \frac{d\mu(t)}{\whw(rt)(1-rt)^\beta}\lesssim \int_0^x \int_0^t \frac{ds}{\whw(rs)(1-rs)^{\beta+1}}d\mu(t)+\whm(0)\leq\int_0^x\frac{\whm(s)}{\whw(rs) (1-rs)^{\beta+1}}ds+\whm(0).$$
Next, we will show that
$$\int_0^1 \left(\int_0^x \frac{\whm(t)\,dt}{\whw(rt)(1-rt)^{\beta+1}}\right)^p \,dx\lesssim \int_0^1 \left(\frac{\whm(s)}{\whw(rs)(1-rs)^\beta}\right)^pds.$$
For $p=1$ the result is immediate from Fubini's theorem.
Assume first that $p>1$ and take $\beta+1= \alpha+\gamma$ with $1/p'<\gamma<1$. H\"older's inequality then yields
\begin{equation*}
\begin{split}
    \int_0^1 \left(\int_0^x \frac{\whm(t)\,dt}{\whw(rt)(1-rt)^{\beta+1}}\right)^p dx
&\le\int_0^1\left( \int_0^x \left(\frac{\whm(t)}{\whw(rt)(1-rt)^{\alpha}}\right)^p\,dt \right)\left(\int_0^x\frac{\,dt}{(1-rt)^{\gamma p'}}\right)^{p-1} dx\\
&\lesssim\int_0^1\left( \int_0^x \left(\frac{\whm(t)}{\whw(rt)(1-rt)^{\alpha}}\right)^p\,dt \right)\frac{dx}{(1-rx)^{\gamma p-p+1}}\\
&=\int_0^1 \left(\frac{\whm(t)}{\whw(rt)(1-rt)^{\alpha}}\right)^p\left(\int_t^1\frac{\,dx}{(1-rx)^{\gamma p-p+1}}\right)\,dt\\
&\lesssim \int_0^1 \left(\frac{\whm(t)}{\whw(rt)(1-rt)^{\beta}}\right)^p \,dt.
\end{split}
\end{equation*}
For $0<p<1$, we let $t_k=1-2^{-k}$ for every $k\in \N$. By subadditivity and Lemma~\ref{DoublingLemma} we obtain
\begin{equation*}
\begin{split}
 \int_0^1 \left(\int_0^x \frac{\whm(t)}{\whw(rt)(1-rt)^{\beta+1}}\,dt\right)^p \,dx
&\lesssim \sum_{n=0}^\infty \left(\int_0^{t_{n+1}} \frac{\whm(t)}{\whw(rt)(1-rt)^{\beta+1}} \,dt\right)^p 2^{-n}\\
&\lesssim \sum_{n=0}^\infty \left(\sum_{k=0}^{n}\int_{t_k}^{t_{k+1}} \frac{\whm(t)}{\whw(rt)(1-rt)^\beta (1-t)} \,dt\right)^p 2^{-n}\\
&\lesssim \sum_{n=0}^\infty \sum_{k=0}^{n} \left(\frac{\whm(t_k)}{\whw(rt_k)(1-rt_k)^\beta}\right)^p 2^{-n}\\
&= \sum_{k=0}^\infty  \left(\frac{\whm(t_k)}{\whw(rt_k)(1-rt_k)^\beta}\right)^p \sum_{n=k}^{\infty}2^{-n}\\
&\asymp \sum_{k=0}^\infty  \left(\frac{\whm(t_k)}{\whw(rt_k)(1-rt_k)^{\beta}}\right)^p2^{-k}\\
&\lesssim \int_0^1 \left( \frac{\whm(t)}{\whw(rt)(1-rt)^\beta}\right)^p \,dt+\whm(0)^p, \quad 0\leq r\leq 1,
\end{split}
\end{equation*}
which concludes the proof.
\end{proof}
The proofs in the paper heavily rely on the recently established off-diagonal estimates for the Bergman kernels induced by weights in the class $\DD$ shown in \cite[Theorem 1]{PRWW}.

The result we need is the following inequality for weights $\om \in \DD$:
\begin{equation}\label{Eq: PointwiseKernelEstimate}
    |B^{\omega}_a(z)| \lesssim \frac{1}{\om_{\frac{2}{|1-\overline{a}z|}}|1-\overline{a}z|}, \quad a,z \in \D.
\end{equation}
We note that the pointwise estimate \eqref{Eq: PointwiseKernelEstimate} in fact characterizes the class $\DD$. These pointwise estimates yield sharp $L^p$-estimates for these kernels by \cite[Corollary 2]{PRWW}, see also \cite[Theorem 1]{PR2016} for the following.
\begin{letterlemma}\label{Lemma: KernelLpEstimate}
Let $0<p<\infty$ and $\om,\nu \in \DD$. Then we have
\begin{equation*}
            M_p^p(r,B^{\om}_a) \asymp \int_{0}^{|a|r}\frac{dt}{\whw(t)^p(1-t)^p}+1, \quad a \in \D, \quad 0\leq r<1,
\end{equation*}
and
\begin{equation*}
           \|B^{\om}_a\|_{A^p_{\nu}}^p \asymp \int_{0}^{|a|}\frac{\whv(t)}{\whw(t)^p(1-t)^p}\,dt+1, \quad a \in \D.
\end{equation*}
\end{letterlemma}
Due to technicalities, we also need to use the maximal version of the fundamental function $F_{\mu,\om}$, which we denote as $F^{+}_{\mu,\om}$ and define by
\begin{equation*}
     F_{\mu,\om}^+(z)=\int_0^1 |B^\om_t(z)|d\mu(t), \quad z \in \D.
\end{equation*}
\begin{proposition} \label{FHpr} Let $0<p<\infty$, $\om\in\DD$, and let $\mu\in M^+([0,1))$. Then
\begin{equation*}
M_p^p(F_{\mu,\om},r)\leq  M_p^p(F_{\mu,\om}^+,r) \lesssim  \int_0^r \left(\frac{\whm(t)}{\whw(t) (1-t)}\right)^pdt+\whm(0)^p, \quad 0 \leq r \leq 1.
\end{equation*}
\end{proposition}
\begin{proof}
Assume that $0<r<1$ first. Note that for each $0<r, |\theta|<1$ we have
$$|1-rt e^{\pi i \theta}|^2 = (1-rt)^2+4rt\sin^2 \frac{\pi}{2} \theta\asymp (1-rt)^2+\theta^2\asymp \max\{(1-rt)^2, \theta^2\}.$$
We next use the pointwise estimate \eqref{Eq: PointwiseKernelEstimate} to obtain
\begin{equation*}
\begin{split}
 |B^\om(rt e^{\pi i \theta})|&\lesssim \frac{1}{\om_{\frac{1}{|1-rte^{\pi i\theta}|}}|1-rte^{\pi i\theta}|}
\lesssim \frac{1}{\om_{\frac{1}{\max\{(1-rt), |\theta|\}}}\max\{(1-rt), |\theta|\}}.
\end{split}
\end{equation*}
Therefore we have
\begin{equation*}
\begin{split}
M^p_p(F_{\mu,\om},r)
&\lesssim \int_{-1}^1 \left(\int_0^1 |B^\om(rt e^{\pi i \theta})|d\mu(t)\right)^p\,d\theta\\
&\lesssim \int_0^1 \left(\int_0^1\frac{d\mu(t)}{\om_{\frac{1}{\max\{(1-rt), \theta\}}}\max\{(1-rt), \theta\}}\right)^p d\theta\\
&\lesssim \int_0^{1-r} \left(\int_0^1\frac{d\mu(t)}{\om_{\frac{1}{1-rt}}(1-tr)}\right)^p d\theta
+\int_{1-r}^1 \left(\int_0^{\frac{1-\theta}{r}}\frac{d\mu(t)}{\om_{\frac{1}{ 1-rt}}(1-rt)}\right)^p d\theta\\
&\quad + \int_{1-r}^1 \left(\int_{\frac{1-\theta}{r}}^1\frac{d\mu(t)}{\om_{\frac{1}{ \theta}}\theta}\right)^p d\theta\\
&\lesssim  (1-r)\left(\int_0^{1} \frac{d\mu(t)}{\whw(rt)(1-rt)}\right)^p
+\int_{1-r}^1 \left(\int_0^{\frac{1-\theta}{r}}\frac{d\mu(t)}{\om_{\frac{1}{ 1-tr}}(1-tr)}\right)^p d\theta\\
&\quad +\int_{1-r}^1 \left(\frac{\whm(\frac{1-\theta}{r})}{\whw(1-\theta)\theta}\right)^p d\theta\\
&\lesssim (1-r)\left(\int_0^{1} \frac{d\mu(t)}{\whw(rt)(1-rt)}\right)^p
+r\int_{0}^1 \left(\int_0^{x}\frac{d\mu(t)}{\whw(tr)(1-tr)}\right)^p dx\\
&+r\int_{0}^1 \left(\frac{\whm(x)}{\whw(rx)(1-rx)}\right)^p dx\\
&=I_1(r)+I_2(r)+I_3(r), \quad 0<r<1.
\end{split}
\end{equation*}
By Lemma~\ref{mainlemma0}, $I_2(r)$ is of the same form as $I_3(r)$. After a change of variables, and the fact that $\whm(r)$ decreases with respect to $r$, we obtain the desired estimate for $I_2(r)$ and $I_3(r)$. 

For $I_1(r)$, we first use (\ref{int}) to have
\begin{equation*}
\begin{split}
       I_1(r)&\lesssim (1-r)\left(\int_0^1 \left(\int_0^t \frac{dx}{\whw(rx)(1-rx)^2}\right)d\mu(t)\right)^p +\whm(0)^{p} \\
&\asymp (1-r)\left(\int_0^1 \frac{\widehat\mu(x)}{\whw(rx)(1-rx)^2}dx\right)^p+\whm(0)^{p}\\
       &\lesssim(1-r) \left(\int_{0}^{r}\frac{\whm(x)}{\whw(x)(1-x)^2}\,dx\right)^{p}+\whm(0)^{p}.
\end{split}
\end{equation*}
For $p=1$,  the result now follows by
$$\int_0^{{r}} \frac{\widehat\mu(x)}{{\whw(x)(1-x)^2}}dx\le
\frac{1}{(1-r)}\int_0^r \frac{\widehat\mu(x)}{\whw(x)(1-x)}dx, {\quad 0\leq r<1.}$$
For $p>1$, we can use H\"older's inequality
\begin{equation*}
\begin{split}
\left(\int_{0}^{r}\frac{\whm(x)}{\whw(x)(1-x)^2}\,dx\right)^{p}&\le\left(\int_{0}^{r}\left(\frac{\whm(x)}{\whw(x)(1-x)}\right)^p\,dx\right)
\left(\int_0^r \frac{dx}{(1-x)^{p'}}\right)^{p-1}\\
&\lesssim \left(\int_{0}^{r}\left(\frac{\whm(x)}{\whw(x)(1-x)}\right)^p\,dx\right)
\frac{1}{(1-r)^{(p'-1)(p-1)}},
\quad 0\leq r<1.
\end{split}
\end{equation*}
For $p<1$, we consider $t_k=1-2^{-k}$  for every $k \in \N_0$ and let $N=N(r) \in \N_0$ be such that $t_N \leq r < t_{N+1}$. By   subadditivity we obtain
\begin{equation*}
\begin{split}
      I_1(r) &{\leq} (1-r)\left(\sum_{k=0}^{N}\int_{t_k}^{t_{k+1}}\frac{\whm(x)}{\whw(x)(1-x)^2}\,dx\right)^{p}+\whm(0)^{p}\\
       &\lesssim \sum_{k=0}^{N}\left(\frac{\whm(t_k)}{\whw(t_k)(1-t_k)}\right)^{p}(1-r)+\whm(0)^{p}\\
       &\asymp \sum_{k=1}^{N}\left(\int_{t_{k-1}}^{t_k}\,dt\right)\left(\frac{\whm(t_k)}{\whw(t_k)(1-t_k)}\right)^{p}\frac{1-r}{1-t_k}+\whm(0)^{p}\\
       &\lesssim  \int_{0}^{t_N}\left(\frac{\whm(t)}{\whw(t)(1-t)}\right)^{p}\frac{1-r}{1-t}\,dt +\whm(0)^{p}\\
       &\leq \int_{0}^{r}\left(\frac{\whm(t)}{\whw(t)(1-t)}\right)^{p}\,dt+\whm(0)^{p}, \quad 0< r < 1,
\end{split}
\end{equation*}
which concludes the proof, as the case $r=0$ is elementary and $r=1$ is a simple modification of the arguments.
\end{proof}

\begin{proposition}
\label{FHp} Let $0<p<\infty$, $\om,\nu\in\DD$, and $\mu\in M^+([0,1))$. Then we have
\begin{equation*}
   \|F_{\mu,\om}\|_{H^p}^p \asymp \|F_{\mu,\om}^+\|_{L^p(\T)}^p \asymp  \int_0^1 \left(\frac{\whm(r)}{\whw(r) (1-r)}\right)^p\,dr+\whm(0)^p \asymp  \sum_{n=0}^\infty \left(\frac{\mu_n(n+1)}{\om_n}\right)^p \frac{1}{(n+1)^2}.
\end{equation*}
and
\begin{equation*}
       \|F_{\mu,\om}\|_{A^p_{\nu}}^p \asymp \|F_{\mu,\om}^+\|_{L^p_{\nu}}^p \asymp  \int_0^1 \left(\frac{\whm(r)}{\whw(r) (1-r)}\right)^p\whv(r)\,dr+\whm(0)^p \asymp  \sum_{n=0}^\infty \left(\frac{\mu_n(n+1)}{\om_n}\right)^p \frac{\nu_n}{(n+1)^2}.
\end{equation*}
\end{proposition}
\begin{proof}  It is trivial that $\|F_{\mu,\om}\|_{H^p}^p\leq \|F_{\mu,\om}^+\|_{L^p(\T)}^p$. The asymptotic inequality
\begin{equation*}
    \|F_{\mu,\om}^+\|_{L^p(\T)}^p \lesssim \int_0^1 \left(\frac{\whm(r)}{\whw(r) (1-r)}\right)^p\,dr+\whm(0)^p
\end{equation*}
follows by Proposition~\ref{FHpr}. For the next asymptotic inequality, we show a slightly more general case. Let $a,b,c>0$, $d\ge 0$, and $\nu \in \DD$. Note that $\mu_n \gtrsim  \whm(1-\frac{1}{n})$ for every $n\in\N$. Then, by Lemma~\ref{DoublingLemma} we have
\begin{equation}\label{Eq:inequalityomnumu}
\begin{split}
    \sum_{n=0}^\infty \frac{(n+1)^{b-2}\mu_n^c\nu_n^d}{\om_n^a}
    &\gtrsim
\whm(0)^c+ \sum_{n=1}^\infty \int_{1-\frac{1}{n}}^{1-\frac{1}{n+1}}\frac{(n+1)^{b}\whm(r)^c\whv\left(1-\frac{1}{n}\right)^d}{\whw\left(1-\frac{1}{n}\right)^a}dr\\
&\asymp \whm(0)^c+ \int_0^1\frac{\whm(r)^c\whv(r)^d}{\whw(r)^a (1-r)^b}dr,\\
\end{split}
\end{equation}
from where our statement follows from the case $a=b=c=p$ and $d=0$. The last asymptotic inequality follows from F\'ejer-Riesz inequality, see \cite[Theorem 3.13]{Duren}, and Lemma~\ref{DoublingLemma}, since
\begin{equation*}\label{Eq: Fmuomtest}
\begin{split}
    \|F_{\mu,\om}\|^p_{H^p}
&\gtrsim  \int_0^1F_{\mu,\om}(s)^p\,ds\asymp  \int_0^1 \left(\sum_{n=0}^\infty \frac{\mu_n}{\om_n}s^n\right)^p\,ds\\
&\gtrsim \sum_{k=1}^\infty \int_{1-\frac{1}{k}}^{1-\frac{1}{k+1}} \left(\sum_{n=\lfloor k/2 \rfloor}^k\frac{\mu_n}{\om_n}s^n\right)^p\,ds+\mu_0^p\\
&\gtrsim\sum_{k=1}^\infty\left(\frac{\mu_k(k+1)}{\om_k}\right)^p\left( \int_{1-\frac{1}{k}}^{1-\frac{1}{k+1}}\,ds\right)+\mu_0^p\\
&\asymp  \sum_{k=0}^\infty \left(\frac{\mu_k(k+1)}{\om_k}\right)^p \frac{1}{(k+1)^2},
\end{split}
\end{equation*}
which concludes the Hardy space case.

 For the Bergman space, the first inequality $\|F_{\mu,\om}\|_{A^p_{\nu}}^p \leq \|F_{\mu,\om}^+\|_{A^p_{\nu}}^p$ is once again trivial. Again by Proposition~\ref{FHpr} and Fubini's theorem, we have
\begin{equation*}
\begin{split}
    \|F_{\mu,\om}^+\|_{A^p_{\nu}}^p
    &\lesssim \int_{0}^{1}M_p^p(r,F_{\mu,\om}^+)\nu(r)\,dr\\
    &\lesssim \int_{0}^{1}\int_{0}^{r}\left(\frac{\whm(t)}{\whw(t) (1-t)}\right)^p\,dt\nu(r)\,dr+\whm(0)^p\\
    &= \int_0^1 \left(\frac{\whm(t)}{\whw(t) (1-t)}\right)^p\whv(t)\,dt+\whm(0)^p,
\end{split}
\end{equation*}
which shows the second asymptotic inequality. Then by \eqref{Eq:inequalityomnumu}, choosing $a=b=c=p$ and $d=1$, we have
\begin{equation*}
    \int_0^1 \left(\frac{\whm(t)}{\whw(t) (1-t)}\right)^p\whv(t)\,dt+\whm(0)^p \lesssim \sum_{n=0}^\infty \left(\frac{\mu_n(n+1)}{\om_n}\right)^p\frac{\nu_n}{(n+1)^2}.
\end{equation*}
For the final asymptotic inequality, we need the inequality $\|f\|_{A^p_{\nu}} \gtrsim \int_{0}^{1}M^p_{\infty}(s,f)\whv(s)\,ds$, valid for every radial weight $\nu$, which follows from \cite[Theorem 2.16]{Pavlovic2014} and Fubini's theorem. Then we take exactly the same steps as was done for Hardy spaces, which yields
\begin{equation*}
    \sum_{n=0}^\infty \left(\frac{\mu_n(n+1)}{\om_n}\right)^p\frac{\nu_n}{(n+1)^2} \lesssim \|F_{\mu,\om}\|_{A^p_{\nu}}^p,
\end{equation*}
which concludes the proof.
\end{proof}

\section{Proof of Theorem \ref{MainTheorem: Hardy}}\label{Sec: Hardy}
\subsection{ Case $p\le q$}
Note that Proposition~\ref{Prop: Fmuom sup norm space} yields the equivalence between \textup{(b)}, \textup{(c)} and \textup{(d)} in part \textup{(i)} of Theorem \ref{MainTheorem: Hardy}. By Corollary~\ref{Corollary: TestConditionpleqq}, \textup{(a)} implies \textup{(b)}. It remains to show that the condition 
 \begin{equation*}
            \|\mu^{(p,q)}_{\om}\|_{L^{\infty}}=\sup_{0\leq r<1}\frac{\whm(r)}{\whw(r)(1-r)^{1+\frac{1}{p}-\frac{1}{q}}}<\infty
        \end{equation*} implies that  $C_{\mu,\om}:H^p \to H^q$ is bounded.

We now split into different cases.

\centerline{ \bf Case $p\le q$, $q>1$}

In this case we follow the reasoning by Andersen \cite{Andersen}, which was developed more by authors in \cite{GalanoGirelaMerchCesaro}.

\begin{letterlemma}\label{AndersenLemma} \cite[Lemma 2.1, Lemma 2.2]{Andersen} Let $1<q<\infty$ and consider $$K(\phi,\t)= \int_{0}^{1}\frac{x}{(x^2+\phi^2)(x^2+\t^2)^{\frac{1}{2}}}\,dx, \quad 0<\phi,\t<2\pi.$$
Then there exists $C=C(p)>0$ such that \begin{equation}
\left(\int_{0}^{2\pi}
\left(\int_{0}^{2\pi}|g(e^{i(\phi+\theta)})|K(\phi,\theta)\,\frac{d\phi}{2\pi}\right)^q\,\frac{d\theta}{2\pi}\right)^{1/q}\le C \left(\int_{0}^{2\pi}|g(e^{i\theta})|^q \frac{d\theta}{2\pi}\right)^{1/q}.
\end{equation}
for any measurable function $g\in L^q(\mathbb T)$.
\end{letterlemma}

\begin{lemma}\label{AndersenLemma2} Let $\beta>0$, $\om \in \DD$ and $\mu\in M^+([0,1))$. Let us write
$$K_{\mu,\om,\beta}( \phi,\t)=\int_{0}^{1}\frac{(1-t)^{1-\beta}}{\widehat\om(t)((1-t)^2+\frac{t}{2}\t^2)^{\frac{1}{2}}((1-t)^2+t\phi^2)}\,d\mu(t), \quad 0<\phi,\theta<2\pi.$$
Then
$$ K_{\mu,\om,\beta}( \phi,\t)\lesssim \int_0^1 \frac{(1-r)^{-\beta} \widehat\mu(r)}{\widehat\om(r) ((1-r)^2+\frac{\t^2}{2})^{\frac{1}{2}} ((1-r)^2+\phi^2)}dr+ \widehat\mu(0), \quad 0<\phi,\theta<2\pi.$$
\end{lemma}
\begin{proof}
Let $t_k=1-2^{-k}$ for every $k \in \N$.  We have
\begin{equation*}
\begin{split}
    K_{\mu,\om,\beta}( \phi,\t)&= \sum_{k=1}^{\infty}\int_{t_k}^{t_{k+1}}\frac{(1-t)^{1-\beta}}{\widehat\om(t)((1-t)^2+\frac{t}{2}\t^2)^{\frac{1}{2}}((1-t)^2+t\phi^2)}\,d\mu(t)\\
    &+\int_{0}^{1/2}\frac{(1-t)^{1-\beta}}{\widehat\om(t)((1-t)^2+\frac{t}{2}\t^2)^{\frac{1}{2}}((1-t)^2+t\phi^2)}\,d\mu(t)\\
    &\lesssim \sum_{k=1}^{\infty}\frac{(t_k-t_{k-1})(1-t_{k+1})^{-\beta}\widehat\mu(t_k)}{\widehat\om(t_k)((1-t_{k+1})^2+\frac{\t^2}{4})^{\frac{1}{2}}((1-t_{k+1})^2+\frac{\phi^2}{2})}\,+ \widehat\mu(0)\\
     &\lesssim \sum_{k=1}^{\infty}\frac{(1-t_{k+1})^{-\beta}}{\widehat\om(t_k)((1-t_{k})^2+\frac{\t^2}{2})^{\frac{1}{2}}((1-t_{k})^2+\phi^2)}
     \left(\int_{t_{k-1}}^{t_k} \widehat\mu(r) dr\right)\,+ \widehat\mu(0)\\
     &\lesssim \int_0^1 \frac{(1-r)^{-\beta} \widehat\mu(r)}{\widehat\om(r) ((1-r)^2+\frac{\t^2}{2})^{\frac{1}{2}} ((1-r)^2+\phi^2)}dr+ \widehat\mu(0),
\end{split}
\end{equation*}
which concludes the proof.
\end{proof}
We are now ready to give the proof of Theorem~\ref{MainTheorem: Hardy}(i), when $q>1$.
 \begin{proof} Let us show that (b) implies (a). 
Assume now that $\|\mu^{(p,q)}_{\om}\|_{L^\infty}<\infty$. 
For every $0\leq r<1$, the pointwise estimates for Hardy spaces, subharmonicity of $|f|^{\frac{p}{q}}$ by \cite[Theorem 2.4.1]{Ransford} and Fubini's theorem  yield
\begin{equation*}
\begin{split}
    \left|C_{\mu,\om}(f)(re^{i\t})\right|
    &\leq \int_{0}^{1}|f(tre^{i\t})||B^{\om}_t(re^{i\t})|\,d\mu(t)\\
    &=\int_{0}^{1}|f(tre^{i\t})|^{1-\frac{p}{q}+\frac{p}{q}}|B^{\om}_t(re^{i\t})|\,d\mu(t)\\
    &\lesssim \|f_r\|_{H^p}^{1-\frac{p}{q}}\int_{0}^{1}|f(tre^{i\t})|^{\frac{p}{q}}\frac{|B^{\om}_t(re^{i\t})|}{(1-t)^{\frac{1}{p}-\frac{1}{q}}}\,d\mu(t)\\
    &\leq\|f_r\|_{H^p}^{1-\frac{p}{q}}\int_{0}^{1}\left(\int_{0}^{2\pi}\frac{1-t^2}{|1-te^{i(\phi-\theta)}|^2}|f_r(e^{i\phi})|^{\frac{p}{q}}\,
   \frac{d\phi}{2\pi}\right) \frac{|B^{\om}_t(re^{i\t})|}{(1-t)^{\frac{1}{p}-\frac{1}{q}}}\,d\mu(t)\\
    &\lesssim\|f_r\|_{H^p}^{1-\frac{p}{q}}\int_{0}^{2\pi}|f_r(e^{i(\phi+\theta)})|^{\frac{p}{q}}\left(\int_{0}^{1}\frac{(1-t)^{1-\frac{1}{p}+
    \frac{1}{q}}|B^{\om}_t(re^{i\t})|}{|1-te^{i\phi}|^2}\,d\mu(t)\right)\,\frac{d\phi}{2\pi}\\
\end{split}
\end{equation*}
Next, we consider the inner integral and the pointwise estimates \eqref{Eq: PointwiseKernelEstimate}, which yield
\begin{equation*}
\begin{split}
\int_{0}^{1}\frac{(1-t)^{1-\frac{1}{p}+
    \frac{1}{q}}|B^{\om}_t(re^{i\t})|}{|1-te^{i\phi}|^2}\,d\mu(t)&\lesssim \int_{0}^{1}\frac{(1-t)^{1-\frac{1}{p}+
    \frac{1}{q}}}{\om_{\frac{2}{|1-tre^{i\t}|}}|1-tre^{i\t}||1-te^{i\phi}|^2}\,d\mu(t)\\
    &\lesssim \int_{0}^{1}\frac{(1-t)^{1-\frac{1}{p}+
    \frac{1}{q}}}{\om_{\frac{1}{1-t}}((1-tr)^2+tr\t^2)^{\frac{1}{2}}((1-t)^2+t\phi^2)}\,d\mu(t)\\
    &\lesssim \int_{0}^{1}\frac{(1-t)^{1-\frac{1}{p}+
    \frac{1}{q}}}{\widehat\om(t)((1-t)^2+\frac{t}{2}\t^2)^{\frac{1}{2}}((1-t)^2+t\phi^2)}\,d\mu(t)\\
    &= K_{\mu,\om,\frac{1}{p}-
    \frac{1}{q}}(\phi,\t), \quad 0<\phi,\theta<2\pi.
    \end{split}
\end{equation*}
Now Lemma \ref{AndersenLemma2} and the assumption give
 $$K_{\mu,\om,\frac{1}{p}-
    \frac{1}{q}}(\phi,\t)\lesssim \|\mu^{(p,q)}_{\om}\|_{L^{\infty}} K(\phi, \t), \quad 0<\phi,\theta<2\pi,$$
    which combined with Lemma \ref{AndersenLemma}  yields
\begin{equation}\label{Eq: Lp-ineqCesaroq>1}
\begin{split}
M_q^q(r,C_{\mu,\om}(f)) &\lesssim \|f_r\|_{H^p}^{q-p}\|\mu^{(p,q)}_{\om}\|_{L^{\infty}}^q\int_{0}^{2\pi}
\left(\int_{0}^{2\pi}|f_r(e^{i(\phi+\theta)})|^{\frac{p}{q}}K(\phi,\theta)\,\frac{d\phi}{2\pi}\right)^q\,\frac{d\theta}{2\pi}\\
&\lesssim \|f_r\|_{H^p}^{q-p}\|\mu^{(p,q)}_{\om}\|_{L^{\infty}}^q\int_{0}^{2\pi}
|f_r(e^{i\theta})|^{p}\,\frac{d\theta}{2\pi}\\
&\lesssim \|f_r\|_{H^p}^{q}\|\mu^{(p,q)}_{\om}\|_{L^{\infty}}^q, \quad 0\leq r<1.
\end{split}
\end{equation}
 This shows that $\|C_{\mu,\om}(f)\|_{H^q}\lesssim \|\mu^{(p,q)}_{\om}\|_{L^{\infty}}\|f\|_{H^p}$.
\end{proof}
\bigskip

\centerline{ \bf Case $0<p\le q\le 1$}
\begin{proof}
Assume $\|\mu^{(p,q)}_{\om}\|_{L^{\infty}}<\infty$ and denote,
as above,  $t_k=1-2^{-k}$ for every $k \in \N$. Define $g(z)=f(z)B^{\om}(z)$. Then, one obtains by subadditivity and Fubini's theorem that
\begin{equation*}
\begin{split}
M_q^q(r,C_{\mu,\om}(f))
&= \int_{0}^{2\pi}\left|\int_{0}^{1}f(tre^{i\t})B^{\om}_t(re^{i\t})\,d\mu(t)\right|^q\,\frac{d\t}{2\pi}\\
&\leq \int_{0}^{2\pi}\left(\int_{0}^{1}\left|g(tre^{i\t})\right|\,d\mu(t)\right)^q\,\frac{d\t}{2\pi}\\
& \leq \sum_{k=0}^{\infty}\int_{0}^{2\pi}\left(\int_{t_k}^{t_{k+1}}\left|g(tre^{i\t})\right|\,d\mu(t)\right)^q\,\frac{d\t}{2\pi}\\
&\leq \sum_{k=0}^{\infty}\whm(t_k)^{q}\int_{0}^{2\pi}\sup_{0\leq t<t_{k+1}}\left|g(tre^{i\t})\right|^q\,\frac{d\t}{2\pi}\\
&\leq \|\mu^{(p,q)}_{\om}\|_{L^{\infty}}^q\sum_{k=0}^{\infty}\om(S(t_k))^{q}(1-t_k)^{\frac{q}{p}-1}\int_{0}^{2\pi}\sup_{0\leq t<t_{k+1}}\left|g(tre^{i\t})\right|^q\,\frac{d\t}{2\pi}, \quad 0\leq r<1.
\end{split}
\end{equation*}
By the radial maximal theorem \cite[Theorem 7.1.4]{Pavlovic2004}, and the fact that $L^p$-means of subharmonic functions are increasing, we have
\begin{equation*}
\begin{split}
M_q^q(r,C_{\mu,\om}(f))
&\lesssim\|\mu^{(p,q)}_{\om}\|_{L^{\infty}}^q\sum_{k=0}^{\infty}\om(S(t_k))^{q}(1-t_k)^{\frac{q}{p}-1}\int_{0}^{2\pi}\left|g(t_{k+1}re^{i\t})\right|^q\,d\t\\
&\asymp \|\mu^{(p,q)}_{\om}\|_{L^{\infty}}^q\sum_{k=0}^{\infty}\int_{t_{k+1}}^{t_{k+2}}\om(S(t_k))^{q}(1-t_k)^{\frac{q}{p}-2}\int_{0}^{2\pi}\left|g(t_{k+1}re^{i\t})\right|^q\,d\t\,ds\\
&\lesssim \|\mu^{(p,q)}_{\om}\|_{L^{\infty}}^q\sum_{k=0}^{\infty}\int_{t_{k+1}}^{t_{k+2}}\om(S(s))^{q}(1-s)^{\frac{q}{p}-2}\int_{0}^{2\pi}\left|g(sre^{i\t})\right|^q\,d\t\,ds\\
&\leq \|\mu^{(p,q)}_{\om}\|_{L^{\infty}}^q\int_{0}^{1}\om(S(s))^{q}(1-s)^{\frac{q}{p}-2}\int_{0}^{2\pi}\left|g(sre^{i\t})\right|^q\,d\t\,ds, \quad 0 \leq r <1.
\end{split}
\end{equation*}
Let $x,x'$ be fixed H\"older conjugates such that $qx'>1$. By using the definition of the function $g$, H\"older's inequality and the $L^p$-estimates for the kernels given in Lemma~\ref{Lemma: KernelLpEstimate}, we obtain
\begin{equation*}
\begin{split}
\int_{0}^{2\pi}\left|g(sre^{i\t})\right|^q\,d\t
&=\int_{0}^{2\pi}\left|f(sre^{i\t})\right|^q\left|B^{\om}_s(re^{i\t})\right|^q\,d\t\\
&\leq \left(\int_{0}^{2\pi}\left|f(sre^{i\t})\right|^{qx}\,d\t \right)^{\frac{1}{x}}\left(\int_{0}^{2\pi}\left|B^{\om}_s(re^{i\t})\right|^{qx'}\,d\t \right)^{\frac{1}{x'}}\\
&\asymp M_{qx}^q(s,f_r)\left(\int_{0}^{sr}\frac{dt}{\whw(t)^{qx'}(1-t)^{qx'}}+1\right)^{\frac{1}{x'}}\\
&\asymp \frac{M_{qx}^q(s,f_r)}{\whw(sr)^{q}(1-sr)^{q-\frac{1}{x'}}}, \quad 0\leq s,r<1.
\end{split}
\end{equation*}
Plugging this back into the previous formulas and applying a theorem of Hardy-Littlewood, see \cite[Theorem 5.11]{Duren}, we obtain
\begin{equation}\label{Eq: Lp-ineqCesaroq<1}
\begin{split}
M_q^q(r,C_{\mu,\om}(f)) &\lesssim \|\mu^{(p,q)}_{\om}\|_{L^{\infty}}^q\int_{0}^{1}\om(S(s))^{q}(1-s)^{\frac{q}{p}-2}\frac{M_{qx}^q(s,f_r)}{\whw(sr)^{q}(1-sr)^{q-\frac{1}{x'}}}\,ds\\
&\lesssim \|\mu^{(p,q)}_{\om}\|_{L^{\infty}}^q\int_{0}^{1}M_{qx}^q(s,f_r)(1-s)^{\frac{q}{p}-2+\frac{1}{x'}}\,ds\\
&=\|\mu^{(p,q)}_{\om}\|_{L^{\infty}}^q\int_{0}^{1}M_{qx}^q(s,f_r)(1-s)^{q\left(\frac{1}{p}-\frac{1}{qx}\right)-1}\,ds\\
&\lesssim \|\mu^{(p,q)}_{\om}\|_{L^{\infty}}^q\|f_r\|_{H^p}^{q}, {\quad 0\leq r<1,}
\end{split}
\end{equation}
which finishes the proof of Theorem~\ref{MainTheorem: Hardy}(i).
\end{proof}
Note that in both cases the proof shows that we actually have an estimate for the $L^p$-means at radius $r$ instead of only the norm inequality.

\subsection{ Case $q<p$.}   
 When $q<p$, the Carleson-type condition shown in Corollary~\ref{Corollary: TestConditionpleqq} is no longer the right condition, which requires us to use different tools. We will need some basic knowledge from the tent spaces $T^q_p(\mu)$, which we define here for the convenience of the reader. For $0<p,q<\infty$, define the tent space $T^{q}_p(\mu)$ as the $\mu$-measurable functions satisfying
\begin{equation*}
    \|f\|_{T^{q}_p(\mu)}^q = \int_{\T}\left(\int_{\Gamma(\z)}|f(z)|^p \,d\mu(z)\right)^{\frac{q}{p}}\,|d\z|+|f(0)|^p\mu(\{0\})^p
\end{equation*}
Here $\Gamma(\zeta)=\{z \in \D: |\arg z - \arg \z|<\frac{1}{2}(1-|z|)\}$ for $\zeta\in \T$. We also define the tent $I(z)$ by $I(z)=\{\zeta \in \T: z \in \Gamma(\zeta)\}$ for $z \in \D$.
When $\mu=\sum_k\delta_{z_k}$, meaning $\mu$ is a point mass measure, we denote $T^{q}_p(\mu)=T^{q}_p(\{z_k\})$. In this case, we can consider these tent spaces as sequence spaces, by identifying a function $f$ with the sequence $\{f(z_k)\}$, which should not confuse the reader. The next lemma follows originally from ideas of Luecking \cite[Lemma 3]{LueckingProc}, see also \cite[Lemma D]{PelaezHp}.
\begin{letterlemma}\label{Lemma: q<pHpTestfunction}
Let $\{z_k\}$ be an $\veps$-separated sequence. Then, for every fixed $\lambda>\max\{1,\frac{2}{p}\}$ there exists a constant $C=C(p, \veps,\lambda)>0$ such that the operator $S_{\lambda}: T^{p}_2(\{z_k\}) \to H^p$, defined by
\begin{equation*}
    S_{\lambda}(f)(z)=\sum_{k=0}^{\infty}f(z_k)\left(\frac{1-|z_k|}{1-\overline{z_k}z}\right)^{\lambda}, \quad z \in \D,
\end{equation*}
satisfies
\begin{equation*}
    \|S_{\lambda}(f)\|_{H^p}^p \leq C\|f\|_{T^{p}_2(\{z_k\})}^p.
\end{equation*}
\end{letterlemma}
We also need the following auxiliary results.
\begin{lemma}\label{Lemma:LowerBoundEq}
Let $\lambda>0$ and $\om\in \widehat{\mathcal D}$. Then
\be
 \frac{\mu_{\lfloor \frac{1}{1-r} \rfloor-1}}{\om_{\frac{1}{1-r}}(1-r)^{\lambda+1}}s^{\lfloor \frac{1}{1-r} \rfloor-1} \lesssim \int_0^1 \frac{ B^\om_t(s)}{(1-rst)^\lambda} d\mu(t), \quad 0\leq r,s<1.
\ee
\end{lemma}
\begin{proof} By Lemma~\ref{DoublingLemma} we obtain
$$\frac{ B^\om_t(s)}{(1-rst)^\lambda}\asymp \sum_{n=0}^\infty \left(\sum_{k=0}^n \frac{(k+1)^{\lambda -1}r^k}{\om_{n-k}}\right)t^ns^n, \quad 0\leq r,t,s<1,$$
which yields
\begin{equation*}
\begin{split}
\int_0^1 \frac{ B^\om_t(s)}{(1-rst)^\lambda} d\mu(t)
&\asymp \sum_{n=0}^\infty\mu_n \left(\sum_{k=0}^n \frac{(k+1)^{\lambda -1}r^k}{\om_{n-k}}\right)s^n\\
&\gtrsim\sum_{n=0}^{\lfloor \frac{1}{1-r} \rfloor-1} \mu_n\left(\sum_{k=0}^n \frac{(k+1)^{\lambda -1}}{\om_{n-k}}\right)s^{\lfloor \frac{1}{1-r} \rfloor-1}\\
&\gtrsim\sum_{n=0}^{\lfloor \frac{1}{1-r} \rfloor-1} \mu_n\frac{(n+1)^{\lambda }}{\om_{n}}s^{\lfloor \frac{1}{1-r} \rfloor-1}\\
&\gtrsim\frac{\mu_{\lfloor \frac{1}{1-r} \rfloor-1}}{\om_{\frac{1}{1-r}}(1-r)^{\lambda+1}}s^{\lfloor \frac{1}{1-r} \rfloor-1}, \quad 0\leq r,s<1,
\end{split}
\end{equation*}
which concludes the proof.
\end{proof}

\begin{lemma}\label{Lemma:T^pequivalence}
Let $0<p<\infty$, $t_k=1-2^{-k}$ for $k\in \mathbb N$ and let $\{b_k\}$ be a sequence of complex numbers. Then
\be  \int_\T \left(\sum_{t_k\in \Gamma(\xi)}|b_k|^2\right)^{\frac{p}{2}} |d\xi|\asymp \sum_{k=1}^\infty |b_k|^p2^{-k}.\ee
\end{lemma}
\begin{proof} Set $I(z)=\{\xi\in \mathbb T: z \in \Gamma(\xi)\}$ as the tent related to the  approach region $\Gamma$. The assertion is clear for $p=2$, since Fubini's theorem yields
$$ \int_\T \sum_{t_k\in \Gamma(\xi)}|b_k|^2 |d\xi|=\sum_{k=1}^\infty |b_k|^2 |I(t_k)| \asymp \sum_{k=1}^\infty |b_k|^2 2^{-k}.$$
For $p>0$, we have the lower bound as
\begin{equation*}
\begin{split}
\int_{\mathbb T}\left(\sum_{t_k\in \Gamma(\xi)} |b_k|^2\right)^{\frac{p}{2}}|d\xi|
&\ge \int_{I(t_1)}\left(\sum_{t_k\in \Gamma(\xi)} |b_k|^2\right)^{\frac{p}{2}}|d\xi|\\
&= \sum_{n=1}^\infty\int_{I(t_n)\setminus I(t_{n+1})}\left(\sum_{t_k\in \Gamma(\xi)} |b_k|^2\right)^{\frac{p}{2}}|d\xi|\\
&\ge \sum_{n=1}^\infty\int_{I(t_n)\setminus I(t_{n+1})} |b_n|^p|d\xi|\\
&\ge \sum_{n=1}^\infty |b_n|^p\left(2^{-n}-2^{-(n+1)}\right)\\
&\asymp \sum_{n=1}^\infty |b_n|^p2^{-n}
\end{split}
\end{equation*}
For $p< 2$, by subadditivity we have
$$\int_\T \left(\sum_{t_k\in \Gamma(\xi)}|b_k|^2\right)^{\frac{p}{2}} |d\xi|\le \int_\T \sum_{t_k\in \Gamma(\xi)}|b_k|^p |d\xi|$$
and therefore the assertion holds by Fubini's theorem.

The case $p>2$ can now be shown using duality since $(T_2^p(\eta))^*= T_2^{p'}(\eta)$ for any measure $\eta$ defined on $\D$ under the pairing $$\langle f,g\rangle =\int_\D f(z)\overline{g(z)} (1-|z|)d\eta(z),$$
see for instance \cite[Lemma 6]{Arsenovic}. In particular, when applied to $d\eta= \sum_k \delta_{t_k}$, we obtain for $p>2$
\begin{equation*}
\begin{split}
\left(\int_{\mathbb T}\left(\sum_{t_k\in \Gamma(\xi)} |b_k|^2\right)^{\frac{p}{2}}|d\xi|\right)^{\frac{1}{p}}
&\asymp \sup\left\{\left|\sum_{k=1}^\infty b_k\overline{a_k}2^{-k}\right|: \|a_k\|_{T^{p'}_2}\le 1\right\}\\
&\asymp\sup\left\{\left|\sum_{k=1}^\infty b_k\overline{a_k}2^{-k}\right|:\sum_{k=1}^\infty |a_k|^{p'}2^{-k}\le 1\right\}\\
&\asymp \left(\sum_{k=1}^\infty |b_k|^p2^{-k}\right)^{\frac{1}{p}},
\end{split}
\end{equation*}
which concludes the proof.
\end{proof}
We can now give the proof of Theorem \ref{MainTheorem: Hardy}(ii).
\begin{proof} Note that \textup{(b)}, \textup{(c)}, and \textup{(d)} in Theorem \ref{MainTheorem: Hardy}(ii) are equivalent due to Proposition~\ref{FHp}.

Let us show that \textup{(b)} implies \textup{(a)}. Let $g^*$ denote the radial maximal function defined by $g^*(z)=\sup_{0\leq t<1}g(tz)$. By a trivial estimate, we have $|C_{\mu,\om}(f)(z)|\leq f^*(z) F^+_{\mu,\om}(z)$. Therefore, H\"older's inequality and Proposition~\ref{FHp} implies
\begin{equation*}
\begin{split}
M_q(r,C_{\mu,\om}(f))
&\leq M_q(r,f^*F^+_{\mu,\om})\\
&\leq M_p(r,f^*)M_{\frac{qp}{p-q}}(r,F^{+}_{\mu,\om})\\
&\lesssim M_p(r,f)\|\mu_{\om}\|_{L^{\frac{qp}{p-q}}}, \quad 0<r<1.
\end{split}
\end{equation*}
In the last asymptotic inequality we have applied the radial maximal theorem \cite[Theorem 7.1.4]{Pavlovic2004}. Thus we obtain the wanted estimate $\|C_{\mu,\om}\|_{H^p \to H^q} \lesssim\|\mu_{\om}\|_{L^{\frac{qp}{p-q}}}$.

To complete the proof let us show that  \textup{(a)} implies \textup {(c)}. Let $N \in \N$ be arbitrary and consider sequences $\{b_k\}$ such that $b_{k}=0$ for every $k>N$. Let $t_k=1-2^{-k}$ for every $k \in \N$. Let $f_{x,N}$ be given by $f_{x,N}(z)=\sum_{k=0}^{N}b_kr_k(x)\left(\frac{1-|t_k|}{1-\overline{t_k}z}\right)^{\lambda}$ for every $x \in [0,1]$ and $z \in \D$, where $\lambda$ is that of Lemma~\ref{Lemma: q<pHpTestfunction}. Here $r_k$ is the $k$th Rademacher function. Then, due to Lemma \ref{Lemma: q<pHpTestfunction} and the use of the Fejer-Riesz inequality (see \cite[(5.15)]{Pavlovic2004}), we obtain
\begin{equation*}
\begin{split}
\|(b_k)\|_{T^{p}_2(\{t_k\})}^q \|C_{\mu,\om}\|_{H^p\to H^q}^q
&\gtrsim \|f_{x,N}\|_{H^{p}}^q \|C_{\mu,\om}\|_{H^p\to H^q}^q\\
&\gtrsim \int_{0}^{2\pi}\left|\int_{0}^{1}\sum_{k=0}^{N}b_kr_k(x)\left(\frac{1-|t_k|}{1-\overline{t_k}te^{i\t}}\right)^{\lambda}B^{\om}_t(e^{i\t})\,d\mu(t)\right|^q\,d\t\\
&=\int_{0}^{2\pi}\left|\sum_{k=0}^{N}b_kr_k(x)(1-|t_k|)^{\lambda}\int_{0}^{1}\frac{B^{\om}_t(e^{i\t})}{\left(1-\overline{t_k}te^{i\t}\right)^{\lambda}}\,d\mu(t)\right|^q\,d\t\\
&\gtrsim \int_{0}^{1}\left|\sum_{k=0}^{N}b_kr_k(x)(1-|t_k|)^{\lambda}\int_{0}^{1}\frac{B^{\om}_t(s)}{\left(1-\overline{t_k}ts\right)^{\lambda}}\,d\mu(t)\right|^q\,ds
\end{split}
\end{equation*}
Integrating both sides from $0$ to $1$ with respect to $x$, applying Fubini's theorem and Khinchine's inequality, one obtains
\begin{equation*}
\begin{split}
\|b_k\|_{T^{p}_2(\{t_k\})}^q \|C_{\mu,\om}\|_{H^p\to H^q}^q &\gtrsim \int_{0}^{1}\left(\sum_{k=0}^{N}|b_k|^2(1-|t_k|)^{2\lambda}\left(\int_{0}^{1}\frac{B^{\om}_t(s)}{\left(1-\overline{t_k}ts\right)^{\lambda}}\,d\mu(t)\right)^2\right)^{\frac{q}{2}}\,ds
\end{split}
\end{equation*}
Next, by applying Lemma~\ref{Lemma:LowerBoundEq}, we have
\begin{equation*}
\begin{split}
&\quad \int_{0}^{1}\left(\sum_{k=0}^{N}|b_k|^2(1-|t_k|)^{2\lambda}\left(\int_{0}^{1}\frac{B^{\om}_t(s)}{\left(1-\overline{t_k}ts\right)^{\lambda}}\,d\mu(t)\right)^2\right)^{\frac{q}{2}}\,ds\\
&=\sum_{m=0}^{\infty} \int_{t_m}^{t_{m+1}}\left(\sum_{k=0}^{N}|b_k|^2(1-|t_k|)^{2\lambda}\left(\int_{0}^{1}\frac{B^{\om}_t(s)}{\left(1-\overline{t_k}ts\right)^{\lambda}}\,d\mu(t)\right)^2\right)^{\frac{q}{2}}\,ds\\
&\geq \sum_{m=0}^{\infty} \int_{t_m}^{t_{m+1}}|b_m|^q(1-|t_m|)^{q\lambda}\left(\int_{0}^{1}\frac{B^{\om}_t(s)}{\left(1-\overline{t_m}ts\right)^{\lambda}}\,d\mu(t)\right)^q\,ds\\
&\gtrsim \sum_{m=0}^{N}\int_{t_m}^{t_{m+1}}|b_m|^q(2^m)^{-q\lambda}\frac{\mu_{2^m-1}^q (2^m)^{q(\lambda+1)}}{\om_{2^m}^q}\,ds\\
&\asymp \sum_{m=0}^{N}|b_m|^q\frac{\mu_{2^m-1}^q (2^m)^{q-1}}{\om_{2^m}^q}, \quad N \in \N.
\end{split}
\end{equation*}
Now by Lemma~\ref{Lemma:T^pequivalence}, we have that
\begin{equation*}
    \|(b_k)\|_{T^{p}_2(\{t_k\})}^p \asymp \sum_{k=0}^{N}\frac{|b_k|^p}{2^k}, \quad N \in \N.
\end{equation*}
Let $b_k=\left(\frac{\mu_{2^k-1} 2^k}{\om_{2^k}}\right)^{\frac{q}{p-q}}$ for $k\leq N$. We note that $|b_k|^p=|b_k|^q\frac{\mu_{2^k-1}^q(2^k)^{q}}{\om_{2^k}^q}=\left(\frac{\mu_{2^k-1} 2^k}{\om_{2^k}}\right)^{\frac{qp}{p-q}}$. Combining this with the previous estimate, we obtain by dividing both sides with $\left(\sum_{k=0}^{N}\frac{|b_k|^p}{2^k}\right)^{\frac{q}{p}}$ that
\begin{equation*}
    \|C_{\mu,\om}\|_{H^p \to H^q} \gtrsim \left(\sum_{k=0}^{N}\left(\frac{\mu_{2^k-1} 2^k}{\om_{2^k}}\right)^{\frac{qp}{p-q}}\frac{1}{2^k}\right)^{\frac{p-q}{qp}}, \quad N \in \N.
\end{equation*}
As the constants of comparison are independent of $N$, we obtain the dyadic version of the conclusion by taking the limit as $N \to \infty$. By elementary calculation and Lemma~\ref{DoublingLemma}, we can easily see that
\begin{equation}\label{Eq: SeriesComparisonBlockHq<p}
    \sum_{k=0}^{\infty}\left(\frac{\mu_{2^k-1} 2^k}{\om_{2^k}}\right)^{\frac{qp}{p-q}}\frac{1}{2^k} \asymp \sum_{k=0}^{\infty}\left(\frac{\mu_k (k+1)}{\om_{2k+1}}\right)^{\frac{qp}{p-q}}\frac{1}{(k+1)^2},
\end{equation}
with the constants of comparison depending only on $p,q$ and $\om$, which is equal to \textup{(c)}.
\end{proof}
 To end this section, we give the statement of Theorem~\ref{MainTheorem: Hardy} in the setting of $C_{\mu,\beta}$.

\begin{theorem}\label{MainTheorem: HardyBeta}
Let $0<p,q<\infty$, $0<\beta<\infty$, and let $\mu \in M^+([0,1))$. Then the following holds:
\begin{enumerate}
    \item[\textup{(i)}] If $p\leq q$, the following conditions are equivalent:
    \begin{enumerate}
        \item[\textup{(a)}] $C_{\mu,\beta}:H^p \to H^q$ is bounded;
        \item[\textup{(b)}]
        \begin{equation*}
            \|\mu^{{(p,q)}}_{\beta}\|_{L^{\infty}}=\sup_{0\leq r<1}\frac{\whm(r)}{(1-r)^{\beta+\frac{1}{p}-\frac{1}{q}}}<\infty;
        \end{equation*}
        \item[\textup{(c)}]
            \begin{equation*}
                 \|\mu^{{(p,q)}}_{\beta}\|_{\ell^{\infty}}=\sup_{n\in \N_0}\mu_n(n+1)^{\beta+\frac{1}{p}-\frac{1}{q}}<\infty;
            \end{equation*}
        \item[\textup{(d)}] For every fixed $\gamma>0$, we have $D^{\gamma+\frac{1}{p}-\frac{1}{q}}F_{\mu,\beta} \in X_{\gamma}$.
    \end{enumerate}
    \begin{equation*}
    \end{equation*}
    Moreover, we have $\|C_{\mu,\beta}\|_{H^p \to H^q} \asymp \|\mu^{{(p,q)}}_{\beta}\|_{L^{\infty}} \asymp \|\mu^{{(p,q)}}_{\beta}\|_{\ell^{\infty}}\asymp \|D^{\gamma+\frac{1}{p}-\frac{1}{q}}F_{\mu,\beta}\|_{X_{\gamma}}$.
    \item[\textup{(ii)}] If $q<p$, the following conditions are equivalent:
    \begin{enumerate}
        \item[\textup{(a)}] $C_{\mu,\beta}: H^p \to H^q$ is bounded;
        \item[\textup{(b)}] \begin{equation*}
        \|\mu_{\beta}\|_{L^{\frac{qp}{p-q}}}^\frac{qp}{p-q}=\int_{0}^{1}\left(\frac{\whm(r)}{(1-r)^{\beta}}\right)^{\frac{qp}{p-q}}\,dr +\whm(0)^{\frac{qp}{p-q}}<\infty;
        \end{equation*}
        \item[\textup{(c)}]
        \begin{equation*}
         \|\mu_{\beta}\|_{\ell^{\frac{qp}{p-q}}}^\frac{qp}{p-q}=\sum_{n=0}^{\infty}\left(\mu_n (n+1)^{\beta}\right)^{\frac{qp}{p-q}}\frac{1}{(n+1)^2}<\infty;
        \end{equation*}
        \item[\textup{(d)}] $F_{\mu,\beta} \in H^{\frac{qp}{p-q}}.$
    \end{enumerate}
Moreover, we have $\|C_{\mu,\beta}\|_{H^p \to H^q} \asymp \|\mu_{\beta}\|_{\ell^{\frac{qp}{p-q}}} \asymp \|\mu_{\beta}\|_{L^{\frac{qp}{p-q}}} \asymp \|F_{\mu,\beta}\|_{H^{\frac{qp}{p-q}}}$.
\end{enumerate}
\end{theorem}

\section{Proof of Theorem \ref{MainTheorem: Bergman}.}\label{Sec:Bergman}
There are cases which actually follow from the previous section.
\subsection{Case $p=q$.}
\begin{theorem} Let $0<p<\infty$, $\om\in \DD$, and let $\nu$ be a radial weight. The following conditions are equivalent.
\begin{enumerate}
            \item[\textup{(i)}] $C_{\mu,\om}: A^p_{\nu} \to A^p_{\nu}$ is bounded;
            \item[\textup{(ii)}] \begin{equation*}
                 \|\mu_{\om}\|_{L^\infty}=\sup_{0\leq r<1}\frac{\widehat{\mu}(r)}{\whw(r)(1-r)}<\infty;
            \end{equation*}
            \item[\textup{(iii)}]
            \begin{equation*}
                 \|\mu_\om\|_{\ell^\infty}=\sup_{n \in \N_0}\frac{\mu_n(n+1)}{\om_n}<\infty;
            \end{equation*}
\item[\textup{(iv)}] $C_{\mu,\om}: H^p \to H^p$ is bounded.
        \end{enumerate}
    Moreover, we have $\|C_{\mu,\om}\|_{A^p_{\nu} \to A^p_{\nu}}\asymp \|\mu_{\om}\|_{L^\infty} \asymp \|\mu_{\om}\|_{\ell^\infty}\asymp \|C_{\mu,\om}\|_{H^p \to H^p}$.
\end{theorem}
\begin{proof} By Corollary \ref{Corollary: TestConditionpleqq}, we get that \textup{(i)} implies \textup{(ii)}. Using now Theorem \ref{MainTheorem: Hardy} one sees that \textup{(ii)}, \textup{(iii)} and \textup{(iv)} are equivalent. {Finally assume \textup{(ii)} and use
the polar coordinates together with the estimate
$M_p(r,C_{\mu,\om}(f))\lesssim \|\mu_\om\|_{L^\infty}M_p(r,f)$ shown in \eqref{Eq: Lp-ineqCesaroq>1} and \eqref{Eq: Lp-ineqCesaroq<1} for $0\leq r<1$ to obtain (i).}
\end{proof}

\subsection{ Case $p< q$}
To prove our main results of this section, we need some essential auxiliary results. The most important one is a result which allows us to decompose weighted mixed norm space to $H^p$-blocks, with size of the blocks depending on the inducing weight. To state the theorem, we need to give some definitions on an universal Ces\`aro-basis for $H^p$. For $f \in \mathcal{H}(\D)$ and a polynomial $g$, given by $f(z)=\sum_{k=0}^{\infty}a_k z^k$ and $g(z)=\sum_{k=m}^{M}b_k z^k$ for some $m,M \in \N_0$, their Hadamard product $f * g$ is defined by
\begin{equation*}
    (f*g)(z)=\sum_{k=m}^{M}a_kb_k z^k, \quad z \in \D.
\end{equation*}
For a function $\Psi: \R \to \C$ such that $\Psi \in C^{\infty}(\R)$ is compactly supported  and $N \in \N$, we define the polynomials $W_N^{\Psi}$ by
\begin{equation*}
    W_N^{\Psi}(z)=\sum_{k=-\infty}^{\infty}\Psi\left(\frac{k}{N}\right)z^k, \quad z \in \D, \quad N \in \N.
\end{equation*}
We will consider polynomials of the form $W_N^{\Psi} * f$. For plenty of relevant information on these polynomials, see \cite[p. 140-144]{Pavlovic2014}. We will use the result given in \cite[Proposition 4]{PelaezDeLaRosa}, which states the following.
\begin{letterlemma}
    Let $K \in \N \setminus \{1\}$ and $\Psi: \R \to \R$ be a $C^{\infty}$-function such that $\Psi(x)=1$ for $x \leq 1$, $\Psi(x)=0$ for $x \ge K$ and $\Psi$ is decreasing and positive on $(1,K)$. Set $\Phi(x)=\Psi\left(\frac{x}{K}\right)-\Psi(x)$ for all $x \in \R$. Let $V_{0,K}(z)=\sum_{k=0}^{K-1}\Psi(k)z^k$ and
    $V_{n,K}(z)=\sum_{k=K^{n-1}}^{K^{n+1}-1}\Phi\left(\frac{k}{K^{n-1}}\right)z^k$ for every $n\in \mathbb N$ and  $z \in \D$. Then for every $f \in \H(\D)$ we have
    \begin{equation}\label{Eq:fvn}
        f(z)=\sum_{n=0}^{\infty}(V_{n,K}*f)(z), \quad z \in \D.
    \end{equation}
    Moreover, for every $0<p<\infty$ there exists $C=C(p,\Psi,K)>0$ such that
    \begin{equation}\label{Eq:vn}
        \|V_{n,K}*f\|_{H^p} \leq C\|f\|_{H^p}, \quad n \in \N.
    \end{equation}
\end{letterlemma}
We also need the following simple result from \cite[Lemma 3.1]{PavlovicLp}. Let $0<p<\infty$ and $f(z)=\sum_{k=m}^{M}a_k z^k$. Then
\begin{equation}\label{Eq: Mp vs Hp Polynomial}
    r^M\|f\|_{H^p} \leq M_p(r,f) \leq r^m\|f\|_{H^p}, \quad 0\leq r<1.
\end{equation}
The next result we use is essentially given in \cite[Proposition 9]{PelaezDeLaRosa}. Even though their result is stated for Bergman spaces, the proof for mixed norm spaces requires only minor modifications. We include a proof for the sake of completeness.
\begin{lemma}\label{decompositionlemma}
    Let  $0<p,q<\infty$ and $\om \in \DDD$. Then there exists $K_0=K_0(\om,p)$ such that for any fixed $K>K_0$, we have
    \begin{equation*}
    \int_{0}^{1}M_p^q(r,f)\om(r)r\,dr \asymp \sum_{n=0}^{\infty}\|V_{n,K}*f\|_{H^p}^q \om_{K^n}, \quad f \in \H(\D).
    \end{equation*}
\end{lemma}
\begin{proof}
Let $K_0$ be the maximum of constants $K$ in \eqref{eq:ReverseDoublingDef} for $\om$ and $\om_p$, where we define $\om_p(r)=\om(r^{\frac{1}{p}})r^{\frac{2}{p}-1}$ for $0\leq r<1$. By the properties of $V_{n,K}*f$, \eqref{Eq:vn}, \eqref{Eq: Mp vs Hp Polynomial}, and Lemma \ref{ReverseDoubling}, we have
\begin{equation*}
\begin{split}
\int_{0}^{1}M_p^q(r,f)\om(r)r\,dr &\gtrsim
\sum_{n=0}^{\infty}\int_{1-\frac{1}{K^{n+1}}}^{1-\frac{1}{K^{n+2}}}M_p^q(r,f)\om(r)r\,dr\\
&\gtrsim \sum_{n=0}^{\infty}\int_{1-\frac{1}{K^{n+1}}}^{1-\frac{1}{K^{n+2}}}\left(r^{K^{n+1}}\|V_{n,K}*f\|_{H^p}\right)^q\om(r)\,dr\\
&\gtrsim \sum_{n=0}^{\infty}\|V_{n,K}*f\|_{H^p}^q \whw\left(1-\frac{1}{K^{n+1}}\right)\\
&\asymp \sum_{n=0}^{\infty}\|V_{n,K}*f\|_{H^p}^q\om_{K^n}
\end{split}
\end{equation*}
For the other direction,  denoting $\alpha=\min\{p,1\}$, we can use \eqref{Eq:fvn} and \eqref{Eq: Mp vs Hp Polynomial} to obtain
\begin{equation*}
   M^\alpha_p(r,f) \lesssim \sum_{n=0}^{\infty}\|V_{n,K}*f\|^{\alpha}_{H^p}r^{\alpha K^{n-1}}, \quad 0 \leq r <1,
\end{equation*}
which yields
\begin{equation*}
\begin{split}
\int_{0}^{1}M_p^q(r,f)\om(r)r\,dr &\lesssim \|V_{n,K}*f\|_{H^p}^q\om_1+\int_{0}^{1}\left(\sum_{n=0}^{\infty}\left\|V_{n+1,K}*f\right\|^{\alpha}_{H^p}r^{\alpha K^n}\right)^{\frac{q}{\alpha}}\om(r)r\,dr\\
&\lesssim \|V_{n,K}*f\|_{H^p}^q\om_1
+\sum_{n=0}^{\infty}\om_{K^n}\|V_{n+1,K}*f\|_{H^p}^q\\
&\asymp \sum_{n=0}^{\infty}\om_{K^n}\|V_{n,K}*f\|_{H^p}^q,
\end{split}
\end{equation*}
where we have used \cite[Proposition 8]{PelaezDeLaRosa} on the last asymptotic inequality. Note that in the case $0<p<1$, we have used change of variables $r^{p}=x$, which together with Lemma~\ref{DoublingLemma} and the assumption on $K_0$ being large enough concludes the proof.
\end{proof}
\centerline{ \bf Case $p< q$, $q>1$}

The next result is fundamental to our approach, and it directly follows from the proof shown in \cite[Lemma 3.4]{LaPavlovic}.
\begin{letterlemma}\label{monotonepullout}
    Let $1<p<\infty$, let $\{\lambda_n\}$ be a monotonic sequence, and let $f \in \mathcal{H}(\D)$ be given by $f(z)=\sum_{n=0}^{\infty}a_nz^n$. Then for $M_1,M_2 \in \N$ satisfying $M_1<M_2$ there exist constants $c=c(p)>0$ and $C=C(p)>0$ such that
    \begin{equation*}
    \begin{split}
        \min\left\{|\lambda_{M_1}|, |\lambda_{M_2}|\right\}\left\|\sum_{n=M_1}^{M_2}a_n(\cdot)^n\right\|_{H^p} \leq c\left\|\sum_{n=M_1}^{M_2}\lambda_na_n(\cdot)^n\right\|_{H^p}
        \end{split}
    \end{equation*}
    and
    \begin{equation*}
    \left\|\sum_{n=M_1}^{M_2}\lambda_na_n(\cdot)^n\right\|_{H^p} \leq C \max\left\{|\lambda_{M_1}|, |\lambda_{M_2}|\right\}\left\|\sum_{n=M_1}^{M_2}a_n(\cdot)^n\right\|_{H^p}.
    \end{equation*}
\end{letterlemma}
Following results on weights belonging to $\DDD$ will be needed in the upcoming proofs.
\begin{lemma}\label{etahateta}  Let $\gamma>0$ and $\eta\in \DDD$. Then
\begin{equation} 
\int_r^1 (1-t)^\gamma \frac{\eta(t)}{\widehat\eta(t)}\asymp (1-r)^\gamma, \quad 0\leq r<1.
\end{equation}
\end{lemma}
\begin{proof} {Since $\eta\in \DD$, by Lemma~\ref{DoublingLemma} we can} choose $\epsilon$ small enough to get that $\frac{\widehat\eta(r)^\epsilon}{(1-r)^\gamma}$ is almost increasing in $r$.
Then
\begin{equation*}
\begin{split}
\int_r^1 (1-t)^\gamma\frac{\eta(t)}{\whe(t)} dt&\lesssim \frac{(1-r)^\gamma }{\whe(r)^\epsilon}\int_r^1 \frac{\eta(t)}{\whe(t)^{1-\epsilon}} dt\\
&= \frac{(1-r)^\gamma  }{\whe(r)^\epsilon}\int_r^1 \frac{1}{\epsilon} (-\whe^\epsilon)'(t) dt\\
&\asymp (1-r)^\gamma, \quad 0\leq r<1.
\end{split}
\end{equation*}
The other inequality follows easily using {Lemma~\ref{DoublingLemma} and Lemma~\ref{ReverseDoubling}}, since
\begin{equation*}
\begin{split}
\int_r^1 (1-t)^\gamma\frac{\eta(t)}{\widehat\eta(t)} dt&\gtrsim  (1-r)^\gamma \int_r^{1-\frac{1-r}{K}}\frac{\eta(t) }{\whe(t)}dt\\
&\geq  \frac{(1-r)^\gamma}{\whe (r)}\int_r^{1-\frac{1-r}{K}}\eta(t)dt \\
&\gtrsim  (1-r)^\gamma, \quad 0\leq r<1.
\end{split}
\end{equation*}
\end{proof}

\begin{lemma} \label{weightVpq} Let $0<p<\infty$, $1\le q<\infty$, $\nu, \eta\in \DDD$ and $\om\in \DD$. Let us define $W_{\nu,p}$ and $ V_{\om,\nu}^{p,q}$ by
$$W_{\nu,p}(r)=\frac{1}{p}\whv(r)^{\frac{1}{p}-1}\nu(r) \quad\hbox {and}\quad V_{\om,\nu}^{p,q}(r)=\whw(r)(1-r)^{1+\frac{1}{p}-\frac{1}{q}}W_{\nu,p}(r).$$

Then $W_{\nu,p}$, $V_{\om,\nu}^{p,q}$ and $ \left( \frac{\widehat{V_{\om,\nu}^{p,q}}}{\widehat{W_{\eta, q}}}\right)^q\eta$  belong to $\DDD$. Moreover, we have
$$\widehat{W_{\nu,p}}(r)=\whv(r)^{\frac{1}{p}} \quad\hbox {and}\quad  \widehat{V_{\om,\nu}^{p,q}}(r)\asymp\whw(r)\widehat\nu(r)^{\frac{1}{p}}(1-r)^{1+\frac{1}{p}-\frac{1}{q}}.$$
\end{lemma}
\begin{proof} 
It is easy to see that $\widehat{W_{\nu,p}}(r)=\whv(r)^{\frac{1}{p}}$ for $0\leq r<1$, which clearly implies that $W_{\nu,p} \in \DDD$ by Lemma~\ref{DoublingLemma} and  Lemma~\ref{ReverseDoubling}. For the other cases we shall make use of the following result 
(see  part (ii) in  \cite[Lemma 3]{PR2025}) where it is shown that for any  $v\in \DD$ and $u\in \DDD$ then 
\begin{equation} \label{DDD} \widehat v^a u\in \DDD \hbox{ and }\int_r^1\widehat v(t)^a u(t)dt\asymp \widehat v (r)^a\widehat u(r), \quad 0\leq r<1.\end{equation} 
Selecting $v(r)=(1+\frac{1}{p}-\frac{1}{q})(1-r)^{\frac{1}{p}-\frac{1}{q}}$, we have that $\widehat v(r)=(1-r)^{1+\frac{1}{p}-\frac{1}{q}}$, $u=\widehat v W_{\nu,p}\in \DDD$ and $\widehat u\asymp \widehat v\widehat{W_{\nu,p}}$ . Now, using (\ref{DDD}) for $v=\om$ 
 we have $V_{\om,\nu}^{p,q}=\whw \widehat v  W_{\nu,p}\in \DDD$ and \begin{equation} \label{vpq}\widehat{ V_{\om,\nu}^{p,q}}(r) \asymp \widehat{W_{\nu,p}}(r)\whw(r)\widehat v(r), \quad 0 \leq r<1. \end{equation}  
Let us now see that $\left( \frac{\widehat{V_{\om,\nu}^{p,q}}}{\widehat{W_{\eta, q}}}\right)^q\eta\in \DDD$.
By using \eqref{vpq}
we can  write $$  \left( \frac{\widehat{V_{\om,\nu}^{p,q}}}{\widehat{W_{\eta, q}}}\right)^q\eta\asymp\widehat\om^q \widehat v^{q}  \left(   \frac{\widehat{W_{\nu,p}}}{\widehat{W_{\eta, q}}}\right)^q\eta.$$
Hence from (\ref{DDD})  to finish the proof it suffices to show that $\widehat v^{q}  \left(   \frac{\widehat{W_{\nu,p}}}{\widehat{W_{\eta, q}}}\right)^q\eta\in \DDD$. Notice that 
$$\widehat v(r)^{q}\left( \frac{\widehat{W_{\nu,p}}(r)}{\widehat{W_{\eta, q}}(r)}\right)^q\eta(r)=(1-r)^{\gamma}\widehat\nu(r)^{\frac{q}{p}}\frac{\eta(r)}{\widehat\eta(r)} , \quad 0\leq r<1,$$
where $\gamma=q(1+\frac{1}{p}-\frac{1}{q})>0$ . This now follows using Lemma \ref{etahateta}, since
$$
\int_r^1 (1-t)^\gamma \widehat\nu(t)^{\frac{q}{p}}\frac{\eta(t)}{\widehat\eta(t)} dt\le \widehat\nu(r)^{\frac{q}{p}}\int_r^1 (1-t)^\gamma\frac{\eta(t)}{\widehat\eta(t)} dt \asymp  \widehat\nu(r)^{\frac{q}{p}}(1-r)^\gamma $$
for every $0\leq r<1$. By Lemma~\ref{DoublingLemma} and Lemma~\ref{ReverseDoubling} we also obtain
\begin{equation*}
\begin{split}
\int_r^1 (1-t)^\gamma\widehat\nu(t)^{\frac{q}{p}}\frac{\eta(t)}{\widehat\eta(t)} dt&\gtrsim  (1-r)^\gamma \int_r^{1-\frac{1-r}{K}}\frac{ \widehat\nu(t)^{\frac{q}{p}} }{\widehat\eta(t)}\eta(t)dt\\
&\ge \frac{(1-r)^\gamma\widehat\nu(1-\frac{1-r}{K})^{\frac{q}{p}}}{\widehat\eta (r)}\int_r^{1-\frac{1-r}{K}}\eta(t)dt \\
&\gtrsim  (1-r)^\gamma\widehat\nu(r)^{\frac{q}{p}}, 
\quad 0\leq r<1.
\end{split}
\end{equation*}
Therefore the result is complete.
\end{proof}

Using these results and fractional derivatives induced by weights, we can establish the following result, which contains part of Theorem~\ref{MainTheorem: Bergman}.
\begin{theorem}\label{Bergmanp<q}
   $0<p\le q<\infty$, $1<q<\infty$, $\om \in \DD$, $\nu,\eta \in \DDD$, and let $\mu \in M^+([0,1))$. Then the following conditions are equivalent:
        \begin{enumerate}
            \item[\textup{(i)}] $C_{\mu,\om}: A^p_{\nu} \to A^q_{\eta}$ is bounded;
            \item[\textup{(ii)}] \begin{equation*}
                 \|\mu^{(p,q)}_{\om,\nu,\eta}\|_{L^\infty}=\sup_{0\leq r<1}\frac{\widehat{\mu}(r)\whe(r)^{\frac{1}{q}}}{\whw(r)\whv(r)^{\frac{1}{p}}(1-r)^{1+\frac{1}{p}-\frac{1}{q}}}<\infty;
            \end{equation*}
            \item[\textup{(iii)}]
            \begin{equation*}
                 \|\mu^{(p,q)}_{\om,\nu,\eta}\|_{\ell^\infty}=\sup_{n \in \N_0}\frac{\mu_n\eta_n^{\frac{1}{q}}(n+1)^{1+\frac{1}{p}-\frac{1}{q}}}{\om_n\nu_n^{\frac{1}{p}}};
            \end{equation*}
        \end{enumerate}
    Moreover, we have $\|C_{\mu,\om}\|_{A^p_{\nu} \to A^q_{\eta}}\asymp \|\mu^{(p,q)}_{\om,\nu,\eta}\|_{L^\infty} \asymp \|\mu^{(p,q)}_{\om,\nu,\eta}\|_{\ell^\infty}$.
\end{theorem}
\begin{proof}
Lemma~\ref{Momentvstail} shows that (ii) is equivalent to (iii).   By Corollary~\ref{Corollary: TestConditionpleqq}, we thus only need to show $\|\mu_{\om,\nu,\eta}\|_{\ell^\infty}<\infty$ implies the boundedness of $C_{\mu,\om}: A^p_{\nu} \to A^q_{\eta}$.  We will use the decomposition norm given in Lemma~\ref{decompositionlemma}. Let $K=K(\eta)>0$ be fixed large enough depending on \eqref{eq:ReverseDoublingDef}. Denote by $\{\widetilde{K}_n\}$ the sequence $\{0,1,K,K^2, \dots, K^{n-1},\dots\}$. Applying Lemma~\ref{decompositionlemma}, {Minkowski's inequality}, Lemma~\ref{monotonepullout} and Lemma~\ref{DoublingLemma}, we obtain
\begin{equation*}
\begin{split}
\|C_{\mu,\om}(f)\|_{A^q_{\eta}}^q &\asymp \sum_{n=0}^{\infty}\eta_{K^n}\|V_{n,K}*C_{\mu,\om}(f)\|_{H^q}^q\\
 &\asymp \sum_{n=0}^{\infty}\eta_{K^n}\left\|\int_0^1( V_{n,K}* f_tB_t^\om) d\mu(t)\right\|_{H^q}^q\\
&\lesssim \sum_{n=0}^{\infty}\eta_{K^n}\left(\int_0^1 \|V_{n,K}* f_tB_t^\om \|_{H^q}d\mu(t) \right)^q\\
&\lesssim\sum_{n=0}^{\infty}\eta_{K^n}\mu_{\widetilde{K}_n}^q\|V_{n,K}*(fB^{\om})\|_{H^q}^q\\
&\lesssim \|\mu^{(p,q)}_{\om,\nu,\eta}\|_{\ell^\infty}^q \sum_{n=0}^{\infty}\frac{\om_{K^n}^q\nu_{K^n}^{\frac{q}{p}}}{(K^n)^{q+\frac{q}{p}-1}}\|V_{n,K}*(fB^{\om})\|_{H^q}^q.
\end{split}
\end{equation*}
We will next show that the quotient of the several moments appearing can be simplified to a quotient of two moments of the weights introduced in Lemma \ref{weightVpq}, corresponding to
$W_{\nu,p}(r)=\frac{1}{p}\whv(r)^{\frac{1}{p}-1}\nu(r)$,
$
W_{\eta,q}(r)=\frac{1}{q}\whe(r)^{\frac{1}{q}-1}\eta(r)$ and  $V_{\om,\nu}^{p,q}(r)=\whw(r)(1-r)^{1+\frac{1}{p}-\frac{1}{q}}W_{\nu,p}(r)$. Using Lemma \ref{weightVpq} we obtain
\begin{equation*}
\begin{split}
\frac{\om_{K^n}^q\nu_{K^n}^{\frac{q}{p}}}{(K^n)^{q+\frac{q}{p}-1}}&=\left(\frac{\om_{K^n}\nu_{K^n}^{\frac{1}{p}}}{(K^n)^{1+\frac{1}{p}-\frac{1}{q}}}\right)^q\\
&\asymp \left(\left((1-\left(1-\frac{1}{K^n}\right)\right)^{1+\frac{1}{p}-\frac{1}{q}}\whw\left(1-\frac{1}{K^n}\right)\widehat{W_{\nu,p}}\left(1-\frac{1}{K^n} \right)\right)^q\\
&\asymp V_{\om,\nu}^{p,q}\left(1-\frac{1}{K^n}\right)^q\asymp (V_{\om,\nu}^{p,q})_{K^n}^q, \quad n \in \N_0.
\end{split}
\end{equation*}
Since $\eta_{K^n} \asymp (W_{\eta,q})_{K^n}^q$, thus, Lemma~\ref{monotonepullout} yields
\begin{equation*}
\begin{split}
\|C_{\mu,\om}(f)\|_{A^q_{\eta}}^q
&\lesssim \|\mu^{(p,q)}_{\om,\nu,\eta}\|_{\ell^\infty}^q \sum_{n=0}^{\infty}\eta_{K^n}\left(\frac{(V_{\om,\nu}^{p,q})_{K^n}}{(W_{\eta,q})_{K^n}}\right)^q\|V_{n,K}*(fB^{\om})\|_{H^q}^q\\
&\asymp \|\mu^{(p,q)}_{\om,\nu,\eta}\|_{\ell^\infty}^q \sum_{n=0}^{\infty}\eta_{K^n}\left(\frac{(V_{\om,\nu}^{p,q})_{2K^{n+1}-1}}{(W_{\eta,q})_{2K^{n-1}+1}}\right)^q\|V_{n,K}*(fB^{\om})\|_{H^q}^q\\
&\lesssim \|\mu^{(p,q)}_{\om,\nu,\eta}\|_{\ell^\infty}^q \sum_{n=0}^{\infty}\eta_{K^n}\|V_{n,K}*(R^{V_{\om,\nu}^{p,q},W_{\eta,q}}(fB^{\om}))\|_{H^q}^q\\
&\asymp \|\mu^{(p,q)}_{\om,\nu,\eta}\|_{\ell^\infty}^q\int_{\D}\left|R^{V_{\om,\nu}^{p,q},W_{\eta,q}}(fB^{\om})(z)\right|^q\eta(z)\,dA(z).
\end{split}
\end{equation*}
To estimate the integral, we apply recently shown Littlewood-Paley inequalities for fractional derivatives, see \cite[{Theorem 1}]{PelaezDeLaRosa} and \cite[{Theorem 4}]{PRW}.  Since  $\left(\widehat{V_{\om,\nu}^{p,q}}\right)^q\eta$ and $\left(\frac{\widehat{V_{\om,\nu}^{p,q}}^q}{\widehat{W_{\eta,q}}}\right)^q\eta$ belong to $\DDD$ {by Lemma~\ref{weightVpq},} we apply the Littlewood-Paley inequality twice, which yields 
\begin{equation*}
\begin{split}
&\quad \int_{\D}\left|R^{V_{\om,\nu}^{p,q},W_{\eta,q}}(fB^{\om})(z)\right|^q\eta(z)\,dA(z)\\
&\asymp \int_{\D}\left|D^{W_{\eta,q}}(fB^{\om})(z)\right|^q \widehat{V_{\om,\nu}^{p,q}}(z)^q\eta(z)\,dA(z)\\
&=\int_{\D}\left|D^{W_{\eta,q}}(fB^{\om})(z)\right|^q\widehat{W_{\eta,q}}(z)^q \frac{\widehat{V_{\om,\nu}^{p,q}}(z)^q}{\widehat{W_{\eta,q}}(z)^q}\eta(z)\,dA(z)\\
&\asymp \int_{\D}|f(z)|^q|B^{\om}(z)|^q \frac{\widehat{V_{\om,\nu}^{p,q}}(z)^q}{\widehat{W_{\eta,q}}(z)^q}\eta(z)\,dA(z).
\end{split}
\end{equation*}
To conclude the proof, we show that the measure $m$ defined by $dm(z)=|B^{\om}(z)|^q \frac{\widehat{V_{\om,\nu}^{p,q}}(z)^q}{\widehat{W_{\eta,q}}(z)^q}\eta(z)\,dA(z)$ is a $q$-Carleson measure for $A^p_{\nu}$. By \cite[Theorem 1]{PR2015}, it suffices to show that $m(S(a)) \lesssim \nu(S(a))^{\frac{q}{p}}$ for every $a \in \D$. Assume first that $p<q$. By the pointwise estimates for the Bergman kernel \eqref{Eq: PointwiseKernelEstimate}, Lemma~\ref{weightVpq} and Lemma~\ref{etahateta}, we obtain
\begin{equation*}
\begin{split}
\int_{S(a)}|B^{\om}(z)|^q \frac{\widehat{V_{\om,\nu}^{p,q}}(z)^q}{\widehat{W_{\eta,q}}(z)^q}\eta(z)\,dA(z) &\lesssim \int_{S(a)}\frac{\widehat{V_{\om,\nu}^{p,q}}(z)^q}{\whw(z)^q(1-|z|)^q\widehat{W_{\eta,q}}(z)^q}\eta(z)\,dA(z)\\
&\asymp(1-|a|)\int_{|a|}^{1}\whv(r)^{\frac{q}{p}}(1-r)^{\frac{q}{p}-1}\frac{\eta(r)}{\whe(r)}r\,dr\\
&\le(1-|a|)\whv(a)^{\frac{q}{p}}\int_{|a|}^{1}(1-r)^{\frac{q}{p}-1}\frac{\eta(r)}{\whe(r)}\,dr\\
&\asymp \nu(S(a))^{\frac{q}{p}}, \quad a \in \D.
\end{split}
\end{equation*}
When $q=p$, the first three steps of the previous deduction are still valid. Using the hypothesis that $\nu \in \DDD$, we have by Lemma \ref{ReverseDoubling} that there exists $\alpha=\alpha(\nu)>0$ such that $\frac{\whv(r)}{(1-r)^{\alpha}}$ is almost decreasing with respect to $r$. Thus, applying Lemma \ref{etahateta} we obtain
$$
(1-|a|)\int_{|a|}^{1}\whv(r)\frac{\eta(r)}{\whe(r)}\,dr
\lesssim (1-|a|)\frac{\whv(a)}{(1-|a|)^{\alpha}}\int_{|a|}^{1}(1-r)^{\alpha}\frac{\eta(r)}{\whe(r)}\,dr
\asymp \nu(S(a)), \quad a \in \D.$$
 Combining all the previous steps, we have shown that
\begin{equation*}
    \|C_{\mu,\om}\|_{A^q_{\eta}}^q \lesssim \|\mu_{\om,\nu,\eta}\|_{L^{\infty}}^q\|f\|_{A^p_{\nu}}^q,
\end{equation*}
which concludes the proof.
\end{proof}
\centerline{ \bf Case $p\le q\le 1$}

We next consider the case $C_{\mu,\om}:A^p_{\nu} \to A^q_{\nu}$ when $0<p\leq q\leq 1$. We need the following Bergman space analog of a well-known theorem by Hardy and Littlewood, see \cite[Theorem 5.11]{Duren}. We believe that this result is of independent interest.

\begin{theorem}\label{HLforBergman}
Let $0<p<q$, $\lambda\geq p$, and $\om \in \DD$. Then there exists a constant $C=C(\om,p,q,\lambda)>0$ such that
\begin{equation*}
\int_{0}^{1 }\om(S(r))^{\lambda\left(\frac{1}{p}-\frac{1}{q}\right)}\|f_r\|_{A^q_{\om}}^{\lambda}\,\frac{dr}{1-r} \leq C\|f\|_{A^p_{\om}}^{\lambda}.
\end{equation*}
\end{theorem}

\begin{proof}
For $f \in A^p_{\om}$, we first need an estimate for the $A^q_{\om}$-norm of the dilatation of $f$. Observe that for $q>p$ and $\om \in \DD$, we have by the pointwise estimates for functions in $A^p_{\om}$ that
\begin{equation}\label{ApAqineq}
\begin{split}
\|f_r\|_{A^q_{\om}}&=\left(\int_{\D}|f_r(z)|^q\om(z)\,dA(z)\right)^{\frac{1}{q}}\\
&=\left(\int_{\D}|f_r(z)|^p|f_r(z)|^{q-p}\om(z)\,dA(z)\right)^{\frac{1}{q}}\\
&\leq C\frac{\|f\|_{A^p_{\om}}}{\om(S(r))^{\frac{1}{p}-\frac{1}{q}}}, \quad 0<r<1.
\end{split}
\end{equation}
    We first prove the special case $\lambda=p$, from where the general case follows.
   Let $\eta$ be the measure  given by $d\eta(r)=\widehat{\dot{\om}}(r)\,dr$, where $dr$ is the Lebesgue measure on $[0,1)$ and $\dot{\om}(r)=\om(r)r$.  In particular, we have $\om(S(r))=\widehat{\dot{\om}}(r)(1-r)$ for $0<r<1$, which is continuous and strictly decreasing to $0$ and it can be  extended  to $[0,1)$ by setting $\om(S(0))=\widehat{\dot{\om}}(0)$. 
 Define the operator $T$ by the formula
    \begin{equation*}
    T(f)(r)=\om(S(r))^{-\frac{1}{q}}\|f_r\|_{A^q_{\om}}, \quad f \in A^p_{\om}, \quad 0\leq r<1.
    \end{equation*}
Our aim is to show that $T: A^p_{\om} \to L^{p}([0,1], \eta)$ is bounded for every $0<p<q$. 
By Marcinkiewicz interpolation theorem \cite[Theorem 1.3.2]{Grafakos}, it suffices to show that $T: A^p_{\om} \to L^{p,\infty}([0,1], \eta)$ is bounded for every $0<p<q$.

Thus, for every $t>0$ we obtain
    \begin{equation*}
    \begin{split}
    \int_{\{0\leq r<1: \,|T(f)(r)|>t\}}\widehat{\dot{\om}}(r)\,dr
   & \leq \int_{\{0\leq r<1: \, C\om(S(r))^{-\frac{1}{p}}\|f\|_{A^p_{\om}}>t\}}\widehat{\dot{\om}}(r)\,dr\\
&= \int_{\{0\leq r<1: \, \frac{C^p\|f\|_{A^p_{\om}}^p}{t^p}>\om(S(r)\}}\widehat{\dot{\om}}(r)\,dr
    \end{split}
    \end{equation*}
    where the constant $C>0$ comes from \eqref{ApAqineq}.

   Let $t_0>0$ be small enough such that $\om(S(0))<\frac{C^p\|f\|_{A^p_{\om}}^p}{t_0^p}$. Hence for $0<t\le t_0$ we can estimate
    \begin{equation*}
    \int_{\{0\leq r<1: \,|T(f)(r)|>t\}}\widehat{\dot{\om}}(r)\,dr\le\int_{0}^{1}\widehat{\dot{\om}}(r)\,dr \leq \widehat{\dot{\om}}(0)= \om(S(0))\le \frac{C^p\|f\|_{A^p_{\om}}^p}{t^p}.
    \end{equation*}
    For $t>t_0$, let $r_t \in (0,1)$ be the point satisfying $\om(S(r_t)) =\frac{C^p\|f\|_{A^p_{\om}}^p}{t^p}$.This yields
    \begin{equation*}
    \begin{split}
        \int_{\{0\leq r<1: \, C\om(S(r))^{-\frac{1}{p}}\|f\|_{A^p_{\om}}>t\}}\widehat{\dot{\om}}(r)\,dr
        &\leq \int_{r_t}^{1}\widehat{\dot{\om}}(r)\,dr
    \leq \widehat{\dot{\om}}(r_t)(1-r_t)\\
    &=\om(S(r_t))=\frac{C^p\|f\|_{A^p_{\om}}^p}{t^p}, \quad 0<t<\infty.
    \end{split}
    \end{equation*}
    Therefore $T: A^p_{\om} \to L^{p,\infty}([0,1], \eta)$ is bounded for every $0<p<q$ and using interpolation  we have
    \begin{equation*}
    \int_{0}^{1}\|f_r\|_{A^q_{\om}}^p\om(S(r))^{-\frac{p}{q}}\widehat{\dot{\om}}(r)\,dr \lesssim \|f\|_{A^p_{\om}}^{p}.
    \end{equation*}

    For the general case, we let $\lambda>p$. Thus, by \eqref{ApAqineq} once again we have
    \begin{equation*}
    \begin{split}
\int_{0}^{1}\|f_r\|^{\lambda}_{A^q_{\om}}\om(S(r))^{\lambda\left(\frac{1}{p}-\frac{1}{q}\right)}\,\frac{dr}{1-r}
&\lesssim \|f\|_{A^p_{\om}}^{\lambda-p}\int_{0}^{1}\|f_r\|^p_{A^q_{\om}}\om(S(r))^{\lambda\left(\frac{1}{p}-\frac{1}{q}\right)-(\lambda-p)\left(\frac{1}{p}-\frac{1}{q}\right)}\,\frac{dr}{1-r}\\
&=\|f\|_{A^p_{\om}}^{\lambda-p}\int_{0}^{1}\|f_r\|^p_{A^q_{\om}}\om(S(r))^{1-\frac{p}{q}}\,\frac{dr}{1-r}\\
&=\|f\|_{A^p_{\om}}^{\lambda-p}\int_{0}^{1}\|f_r\|^p_{A^q_{\om}}\om(S(r))^{-\frac{p}{q}}\widehat{\dot{\om}}(r)\,dr \lesssim\|f\|^{\lambda}_{A^p_{\om}}.
    \end{split}
    \end{equation*}
This concludes the proof.
\end{proof}

\begin{theorem}\label{Theorem:ApAq q<1}
Let $0<p\leq q\leq 1$, $\om,\nu \in \DD$, and let $\mu \in M^+([0,1))$. Then the following conditions are equivalent:
        \begin{enumerate}
            \item[\textup{(i)}] $C_{\mu,\om}: A^p_{\nu} \to A^q_{\nu}$ is bounded;
            \item[\textup{(ii)}] \begin{equation*}
                 \|\mu^{(p,q)}_{\om,\nu}\|_{L^\infty}=\sup_{0\leq r<1}\frac{\widehat{\mu}(r)}{\whw(r)\whv(r)^{\frac{1}{p}-\frac{1}{q}}(1-r)^{1+\frac{1}{p}-\frac{1}{q}}}<\infty;
            \end{equation*}
            \item[\textup{(iii)}]
            \begin{equation*}
                 \|\mu^{(p,q)}_{\om,\nu}\|_{\ell^\infty}=\sup_{n \in \N_0}\frac{\mu_n(n+1)^{1+\frac{1}{p}-\frac{1}{q}}}{\om_n\nu_n^{\frac{1}{p}-\frac{1}{q}}};
            \end{equation*}
            \item[\textup{(iv)}] For every fixed $\beta>0$, we have $D^{\beta+\frac{1}{p}-\frac{1}{q}}D^{\nu^{\frac{1}{p}-\frac{1}{q}}}F_{\mu,\om} \in X_{\beta}$.
        \end{enumerate}
    Moreover, we have $\|C_{\mu,\om}\|_{A^p_{\nu} \to A^q_{\nu}}\asymp \|\mu^{(p,q)}_{\om,\nu}\|_{L^\infty} \asymp \|\mu^{(p,q)}_{\om,\nu}\|_{\ell^\infty} \asymp \|D^{\beta+\frac{1}{p}-\frac{1}{q}}D^{\nu^{\frac{1}{p}-\frac{1}{q}}}F_{\mu,\om}\|_{X_{\beta}}$.
\end{theorem}

\begin{proof}
Lemma~\ref{Momentvstail} shows that $\|\mu^{(p,q)}_{\om,\nu}\|_{L^{\infty}} \asymp  \|\mu^{(p,q)}_{\om,\nu}\|_{\ell^\infty}$ and methods similar to the proof of Proposition~\ref{Prop: Fmuom sup norm space} yield the equivalence between \textup{(ii)}, \textup{(iii)}, and \textup{(iv)}. By Corollary~\ref{Corollary: TestConditionpleqq}, we only need to show that the condition $\|\mu^{(p,q)}_{\om,\nu}\|_{L^\infty}<\infty$ ensures the boundedness of $C_{\mu,\om}: A^p_{\nu} \to A^q_{\nu}$ with the correct norm estimate. We once again let $g(z)=f(z)B^{\om}(z)$ for every $z \in \D$. Let $t_k=1-2^{-k}$ for every $k \in \N$. With this notation, by Fubini's theorem and subadditivity we have
\begin{equation*}
\begin{split}
\|C_{\mu,\om}(f)\|_{A^q_{\nu}}^q
&=\int_{\D}\left|\int_{0}^{1}f(tz)B^{\om}_t(z)\,d\mu(t)\right|^q\nu(z)\,dA(z)\\
&\leq \int_{\D}\sum_{k=0}^{\infty}\left(\int_{t_k}^{t_{k+1}}|g(tz)|\,d\mu(t)\right)^q\nu(z)\,dA(z)\\
&\leq \sum_{k=0}^{\infty}\whm(t_k)^q\int_{\D}\sup_{0\leq t\leq t_{k+1}}|g(tz)|^q\nu(z)\,dA(z)\\
\end{split}
\end{equation*}
It is easy to see that the radial maximal theorem \cite[Theorem 7.1.4]{Pavlovic2004} yields the corresponding result for every weighted Bergman space induced by a radial weight. Thus, the assumption and the hypothesis that $\om,\nu \in \DD$ yields
\begin{equation*}
\begin{split}
\|C_{\mu,\om}(f)\|_{A^q_{\nu}}^q
&\leq \sum_{k=0}^{\infty}\whm(t_k)^q\int_{\D}\sup_{0\leq t\leq t_{k+1}}|g(tz)|^q\nu(z)\,dA(z)\\
&\lesssim \|\mu^{(p,q)}_{\om,\nu}\|_{L^\infty}^q\sum_{k=0}^{\infty}\om(S(t_k))^q\nu(S(t_k))^{\frac{q}{p}-1}\|g_{t_{k+1}}\|_{A_{\nu}^q}^q\\
&\asymp  \|\mu^{(p,q)}_{\om,\nu}\|_{L^\infty}^q\sum_{k=0}^{\infty}\int_{t_{k+1}}^{t_{k+2}}\,dt\|g_{t_{k+1}}\|_{A_{\nu}^q}^q\frac{\om(S(t_k))^q\nu(S(t_k))^{\frac{q}{p}-1}}{1-t_k}\\
&\lesssim  \|\mu^{(p,q)}_{\om,\nu}\|_{L^\infty}^q\sum_{k=0}^{\infty}\int_{t_{k+1}}^{t_{k+2}}\|g_{t}\|_{A_{\nu}^q}^q \om(S(t))^q\nu(S(t))^{\frac{q}{p}-1}\frac{dt}{1-t}\\
&\leq  \|\mu^{(p,q)}_{\om,\nu}\|_{L^\infty}^q\int_{0}^{1}\|g_{t}\|_{A_{\nu}^q}^q \om(S(t))^q\nu(S(t))^{\frac{q}{p}-1}\frac{dt}{1-t},
\end{split}
\end{equation*}
where we also used the fact that $\|g_r\|_{A^q_{\nu}}$ is increasing with respect to $r$. As $\nu \in \DD$, there exists a $\beta>0$ such that $\whv(r)(1-r)^{-\beta}$ is almost increasing by Lemma~\ref{DoublingLemma}. Let $x,x'>1$ be fixed such that $\frac{1}{x}+\frac{1}{x'}=1$ and $x'>\frac{1+\beta}{q}$. Using the $L^p$-estimates for reproducing kernels shown in Lemma~\ref{Lemma: KernelLpEstimate}, we obtain
\begin{equation*}
\begin{split}
\|g_{t}\|_{A_{\nu}^q}^q&=\int_{\D}|f(tz)|^q|B^{\om}_t(z)|^q\nu(z)\,dA(z)\\
&\leq \left(\int_{\D}|f(tz)|^{qx}\nu(z)\,dA(z)\right)^{\frac{1}{x}}\left(\int_{\D}|B^{\om}_t(z)|^{qx'}\nu(z)\,dA(z)\right)^{\frac{1}{x'}}\\
&\asymp \|f_t\|_{A^{qx}_{\nu}}^q \left(\int_{0}^{t}\frac{\whv(s)}{\whw(s)^{qx'}(1-s)^{qx'}}\,ds+1\right)^{\frac{1}{x'}}\\
&\lesssim\|f_t\|_{A^{qx}_{\nu}}^q \left(\frac{\whv(t)}{(1-t)^{\beta}\whw(t)^{qx'}}\int_{0}^{t}\frac{1}{(1-s)^{qx'-\beta}}\,ds+1\right)^{\frac{1}{x'}}\\
&\asymp \|f_t\|_{A^{qx}_{\nu}}^q \left(\frac{\whv(t)}{(1-t)^{\beta+qx'-1-\beta}\whw(t)^{qx'}}+1\right)^{\frac{1}{x'}}\\
&\asymp \|f_t\|_{A^{qx}_{\nu}}^q \frac{\nu(S(t))^{\frac{1}{x'}}}{\om(S(t))^{q}}, \quad 0\leq t<1.
\end{split}
\end{equation*}
Thus, by Theorem~\ref{HLforBergman}, we obtain
\begin{equation*}
\begin{split}
\int_{0}^{1}\|g_{t}\|_{A_{\nu}^q}^q\om(S(t))^q\nu(S(t))^{\frac{q}{p}-1}\frac{dt}{1-t}
&\lesssim \int_{0}^{1}\|f_t\|_{A^{qx}_{\nu}}^q \nu(S(t))^{\frac{q}{p}-1+\frac{1}{x'}}\frac{dt}{1-t}\\
&=\int_{0}^{1}\|f_t\|_{A^{qx}_{\nu}}^q \nu(S(t))^{q\left(\frac{1}{p}-\frac{1}{qx} \right)}\frac{dt}{1-t}\\
&\lesssim \|f\|^q_{A^p_{\nu}},
\end{split}
\end{equation*}
which concludes the proof with the right norm estimate.
\end{proof}

\subsection{ Case $q<p$.}  Note that unlike in the Hardy spaces, the weighted Bergman spaces have a useful atomic decomposition, which allows us to proceed without any theory on tent spaces.

\begin{lemma}\label{q<ptestingbergman}
    Let $0<q<p<\infty$, $\om \in \DD$ and $\nu,\eta \in \DDD$, and let $\mu \in M^+([0,1))$. Then, if $C_{\mu,\om}: A^{p}_{\nu} \to A^q_{\eta}$, we have
    \begin{equation*}
\|\mu_{\om,\nu,\eta}\|_{\ell_{\frac{\nu_n}{(n+1)^2}}^\frac{qp}{p-q}}=\left(\sum_{n=0}^{\infty}\left(\frac{\mu_{n} (n+1)}{\om_{n}}\frac{\eta_{n}^\frac{1}{q}}{\nu_{n}^\frac{1}{q}}\right)^{\frac{qp}{p-q}}\frac{\nu_{n}}{(n+1)^2}\right)^{\frac{p-q}{qp}} <\infty.
    \end{equation*}
    Moreover, we have
    \begin{equation*}
\|\mu_{\om,\nu,\eta}\|_{\ell_{\frac{\nu_n}{(n+1)^2}}^\frac{qp}{p-q}}\lesssim \|C_{\mu,\om}\|_{A^p_{\nu}\to A^q_{\eta}}.
    \end{equation*}
\end{lemma}

\begin{proof}
We use the atomic decomposition given by \cite[Theorem 1]{PRS}, which says that for a pseudohyperbolically separated sequence $\{z_k\}$, a sequence $\{a_k\} \in \ell^p$, and a sufficiently large $M >0$, the function $f$ defined by
\begin{equation*}
    f(z)=\sum_{k=0}^{\infty}a_k\frac{(1-|z_k|)^{M-\frac{1}{p}}\whv(z_k)^{-\frac{1}{p}}}{(1-\conz_kz)^M}, \quad z \in \D,
\end{equation*}
is analytic and satisfies $\|f\|_{A^p_{\nu}} \lesssim \|a_k\|_{\ell^p}$. Let $K>1$ and let $z_k=1-K^{-k}$, which certainly is a separated sequence and let $\{a_k\} \in \ell^p$. Denote by $r_k$ the $k$th Rademacher function. We will use the assumption and Khinchine's inequality to the function $f_x$ given by the previously mentioned separated sequence $\{z_k\}$ and the sequence $\{a_kr_k(x)\} \in \ell^p$ for every $x \in [0,1]$. We apply the inequality $\|f\|_{A^p_{\nu}} \gtrsim \int_{0}^{1}M^p_{\infty}(s,f)\whv(s)\,ds$, valid for every radial weight $\nu$, which follows from \cite[Theorem 2.16]{Pavlovic2014} and Fubini's theorem. The Dominated convergence theorem and Fubini's theorem then yield
\begin{equation*}
\begin{split}
\|C_{\mu,\om}\|_{A^p_{\nu}\to A^q_{\eta}}^q\|(a_k)\|_{\ell^p}^q
&\gtrsim \|C_{\mu,\om}\|_{A^p_{\nu}\to A^q_{\eta}}^q\|f_x\|_{A^p_{\nu}}^q\\
& \gtrsim \int_{0}^{1}\int_{0}^{1}|C_{\mu,\om}(f_x)(s)|^q\widehat{\eta}(s)\,ds\,dx\\
&=\int_{0}^{1}\int_{0}^{1}\left|\int_{0}^{1}\sum_{k=0}^{\infty}a_kr_k(x)\frac{(1-|z_k|)^{M-\frac{1}{p}}\whv(z_k)^{-\frac{1}{p}}}{(1-z_kts)^M}B^{\om}_t(s)\,d\mu(t)\right|^q\widehat{\eta}(s)\,ds\,dx\\
&=\int_{0}^{1}\int_{0}^{1}\left|\sum_{k=0}^{\infty}a_kr_k(x)(1-|z_k|)^{M-\frac{1}{p}}\whv(z_k)^{-\frac{1}{p}}\int_{0}^{1}\frac{B^{\om}_t(s)}{(1-z_kts)^M}\,d\mu(t)\right|^q\,dx\, \widehat{\eta}(s)\,ds\\
&\asymp \int_{0}^{1}\left(\sum_{k=0}^{\infty}|a_k|^2(1-|z_k|)^{2M-\frac{2}{p}}\whv(z_k)^{-\frac{2}{p}}\left(\int_{0}^{1}\frac{B^{\om}_t(s)}{(1-z_kts)^M}\,d\mu(t)\right)^2\right)^{\frac{q}{2}}\, \widehat{\eta}(s)\,ds.
\end{split}
\end{equation*}
By Lemma~\ref{Lemma:LowerBoundEq}, we have
\begin{equation*}
\begin{split}
&\quad \int_{0}^{1}\left(\sum_{k=0}^{\infty}|a_k|^2(1-|z_k|)^{2M-\frac{2}{p}}\whv(z_k)^{-\frac{2}{p}}\left(\int_{0}^{1}\frac{B^{\om}_t(s)}{(1-z_kts)^M}\,d\mu(t)\right)^2\right)^{\frac{q}{2}}\, \widehat{\eta}(s)\,ds\\
&\geq \sum_{m=0}^{\infty}\int_{z_m}^{z_{m+1}}|a_m|^q(1-|z_m|)^{qM-\frac{q}{p}}\whv(z_m)^{-\frac{q}{p}}\left(\int_{0}^{1}\frac{B^{\om}_t(s)}{(1-z_mts)^M}\,d\mu(t)\right)^q\widehat{\eta}(s)\,ds\\
&\gtrsim \sum_{m=0}^{\infty}\int_{z_m}^{z_{m+1}}|a_m|^q(1-|z_m|)^{qM-\frac{q}{p}}\whv(z_m)^{-\frac{q}{p}}\mu_{K^m-1}^q\frac{(K^m)^{qM+q}}{\om_{K^m}^q}\widehat{\eta}(s)\,ds\\
&\asymp 
\sum_{m=0}^{\infty}|a_m|^q (K^{m})^{\frac{q}{p}+q-1}\frac{\mu_{K^m-1}^q \eta_{K^m}}{\om_{K^m}^q\nu_{K^m}^{\frac{q}{p}}}\\
\end{split}
\end{equation*}
Using the fact that
$$\sup\left\{ \left(\sum_{n=0}^\infty |\lambda_n a_n|^q\right)^{1/q}: \sum_{n=0}^\infty|a_n|^p\le 1\right\}=\left( \sum_{n=0}^\infty |\lambda_n|^{\frac{qp}{p-q}}\right)^{\frac{p-q}{qp}},$$ we obtain
$$\left\|C_{\mu,\om}\right\|_{A^p_{\nu}\to A^q_{\eta}} \gtrsim \left(\sum_{m=0}^{\infty}\left(\frac{\mu_{K^m-1} K^m}{\om_{K^m}}\frac{\eta_{K^m}^\frac{1}{q}}{\nu_{K^m}^\frac{1}{q}}\right)^{\frac{qp}{p-q}}\frac{\nu_{K^m}}{K^m}\right)^{\frac{p-q}{qp}}.
$$
By standard estimates and Lemma~\ref{DoublingLemma}, we see that the previous estimate simplifies to the wanted form as
\begin{equation*}
\left(\sum_{m=0}^{\infty}\left(\frac{\mu_{K^m-1} K^m}{\om_{K^m}}\frac{\eta_{K^m}^\frac{1}{q}}{\nu_{K^m}^\frac{1}{q}}\right)^{\frac{qp}{p-q}}\frac{\nu_{K^m}}{K^m}\right)^{\frac{p-q}{qp}}
    \asymp \left(\sum_{n=0}^{\infty}\left(\frac{\mu_{n} (n+1)}{\om_{n}}\frac{\eta_{n}^\frac{1}{q}}{\nu_{n}^\frac{1}{q}}\right)^{\frac{qp}{p-q}}\frac{\nu_{n}}{(n+1)^2}\right)^{\frac{p-q}{qp}}.
\end{equation*}
\end{proof}
For the main theorem in the case of $q<p$, we can apply methods similar to those in the case of Hardy spaces.
\begin{theorem}\label{Thm: Bergmanq<poneweight}
    Let $0<q<p<\infty$, $\om \in \DD$, $\nu \in \DDD$, and let $\mu \in M^+([0,1))$. Then the following conditions are equivalent:
    \begin{enumerate}
        \item[\textup{(i)}] $C_{\mu,\om}: A^p_{\nu} \to A^q_{\nu}$ is bounded;
        \item[\textup{(ii)}]
        \begin{equation*}
        \|\mu_{\om}\|_{L^\frac{qp}{p-q}_{\widehat\nu}}=\left(\int_0^1 \left(\frac{\whm(r)}{\whw(r) (1-r)}\right)^{\frac{qp}{p-q}}\whv(r)\,dr\right)^{\frac{p-q}{qp}}+\whm(0)<\infty.
        \end{equation*}
        \item[\textup{(iii)}]
        \begin{equation*}
            \|\mu_{\om}\|_{\ell_{\frac{\nu_n}{(n+1)^2}}^\frac{qp}{p-q}}=\left(\sum_{n=0}^{\infty}\left(\frac{\mu_{n} (n+1)}{\om_{n}}\right)^{\frac{qp}{p-q}}\frac{\nu_{n}}{(n+1)^2}\right)^{\frac{p-q}{qp}}<\infty;
        \end{equation*}
        \item[\textup{(iv)}] $F_{\mu,\om} \in A^{\frac{qp}{p-q}}_\nu.$
    \end{enumerate}
    \begin{equation*}
    \end{equation*}
    Moreover, we have $\|C_{\mu,\om}\|_{A^p_{\nu} \to A^q_{\nu}} \asymp \|\mu_{\om}\|_{L^\frac{qp}{p-q}_{\widehat\nu}} \asymp \   \|\mu_{\om}\|_{\ell_{\frac{\nu_n}{(n+1)^2}}^\frac{qp}{p-q}}\asymp \|F_{\mu,\om}\|_{A^{\frac{qp}{p-q}}_{\nu}}$.
\end{theorem}

\begin{proof}
    We first show that \textup{(ii)} implies \textup{(i)}. As $C_{\mu,\om}(f)(z)\leq f^*(z)F_{\mu,\om}^+(z)$, by H\"older's inequality and Proposition~\ref{FHp}, we obtain
    \begin{equation*}
    \begin{split}
    \|C_{\mu,\om}(f)\|_{A^q_{\nu}}
    &\leq \|f^* F_{\mu,\om}^+\|_{L^q_{\nu}}\\
    &\leq \|f^* \|_{A^p_{\nu}}\|F_{\mu,\om}^+\|_{L^\frac{qp}{p-q}_{\nu}}\\
    &\asymp \|f^* \|_{A^p_{\nu}}{\|\mu_{\om}\|_{L^\frac{qp}{p-q}_{\widehat\nu}}}\\
    &\lesssim  \|f\|_{A^p_{\nu}}{\|\mu_{\om}\|_{L^\frac{qp}{p-q}_{\widehat\nu}}}
    \end{split}
    \end{equation*}
    where the last asymptotic inequality follows from the radial maximal theorem \cite[Theorem 7.1.4]{Pavlovic2004}, which yields the corresponding result for every weighted Bergman space induced by a radial weight. Lemma~\ref{q<ptestingbergman} shows that \textup{(i)} implies \textup{(iii)}, and the fact that \textup{(ii)}, \textup{(iii)} and \textup{(iv)} are equivalent follows from Proposition~\ref{FHp}. This concludes the proof.
\end{proof}

By combining these results, we can deduce Theorem~\ref{MainTheorem: Bergman}.

\emph{Proof of Theorem~\ref{MainTheorem: Bergman}.}
Combining Theorem~\ref{Bergmanp<q}, Theorem~\ref{Theorem:ApAq q<1} and Theorem~\ref{Thm: Bergmanq<poneweight}, we obtain Theorem~\ref{MainTheorem: Bergman} with the exception of equivalence between \textup{(i)}(d) and the other conditions when $q>1$. However, this follows from exactly the same method as was shown in Proposition~\ref{Prop: Fmuom sup norm space}.
\hfill$\square$

We now give the corresponding statement of Theorem~\ref{MainTheorem: Bergman} in the case of $C_{\mu,\beta}$.
\begin{theorem} \label{MainTheorem: Bergmanbeta}
     Let $0<p,q<\infty$, $0<\beta<\infty$, $\nu \in \DDD$, and let $\mu \in M^+([0,1))$. Then the following holds:
     \begin{enumerate}
        \item[\textup{(i)}] If $p\leq q$, then the following conditions are equivalent:
        \begin{enumerate}
            \item[\textup{(a)}] $C_{\mu,\beta}: A^p_{\nu} \to A^q_{\nu}$ is bounded;
            \item[\textup{(b)}] \begin{equation*}
                {\|\mu^{(p,q)}_{\beta,\nu}\|_{L^\infty}}=\sup_{0\leq r<1}\frac{\widehat{\mu}(r)}{\whv(r)^{\frac{1}{p}-\frac{1}{q}}(1-r)^{\beta+\frac{1}{p}-\frac{1}{q}}}<\infty;
            \end{equation*}
            \item[\textup{(c)}]
            \begin{equation*}
                 {\|\mu^{(p,q)}_{\beta,\nu}\|_{\ell^\infty}}=\sup_{n \in \N_0}\frac{\mu_n(n+1)^{\beta+\frac{1}{p}-\frac{1}{q}}}{\nu_n^{\frac{1}{p}-\frac{1}{q}}};
            \end{equation*}
            \item[\textup{(d)}] For every fixed $\gamma>0$, we have $D^{\gamma+\frac{1}{p}-\frac{1}{q}}D^{\nu^{\frac{1}{p}-\frac{1}{q}}}F_{\mu,\beta} \in X_{\gamma}$.
        \end{enumerate}
    Moreover, we have $\|C_{\mu,\beta}\|_{A^p_{\nu} \to A^q_{\nu}}\asymp {\|\mu^{(p,q)}_{\beta,\nu}\|_{L^\infty}} \asymp{\|\mu^{(p,q)}_{\beta,\nu}\|_{\ell^\infty}}
    \asymp \|D^{\gamma+\frac{1}{p}-\frac{1}{q}}D^{\nu^{\frac{1}{p}-\frac{1}{q}}}F_{\mu,\beta}\|_{X_{\gamma}}$.
    \item[\textup{(ii)}] If $q<p$, then the following conditions are equivalent:
    \begin{enumerate}
        \item[\textup{(a)}] $C_{\mu,\beta}: A^{p}_{\nu} \to A^q_{\nu}$ is bounded;
        \item[\textup{(b)}]
        \begin{equation*}
                 \|\mu_{\beta,\nu}\|_{L^\frac{qp}{p-q}}^\frac{qp}{p-q}=\int_0^1 \left(\frac{\whm(r)}{(1-r)^{\beta}}\right)^\frac{qp}{p-q}\whv(r)\,dr+\whm(0)^{\frac{qp}{p-q}}<\infty;
        \end{equation*}
        \item[\textup{(c)}]
        \begin{equation*}
\|\mu_{\beta,\nu}\|_{\ell^\frac{qp}{p-q}}^\frac{qp}{p-q}=\sum_{n=0}^{\infty}\left(\mu_{n} (n+1)^{\beta}\right)^{\frac{qp}{p-q}}\frac{\nu_{n}}{(n+1)^2} <\infty;
        \end{equation*}
        \item[\textup{(d)}] $F_{\mu,\beta} \in A^\frac{qp}{p-q}_{\nu}.$
    \end{enumerate}
\begin{equation*}
\end{equation*}
Moreover, we have $\|C_{\mu,\beta}\|_{A^p_{\nu}\to A^q_{\nu}} \asymp \|\mu_{\beta,\nu}\|_{L^\frac{qp}{p-q}} \asymp \|\mu_{\beta,\nu}\|_{\ell^\frac{qp}{p-q}} \asymp \|F_{\mu,\beta}\|_{A^\frac{qp}{p-q}_{\nu}}.$
    \end{enumerate}
\end{theorem}

Next, we consider the two-weight setting when $1<q<p<\infty$.

\begin{theorem}
    Let $1<q<p<\infty$, $\om \in \DD$, $\nu,\eta \in \DDD$, and let $\mu \in M^+([0,1))$. Then $C_{\mu,\om}: A^{p}_{\nu} \to A^q_{\eta}$ is bounded if and only if
\begin{equation*}
\|\mu_{\om,\nu,\eta}\|_{\ell_{\frac{\nu_n}{(n+1)^2}}^\frac{qp}{p-q}}=\left(\sum_{n=0}^{\infty}\left(\frac{\mu_{n} (n+1)}{\om_{n}}\frac{\eta_{n}^\frac{1}{q}}{\nu_{n}^\frac{1}{q}}\right)^{\frac{qp}{p-q}}\frac{\nu_{n}}{(n+1)^2}\right)^{\frac{p-q}{qp}} <\infty.
\end{equation*}
Moreover, we have
\begin{equation*}
    \|\mu_{\om,\nu,\eta}\|_{\ell_{\frac{\nu_n}{(n+1)^2}}^\frac{qp}{p-q}} \asymp \|C_{\mu,\om}\|_{A^p_{\nu}\to A^q_{\eta}}.
\end{equation*}
\end{theorem}
\begin{proof}
By Lemma~\ref{q<ptestingbergman}, we only need to show that $\|\mu_{\om,\nu,\eta}\|_{\ell^\frac{qp}{p-q}} <\infty$ is a sufficient condition with the correct norm estimate. The weight $V_{\om,\nu,p,q}$ defined by $V_{\om,\nu,p,q}(r)=\whw\left(r\right)^p\left(1-r\right)^{1+p-\frac{p}{q}}\nu(r)$ for all $0\leq r<1$ can be easily seen to belong to $\DDD$, see the proof of {Lemma~\ref{weightVpq} and} Theorem~\ref{Bergmanp<q} for a similar situation. Let $K \in \N$ be fixed large enough such that it satisfies \eqref{eq:ReverseDoublingDef} for $\nu$ and $V_{\om,\nu,p,q}$. Denote once again by $\{\widetilde{K}_n\}$ the sequence $\{0,1,K,K^2, \dots, K^{n-1},\dots\}$. We first use Lemma~\ref{decompositionlemma}, then apply Lemma~\ref{monotonepullout} and H\"older's inequality, which yield
\begin{equation}\label{formula}
\begin{split}
\|C_{\mu,\om}(f)\|_{A^q_{\eta}}^q &\asymp \sum_{n=0}^{\infty}\eta_{K^n}\|V_{n,K}*C_{\mu,\om}(f)\|_{H^q}^q\\
&\lesssim \sum_{n=0}^{\infty}\eta_{K^n}\mu_{\widetilde{K}_n}^q\|V_{n,K}*fB^{\om}\|_{H^q}^q\\
&=\sum_{n=0}^{\infty}\eta_{K^n}\mu_{\widetilde{K}_n}^q\frac{\om_{K^n}^q \nu_{K^n}^{\frac{q}{p}}(K^n)^{\frac{q}{p}+q-1}}{\om_{K^n}^q \nu_{K^n}^{\frac{q}{p}}(K^n)^{\frac{q}{p}+q-1}}\|V_{n,K}*fB^{\om}\|_{H^q}^q\\
&\leq \left(\sum_{n=0}^{\infty}\left(\frac{\mu_{\widetilde{K}_n}^q\left(K^{n}\right)^{\frac{q}{p}+q-1}\eta_{K^n}}{\om_{K^n}^q\nu_{K^n}^{\frac{q}{p}}}\right)^{\frac{p}{p-q}}\right)^{\frac{p-q}{p}}\\
&\quad \cdot
\left(\sum_{n=0}^{\infty}\nu_{K^n}\frac{\om_{K_n}^p}{\left(K^n\right)^{1+p-\frac{p}{q}}}\|V_{n,K}*fB^{\om}\|_{H^q}^p\right)^{\frac{q}{p}}\\
&\lesssim
\|\mu_{\om,\nu,\eta}\|_{\ell^\frac{qp}{p-q}}^q\left(\sum_{n=0}^{\infty}\nu_{K^n}\frac{\om_{K_n}^p}{\left(K^n\right)^{1+p-\frac{p}{q}}}\|V_{n,K}*fB^{\om}\|_{H^q}^p\right)^{\frac{q}{p}}
\end{split}
\end{equation}
The last asymptotic inequality follows easily by standard estimations and Lemma~\ref{DoublingLemma}. Here we note that as $q>1$, it is clear that $1+p-\frac{p}{q}>0$. As $\nu \in \DDD$, by the same arguments which were used on the proof of Theorem~\ref{Bergmanp<q} and remembering the definition of $V_{\om,\nu,p,q}$, we note that
\begin{equation*}
\begin{split}
\nu_{K^n}\frac{\om_{K^n}^p}{\left(K^n\right)^{1+p-\frac{p}{q}}}
&\asymp \whv\left(1-\frac{1}{K^n}\right)\whw\left(1-\frac{1}{K^n}\right)^p\left( 1-\left(1-\frac{1}{K^n}\right)\right)^{1+p-\frac{p}{q}}\\
&\asymp \int_{1-\frac{1}{K^n}}^1\whw\left(s\right)^p\left(1-s\right)^{1+p-\frac{p}{q}}\nu(s)\,ds\\
&=\widehat{V_{\om,\nu,p,q}}\left(1-\frac{1}{K^n}\right)
\asymp \left(V_{\om,\nu,p,q}\right)_{K^n}, \quad n \in \N.
\end{split}
\end{equation*}
Plugging this into the previous sum, using H\"older's inequality and the $L^p$-estimates for kernels shown in Lemma~\ref{Lemma: KernelLpEstimate}, we obtain
\begin{equation*}
\begin{split}
&\quad \sum_{n=0}^{\infty}\nu_{K^n}\frac{\om_{K_n}^p}{\left(K^n\right)^{1+p-\frac{p}{q}}}\|V_{n,K}*fB^{\om}\|_{H^q}^p\\
&\asymp \sum_{n=0}^{\infty}\left(V_{\om,\nu,p,q}\right)_{K^n}\|V_{n,K}*fB^{\om}\|_{H^q}^p\\
&\asymp \int_{0}^{1}M_q^p(r, fB^{\om})\whw(r)^p(1-r)^{1+p-\frac{p}{q}}\nu(r)r\,dr\\
&\lesssim \int_{0}^{1}M_p^p(r,f) \left(\int_{0}^{2\pi}|B^{\om}(re^{it})|^{\frac{qp}{p-q}}\,dt\right)^{\frac{p-q}{q}}\whw(r)^p(1-r)^{1+p-\frac{p}{q}}\nu(r)r\,dr\\
&\asymp \int_{0}^{1}M_p^p(r,f) \left(\int_{0}^{r}\frac{ds}{\whw(s)^{\frac{qp}{p-q}}(1-s)^{\frac{qp}{p-q}}}+1\right)^{\frac{p-q}{q}}\whw(r)^p(1-r)^{1+p-\frac{p}{q}}\nu(r)r\,dr\\
&\asymp \int_{0}^{1}M_p^p(r,f) \left(\frac{1}{\whw(r)^{\frac{qp}{p-q}}(1-r)^{\frac{qp}{p-q}-1}}\right)^{\frac{p-q}{q}}\whw(r)^p(1-r)^{1+p-\frac{p}{q}}\nu(r)r\,dr\\
&=\int_{0}^{1}M_p^p(r,f) \nu(r)r\,dr=\|f\|_{A^p_{\nu}}^p.
\end{split}
\end{equation*}
By combining this with \eqref{formula}, we obtain the wanted inequality
\begin{equation*}
    \|C_{\mu,\om}\|_{A^q_{\eta}}^q \lesssim \|\mu_{\om,\nu,\eta}\|_{\ell^\frac{qp}{p-q}}^q\|f\|_{A^p_{\nu}}^q,
\end{equation*}
which concludes the proof.
\end{proof}

\section{Other spaces}\label{Sec: Other}

We begin this section by considering $C_{\mu,\om}: H^{\infty} \to H^p$ for $0<p<\infty$. It turns out that the characterizing condition is in a sense the limit version of the characterizing condition of boundedness of $C_{\mu,\om}: H^{s} \to H^p$, when $s \to \infty$.

\begin{theorem}
    Let $0<p<\infty$, $\om \in \DD$, and let $\mu \in M^+([0,1))$. Then the following are equivalent:
    \begin{enumerate}
        \item[\textup{(i)}] $C_{\mu,\om}: H^{\infty} \to H^p$ is bounded;
        \item[\textup{(ii)}]
        \begin{equation*}
        \|\mu_{\om}\|_{L^p}^p=\int_{0}^{1}\left( \frac{\whm(r)}{\whw(r)(1-r)}\right)^p\,dr+\whm(0)^p<\infty;
        \end{equation*}
        \item[\textup{(iii)}]
        \begin{equation*}
\|\mu_{\om}\|_{\ell^p}^p=\sum_{n=0}^{\infty}\left(\frac{\mu_n(n+1)}{\om_{n}}\right)^p\frac{1}{(n+1)^2}<\infty;
        \end{equation*}
        \item[\textup{(iv)}] $F_{\mu,\om} \in H^p$.
    \end{enumerate}
    Moreover, we have $\|C_{\mu,\om}\|_{H^\infty \to H^p} \asymp \|\mu_{\om}\|_{\ell^p} \asymp \|\mu_{\om}\|_{L^p} \asymp \|F_{\mu,\om}\|_{H^p}$.
\end{theorem}
\begin{proof}
Assume first that \textup{(ii)} holds. By Proposition~\ref{FHp}, we obtain
\begin{equation*}
\begin{split}
    \|C_{\mu,\om}(f)\|_{H^p}^p&\leq \|f\|_{\infty}^p\|F_{\mu,\om}^+\|_{L^{p}(\T)}^p\\
    &\asymp \|f\|_{\infty}^p\|\mu_{\om}\|_{L^p}^p.
\end{split}
\end{equation*}
Assume next \textup{(i)} and we show that it implies \textup{(iv)}. We let $f \equiv 1$. Thus, we obtain
\begin{equation*}
\begin{split}
    \|C_{\mu,\om}\|_{H^\infty \to H^p}^p
    \geq \|C_{\mu,\om}(1)\|_{H^p}^p
=\|F_{\mu,\om}\|_{H^p}^p
\end{split}
\end{equation*}
which concludes the proof, as \textup{(ii)}, \textup{(iii)} and \textup{(iv)} are equivalent by Proposition~\ref{FHp}.
\end{proof}

Next we characterize the boundedness of $C_{\mu,\om}$ acting on $H^{\infty}$ and Korenblum spaces.

\begin{proposition} \label{main1}Let $0<\beta,\gamma<\infty$, $\om\in\DD$, and let $\mu\in M^+([0,1))$. Then the following statements hold:
\begin{enumerate}
    \item[\textup{(i)}] $C_{\mu, \om}: H^{\infty} \to H^{\infty}$ is bounded if and only if  $\sum_{n=0}^\infty \frac{\mu_n}{\om_n}<\infty$;
    \item[\textup{(ii)}] $C_{\mu, \om}: H^\infty \to X_\gamma$ is bounded if and only if $\sup_{n\in\N_0}\frac{\mu_n}{\om_n (n+1)^{\gamma-1}}<\infty$;
    \item[\textup{(iii)}] $C_{\mu, \om}: X_\beta \to X_\gamma$ is bounded if and only if $\sup_{n\in\N_0}\frac{\mu_n}{\om_n (n+1)^{\gamma-\beta-1}}<\infty$.
\end{enumerate}
\end{proposition}
\begin{proof}

For the proof of \textup{(i)}, observe that $$F_{\mu,\om}(z)=C_{\mu, \om} (1)(z)=\sum_{n=0}^\infty \frac{\mu_n}{2\om_{2n+1}}z^{n}.$$  If $C_{\mu, \om}: H^{\infty} \to H^{\infty}$ is bounded then $F_{\mu,\om}\in H^\infty$, which gives $\sum_{n=0}^\infty \frac{\mu_n}{\om_n}<\infty$ by Lemma~\ref{DoublingLemma}. Conversely, if  $\sum_{n=0}^\infty \frac{\mu_n}{\om_n}<\infty$ and $f\in H^\infty$, then
$$M_\infty(r,C_{\mu, \om} (f))\le\|f\|_{H^\infty} \int_0^1 M_\infty(rt,B^\om)\,d\mu(t)\asymp\|f\|_{H^\infty}\sum_{n=0}^\infty \frac{\mu_n}{\om_n}r^n, \quad 0\leq r<1,$$
Which yields {the boundedness of $C_{\mu, \om}: H^{\infty} \to H^{\infty}$.}

For \textup{(ii)}, assume first that $C_{\mu, \om}: H^\infty \to H_\gamma$ is bounded. Therefore $F_{\mu,\om}\in H_\gamma$, that is to say
$$\sum_{n=0}^\infty \frac{\mu_n}{\om_n}r^{n}\lesssim \frac{1}{(1-r)^\gamma}, \quad 0 \leq r<1.$$
Letting $r_N=1-\frac{1}{N+1}$ for every $N \in \N_0$ yields
$$\sum_{k=0}^{N} \frac{\mu_k}{\om_k}\lesssim (N+1)^\gamma, \quad N \in \N_0.$$
In particular $\frac{\mu_{2N}}{\om_{N}}(N+1)\le \sum_{k=N}^{2N} \frac{\mu_k}{\om_k}\lesssim (N+1)^\gamma$ which, due to Lemma~\ref{DoublingLemma} and $\mu_n$ being decreasing, yields
$\sup_{n\in\N_0}\frac{\mu_n}{\om_n (n+1)^{\gamma-1}}<\infty$. Conversely, if $\frac{\mu_n}{\om_n}\lesssim (n+1)^{\gamma-1}$ for every $n \in \N_0$ and $f\in H^\infty$, then
\begin{equation*}
\begin{split}
    M_\infty(r,C_{\mu, \om}(f))&\lesssim \|f\|_{H^\infty}\sum_{n=0}^\infty \frac{\mu_n}{\om_n}r^n\lesssim \|f\|_{H^\infty}\sum_{n=0}^\infty(n+1)^{\gamma-1}r^n\\
    &\asymp\frac{\|f\|_{H^\infty}}{(1-r)^\gamma}, \quad 0 \leq r<1,
\end{split}
\end{equation*}
which shows that $C_{\mu, \om}: H^\infty \to H_\gamma$ is bounded.

To show \textup{(iii)}, for $\beta>0$, we select $f_\beta(z)=\frac{1}{(1-z)^\beta}=\sum_{n=0}^\infty \gamma_n(\beta)z^n$ which  belongs to $X_\beta$ with $\|f_\beta\|_{X_{\beta}}=1$.
It is well known that $\gamma_n(\beta)\asymp (n+1)^{\beta-1}$, which can be seen for example by Stirling's approximation of the Gamma function. Hence
$$C_{\mu, \om} (f_\beta)(z)=\sum_{n=0}^\infty \mu_n\left(\sum_{k=0}^n \frac{\gamma_k(\beta)}{2\om_{2(n-k)+1}}\right)z^{n}.$$
Assuming that  $C_{\mu, \om}: X_{\beta} \to X_\gamma$ is bounded, we conclude that
\be\label{condition1}\sum_{n=0}^\infty \mu_n\left(\sum_{k=0}^n \frac{(k+1)^{\beta-1}}{2\om_{2(n-k)+1}}\right)r^{n}\lesssim\frac{1}{(1-r)^\gamma}, \quad 0 \leq r<1.\ee
In particular, using Lemma~\ref{DoublingLemma} we have that
\begin{equation}\label{Eq:twosidedsumestimate}
    \sum_{k=0}^n \frac{(k+1)^{\beta-1}}{2\om_{2(n-k)+1}}\asymp \frac{(n+1)^\beta}{\om_n}, \quad n \in \N,
\end{equation}
and therefore the condition (\ref{condition1})
yields  $$\sum_{k=0}^N \frac{\mu_k(k+1)^{\beta}}{\om_k}\lesssim(N+1)^\gamma, \quad N\in\N_0,$$ which again implies that
$\frac{\mu_N (N+1)^{\beta+1}}{\om_N}\lesssim (N+1)^\gamma$ for all $N\in\N_0$, by similar methods as was shown in the proof of \textup{(ii)}, as we wanted to show.
Conversely, assume that  $\frac{\mu_n(n+1)^{\beta}}{\om_n} \lesssim (n+1)^{\gamma-1}$ for every $n\in\N_0$ and let $f\in X_\beta$.
Then we have
\begin{equation*}
\begin{split}
M_\infty(r,C_{\mu, \om} f)
&\le \|f\|_{X_\beta}\int_0^1 \frac{M_\infty(rt,B^\om)}{(1-rt)^\beta}d\mu(t)
\asymp \|f\|_{X_\beta} \sum_{n=0}^\infty \mu_n\left(\sum_{k=0}^n \frac{\gamma_k(\beta)}{\om_{n-k}}\right)r^{n}\\
&\lesssim \|f\|_{X_\beta} \sum_{n=0}^\infty \frac{\mu_n (n+1)^\beta}{\om_n}r^{n}
\lesssim \|f\|_{X_\beta} \sum_{n=0}^\infty  (n+1)^{\gamma-1}r^{n}\\
&\asymp \frac{\|f\|_{X_\beta}}{(1-r)^\gamma}, \quad 0 \leq r<1,
\end{split}
\end{equation*}
which concludes the proof.
\end{proof}
The next result characterizes the boundedness of $C_{\mu,\om}$ acting from Korenblum space to a Hardy space.

\begin{proposition}\label{XbetaHp} Let ${1<p<\infty}$, $0<\beta<\infty$, $\om\in\DD$, and let $\mu\in M^+([0,1))$. Then $C_{\mu, \om}: X_\beta \to H^p$ is bounded if and only if 
\begin{equation}\label{HpXbcondition}
    \int_0^1 \left(\frac{\whm(r) }{\whw(r)(1-r)^{(\beta+1)}}\right)^p\,dr+\whm(0)^p <\infty.
\end{equation}
\end{proposition}
\begin{proof}
Assume first that $C_{\mu, \om}: X_\beta \to H^p$. Since $f_\beta(z)=\frac{1}{(1-z)^\beta}\in X_\beta$ with $\|f_\beta\|_{X_{\beta}}=1$, then
\be F_{\beta,\mu,\om}(z)=C_{\mu, \om}(f_\beta)(z)= \sum_{n=0}^\infty \gamma_n z^n=\int_0^1 \frac{B^\om_t(z)}{(1-tz)^\beta}d\mu(t) \in H^p.\ee
As shown in Proposition \ref{main1}\textup{(iii)}, we have that $ \frac{B^\om(z)}{(1-z)^\beta}=\sum_{n=0}^\infty c_n z^n$, with $c_n\asymp \frac{ (n+1)^{\beta}}{\om_n}$. We use the Fejer-Riesz inequality, see \cite[(5.15)]{Pavlovic2004}, Lemma~\ref{DoublingLemma} and standard estimates to obtain
\begin{equation*}
\begin{split}
\|F_{\beta,\mu,\om}\|^p_{H^p}
&\gtrsim  \int_0^1\left(\int_0^1\frac{B^\om_t(s)}{(1-st)^\beta }d\mu(t)\right)^p ds\\
&\asymp \int_0^1 \left(\sum_{n=0}^\infty \frac{(n+1)^\beta \mu_n}{\om_n}s^n\right)^pds\\
&\gtrsim \sum_{k=1}^\infty \int_{1-\frac{1}{k}}^{1-\frac{1}{k+1}} \left(\sum_{n=\lfloor k/2 \rfloor}^k\frac{(n+1)^\beta \mu_n}{\om_n}s^n\right)^pds+\mu_0^p\\
&\gtrsim \sum_{k=1}^\infty\left(\frac{\mu_k(k+1)^{\beta+1}}{\om_k}\right)^p\frac{1}{(k+1)^2}+\mu_0^p\\
&\gtrsim  \sum_{k=0}^\infty \left(\frac{\mu_k(k+1)^{\beta+1}}{\om_k}\right)^p \frac{1}{(k+1)^2},
\end{split}
\end{equation*}
which yields the assertion by \eqref{Eq:inequalityomnumu}, when applying in the case $a=c=p$, $b=p(\beta+1)$ and $d=0$.

Assume next that \eqref{HpXbcondition} holds. Let  $\mu_{\beta}$ be the measure defined by $d\mu_{\beta}(t)=(1-t)^{-\beta}d\mu(t)$.
We first observe that the assumption yields
\begin{equation} \label{estima}
\frac{\widehat\mu(r) }{(1-r)^{\beta+1}} \lesssim \frac{\widehat\om(r)}{(1-r)^{1/p}}, \quad 0\le r<1.
\end{equation}
Indeed, we may assume that $1/2<r<1$ and use
$$\frac{\widehat\mu(r) (1-r)^{1/p}}{\widehat\om(r)(1-r)^{\beta+1}} \le\left(\int_{2r-1}^r \left(\frac{\widehat\mu(t)}{\widehat\om(t)(1-t)^{\beta+1}} \right)^p dt\right)^{1/p}<\infty.$$
 In particular, $\mu_{\beta}$ is a finite measure with
\begin{equation}\label{hatmubeta}
\widehat{\mu_\beta}(r)\lesssim \frac{\widehat\mu(r)}{(1-r)^\beta}+ \int_r^1 \frac{\widehat\mu(s)}{(1-s)^{\beta+1}}, \quad 0\leq r<1.
\end{equation}
Indeed,we have
\begin{equation*}
 \int_{0}^{1}\frac{d\mu(t)}{(1-t)^{\beta}} 
    \asymp \int_{0}^{1}\frac{\widehat\mu(s)}{(1-s)^{\beta+1}}ds+\whm(0) 
\le \widehat\om(0) \int_{0}^{1}\frac{ds}{(1-s)^{1/p}}+\whm(0) <\infty,
\end{equation*}
and 
\begin{equation*}
\begin{split}
\int_{r}^{1}\frac{d\mu(s)}{(1-s)^{\beta}} 
&\asymp \int_{r}^{1}\int_{0}^{s}\frac{dx}{(1-x)^{\beta+1}}\,d\mu(s)+ \whm(r)\\
&=\int_{0}^{1}\frac{\whm(\max\{x,r\})}{(1-x)^{\beta+1}}\,dx+\whm(r)\\
&=\int_{0}^{r}\frac{\whm(r)}{(1-x)^{\beta+1}}\,dx+\int_{r}^{1}\frac{\whm(x)}{(1-x)^{\beta+1}}\,dx+\whm(r)\\
&\lesssim \frac{\whm(r)}{(1-r)^{\beta}}+\int_{r}^{1}\frac{\whm(x)}{(1-x)^{\beta+1}}\,dx, \quad 0\leq r<1.
\end{split}
\end{equation*}
 For $f\in X_\beta$ we have
\begin{equation*}
\begin{split}
    M_p^p(r,C_{\mu,\om}(f)) 
    &\lesssim \int_{0}^{2\pi}\left(\int_{0}^{1}|f(tre^{i\t})||B_t^{\om}(re^{i\t})|\,d\mu(t)\right)^p\,d\t\\
    &\leq \|f\|_{X_{\beta}}^p\int_{0}^{2\pi}\left(\int_{0}^{1}\frac{|B_t^{\om}(re^{i\t})|}{(1-tr)^{\beta}}\,d\mu(t)\right)^p\,d\t\\
    &\leq \|f\|_{X_{\beta}}^p\int_{0}^{2\pi}\left(\int_{0}^{1}|B_t^{\om}(re^{i\t})|\,d\mu_{\beta}(t)\right)^p\,d\t, \quad 0\leq r<1.\\
\end{split}
\end{equation*}
Therefore, we can apply Proposition~\ref{FHpr} and \eqref{hatmubeta}, which yield
\begin{equation*}\label{Eq:Mubetaeq}
\begin{split}
    \int_{0}^{2\pi}\left(\int_{0}^{1}|B_t^{\om}(re^{i\t})|\,d\mu_{\beta}(t)\right)^p
    &\lesssim \int_{0}^{1}\left(\frac{\widehat{\mu_{\beta}}(t)}{\whw(t)(1-t)}\right)^p\,dt+\widehat{\mu_{\beta}}(0)^p\\
 &\lesssim \int_{0}^{1}\left(\frac{\whm(t)}{\widehat\om(t)(1-t)^{\beta+1}}\right)^p\,dt\\
&\quad+ \int_{0}^{1}\left(\int_{t}^{1}\frac{\whm(x)}{(1-x)^{\beta+1}}\,dx\right)^p\frac{dt}{\whw(t)^p(1-t)^p}\\
&
\quad +\widehat{\mu_{\beta}}(0)^p.
\end{split}
\end{equation*}

 For the second term, we denote $\beta+1=\alpha+\gamma$, where {$0<\gamma<1/p'$} and apply H\"older's inequality together with Fubini's theorem to obtain
\begin{equation*}
\begin{split}
    &\quad \int_{0}^{1}\left(\int_{t}^{1}\frac{\whm(x)}{(1-x)^{\beta+1}}\,dx\right)^p\frac{dt}{\whw(t)^p(1-t)^p}\\
 &\le \int_{0}^{1}\left(\int_{t}^{1}\frac{\whm(x)}{\widehat\om(x)(1-x)^{\beta+1}}\,dx\right)^p\frac{dt}{(1-t)^p}\\
    &\leq \int_{0}^{1}\left(\int_{t}^{1}\left(\frac{\whm(x)}{\widehat\om(x)(1-x)^{\alpha}}\right)^p\,dx\right)\left(\int_{t}^{1}\frac{dx}{(1-x)^{p'\gamma }}\right)^{p-1}\frac{dt}{(1-t)^p}\,\\
    &\asymp \int_{0}^{1}\left(\int_{t}^{1}\left(\frac{\whm(x)}{\widehat\om(x)(1-x)^{\alpha}}\right)^p\,dx\right)\frac{dt}{(1-t)^{p\gamma+1}}\,\\
    &=\int_{0}^{1}\left(\frac{\whm(x)}{\widehat\om(x)(1-x)^{\alpha}}\right)^p\left(\int_{0}^{x}\frac{dt}{(1-t)^{p\gamma+1}}\,\right)\,dx\\
    &\lesssim \int_{0}^{1}\left(\frac{\whm(x)}{\whw(x)(1-x)^{\alpha}}\right)^p\frac{dx}{(1-x)^{p\gamma}}\,\\
    &=\int_0^1 \left(\frac{\whm(x)}{\whw(x)(1-x)^{(\beta+1)}}\right)^p\,dx,
\end{split}
\end{equation*}
which concludes the proof.
\end{proof}

Next we study the boundedness of $C_{\mu,\om}$ acting from $H^{\infty}$ and $X_{\beta}$ to the weighted Hardy spaces. The condition is perhaps surprisingly a Carleson type condition depending on the indices inducing the spaces.

\begin{proposition} \label{XbetaHpgamma} Let $0< p<\infty$, $0<\beta,\gamma<\infty$, $\om\in\DD$, and let $\mu\in M^+([0,1))$.
Then the following statements hold:
\begin{enumerate}
    \item[\textup{(i)}] $C_{\mu, \om}: H^{\infty}\to H^p_\gamma$ is bounded if and only if $\sup_{n\in\N_0}\frac{\mu_n}{\om_n (n+1)^{\gamma +\frac{1}{p}-1}}<\infty$.
    \item[\textup{(ii)}] {Let $1<p<\infty$. Then} $C_{\mu, \om}: X_\beta \to H^p_\gamma$ is bounded if and only if $\sup_{n\in\N_0}\frac{\mu_n}{\om_n (n+1)^{\gamma-\beta +\frac{1}{p}-1}}<\infty$.
\end{enumerate}
\end{proposition}
\begin{proof}
Assume first that $C_{\mu, \om}: H^{\infty}\to H^p_\gamma$ is bounded. By choosing $f \equiv 1$, we obtain
\be (1-r)^{\frac{1}{p}}\sum_{n=0}^\infty  \frac{\mu_n}{\om_n}r^{n}\lesssim M_p\left(\frac{1+r}{2},F_{\mu,\om}\right)\lesssim\frac{\|C_{\mu,\om}\|_{H^{\infty} \to H^p_\gamma}}{(1-r)^\gamma}, \quad 0\leq r<1,\ee
From following the reasoning done on the proof of Proposition~\ref{main1}, we obtain $\|C_{\mu,\om}\|_{H^{\infty} \to H^p_\gamma} \gtrsim {\sup_{n\in\N_0}\frac{\mu_n}{\om_n (n+1)^{\gamma +\frac{1}{p}-1}}}$. For the other direction, by Lemma~\ref{DoublingLemma} and the proof of Lemma~\ref{Momentvstail}, it is easy to see that \begin{equation*}
    \frac{\whm(r)(1-r)^{\gamma+\frac{1}{p}-1}}{\whw(r)} \lesssim \sup_{n\in\N_0}\frac{\mu_n}{\om_n (n+1)^{\gamma +\frac{1}{p}-1}}, \quad 0\leq r<1.
\end{equation*}
Therefore, by Proposition~\ref{FHpr}, we obtain
\begin{equation*}
\begin{split}
    M_p(r,C_{\mu,\om}(f)) &\leq \|f\|_{H^{\infty}}M_p(r,F_{\mu,\om}^+)\\
    &\lesssim \|f\|_{H^{\infty}}\left(\int_{0}^{r}\left(\frac{\whm(t)}{\whw(t)(1-t)}\right)^p\,dt+{\whm(0)^p}\right)^{\frac{1}{p}}\\
    &\lesssim \|f\|_{H^{\infty}}\left(\int_{0}^{r}\frac{dt}{(1-t)^{\gamma p+1}}+1\right)^{\frac{1}{p}}\\
    &\asymp \frac{\|f\|_{H^{\infty}}}{(1-r)^{\gamma}}, \quad 0\leq r<1,
\end{split}
\end{equation*}
which concludes the proof of \textup{(i)}. 

We proceed to proof of \textup{(ii)}. For $\beta>0$, we select $f_\beta(z)=\frac{1}{(1-z)^\beta}=\sum_{n=0}^\infty \gamma_n(\beta)z^n$ which belongs to $X_\beta$ with $\|f_\beta\|_{X_{\beta}}=1$.
Hence
$$F_{\beta,\mu,\om}(z)=\sum_{n=0}^\infty \mu_n\left(\sum_{k=0}^n \frac{\gamma_k(\beta)}{2\om_{2(n-k)+1}}\right)z^{n}\in H^p_\gamma.$$
Assuming that  $C_{\mu, \om}: X_\beta \to H^p_\gamma$ is bounded and using Lemma~\ref{DoublingLemma}, we conclude that
\be (1-r)^{\frac{1}{p}}\sum_{n=0}^\infty \mu_n \frac{(n+1)^\beta}{\om_n}r^{n}\lesssim M_p\left(\frac{1+r}{2},F_{\beta,\mu,\om}\right)\lesssim \frac{\|C_{\mu,\om}\|_{X_\beta \to H^p_\gamma}}{(1-r)^\gamma}, \quad 0\leq r<1,\ee
where we have used \eqref{Eq:twosidedsumestimate}. This implies that
$$\sum_{k=0}^n \frac{\mu_k(k+1)^{\beta}}{\om_k}\lesssim \|C_{\mu,\om}\|_{X_\beta \to H^p_\gamma}(n+1)^{\gamma+\frac{1}{p}}, \quad n \in \N_0,$$ which again implies that
$\frac{\mu_n (n+1)^{\beta+1}}{\om_n}\lesssim \|C_{\mu,\om}\|_{X_\beta \to H^p_\gamma}(n+1)^{\gamma+\frac{1}{p}}$ as we wanted to show.
{Conversely, assume that $\frac{\mu_n(n+1)^{\beta-\frac{1}{p}}}{\om_n}\lesssim  (n+1)^{\gamma-1}$ holds for every $n\in\N_0$ and let $f\in X_\beta$. For~$1<p<\infty$,
we apply Minkowski's inequality, Lemma~\ref{Lemma: KernelLpEstimate}, Lemma~\ref{DoublingLemma}, and Lemma~\ref{Twoweightsumestimate} to obtain
\begin{equation*}
\begin{split}
M_p(r,C_{\mu, \om} (f))
&\le \|f\|_{X_\beta}\int_0^1 \frac{M_p(B^\om,rt)}{(1-rt)^\beta}d\mu(t)\\
&\asymp \|f\|_{X_\beta}\int_0^1 \frac{d\mu(t)}{\whw(rt)(1-rt)^{\beta+ \frac{1}{p'}}}\\
&\asymp \|f\|_{X_\beta}\int_0^1 \left(\sum_{n=0}^\infty \frac{(n+1)^{\beta-\frac{1}{p}}}{\om_n} r^nt^n\right)d\mu(t)\\
&= \|f\|_{X_\beta} \sum_{n=0}^\infty \frac{\mu_n (n+1)^{\beta-\frac{1}{p}}}{\om_n}r^{n}\\
&\lesssim \|f\|_{X_\beta} \sum_{n=0}^\infty  (n+1)^{\gamma-1}r^{n}\\
&\asymp \frac{\|f\|_{X_\beta}}{(1-r)^\gamma}, \quad 0\leq r<1.
\end{split}
\end{equation*}
which concludes the proof.}

\end{proof}

Finally, we conclude the paper by the proof of Theorem~\ref{Thm: HinftytoBMOABloch}, characterizing the action of $C_{\mu,\om}$ from $H^{\infty}$ to $BMOA$ and the Bloch space, extending earlier results shown in \cite{GalanoGirelaMerchCesaro} and \cite{BlascoHardy}. We prove it by showing that $C_{\mu,\om}: H^{\infty} \to \bigcup_{1<p<\infty}\Lambda^p_{\frac{1}{p}}$ is bounded, where
$\Lambda^p_{\alpha}$ consists of $f\in \H(\D)$ satisfying
\begin{equation*}
    \|f\|_{\Lambda^p_{\alpha}}=\sup_{0\leq r<1}M_p(r,f')(1-r)^{1-\alpha}+|f(0)|,
\end{equation*} where $1<p<\infty$ and $0<\alpha<\infty$. We also note the well known fact
\begin{equation*}
    \bigcup_{1<p<\infty}\Lambda^p_{\frac{1}{p}} \subset BMOA \subset \B,
\end{equation*}
where the embeddings are continuous, see for instance \cite[Theorem 2.5]{Shapiro} and \cite[Corollary 5.2]{GirelaBMOA}. \newline \noindent
\emph{Proof of Theorem~\ref{Thm: HinftytoBMOABloch}.}
It is clear by Lemma~\ref{Momentvstail} that \textup{(iii)} and \textup{(iv)} are equivalent. Assume next that \textup{(iv)} holds, so $\|\mu_{\om}\|_{\ell^{\infty}}=\sup_{n\in \N_0}\frac{\mu_n(n+1)}{\om_n}<\infty.$ Let $V=V_{\om}$ be defined with $V_{\om}=2\int_{r}^{1}\om(t)t\,dt$ for $0\leq r<1$. By an elementary calculation, it is easy to see that $(k+1)V_{2k+1}=\om_{2k+3}$, from where it follows that $(B^{\om}_t(z))'=tB^{V}_t(z)$. By Lemma~\ref{DoublingLemma} it is also easy to see that $V \in \DD$. We also denote $d\mu^1(t)=td\mu(t)$. Therefore we obtain
\begin{equation*}
    (C_{\mu,\om}(f))'=C_{\mu^1,\om}(f')+C_{\mu^1,V}(f).
\end{equation*}
For $f \in H^{\infty}$, by Schwarz-Pick lemma, we have $f' \in X_1$. Therefore, we only need to show that $C_{\mu^1,\om}: X_1 \to H^p_{\frac{1}{p'}}$ and $C_{\mu^1,V}: H^{\infty} \to H^p_{\frac{1}{p'}}$ are bounded. However, this follows immediately from Proposition~\ref{XbetaHpgamma}\textup{(i)} and \textup{(ii)}. Thus by the embedding $\bigcup_{1<p<\infty}\Lambda^p_{\frac{1}{p}} \subset BMOA$, we obtain
\begin{equation*}
    \|C_{\mu,\om}(f)\|_{BMOA} \lesssim \|\mu_{\om}\|_{\ell^{\infty}}\|f\|_{H^\infty}.
\end{equation*}
This shows that \textup{(iv)} implies \textup{(i)}. As $\|f\|_{\B} \lesssim \|f\|_{BMOA}$, it is clear that \textup{(i)} implies \textup{(ii)}. To conclude the proof, it is enough to show that boundedness of $C_{\mu,\om}: H^{\infty} \to \B$ implies $\|\mu_{\om}\|_{\ell^{\infty}}=\sup_{n\in \N_0}\frac{\mu_n(n+1)}{\om_n}<\infty.$ This follows easily by testing with $f \equiv 1$ and calculations shown for example in Proposition~\ref{main1}. This finishes the proof.
\hfill$\square$


\begin{thebibliography}{99}
\bibitem{Andersen}          K.~F. Andersen, Ces\`aro averaging operators on Hardy spaces, Proc. Roy. Soc. Edinburgh Sect. A 126 (1996), no.~3, 617--624.
\bibitem{Bloch}         J.~M. Anderson, J.~G. Clunie and C. Pommerenke, On Bloch functions and normal functions, J. Reine Angew. Math. 270 (1974), 12--37.
\bibitem{Arsenovic} M. Arsenovi\'c, Embedding derivatives of ${\mathscr M}$-harmonic functions into $L^p$-spaces, Rocky Mountain J. Math. 29 (1999), no.~1, 61--76.
\bibitem{Bao}           G. Bao, L. Tian and H. Wulan, Ces\`aro type operators from some spaces of analytic functions to a mean Lipschitz space, J. Math. Anal. Appl. 554 (2026), no.~1, Paper No. 129903, 20 pp.
\bibitem{Bellavita}    C. Bellavita, G. Nikolaidis, \'A.~M. Moreno and J.~\'A. Pel\'aez, Fractional Volterra-type operator induced by radial weight acting on Hardy space, Math. Z. 312 (2026), no.~2, Paper No. 49, 38 pp.
\bibitem{BlascoHardy}       O. Blasco, Ces\`aro-type operators on Hardy spaces, J. Math. Anal. Appl. 529 (2024), no.~2, Paper No. 127017, 26 pp.
\bibitem{BlascoDirichlet}   O. Blasco, Generalized Ces\`aro operators on weighted Dirichlet spaces, J. Math. Anal. Appl. 540 (2024), no.~1, Paper No. 128627, 21 pp.
\bibitem{BlascoMas}     O. Blasco and A. Mas, Ces\`aro-type operators on mixed norm spaces, Trans. Amer. Math. Soc. 379 (2026), no.~7, 4713--4736.
\bibitem{Shapiro}           P.~S. Bourdon, J.~H. Shapiro and W.~T. Sledd, Fourier series, mean Lipschitz spaces, and bounded mean oscillation, in nalysis at Urbana, Vol.\ I (Urbana, IL, 1986--1987), 81--110, London Math. Soc. Lecture Note Ser., 137, Cambridge Univ. Press, Cambridge.
\bibitem{BHS}           A. Brown, P.~R. Halmos and A.~L. Shields, Ces\`aro operators, Acta Sci. Math. (Szeged) 26 (1965), 125--137.
\bibitem{Cesaro}            E. Ces\`aro. Sur la multiplication de s\'eries. Bull. Sci. Math., 14:114-120, 1890.
\bibitem{Duren}             P.~Duren,
                            Theory of $H^p$ Spaces, Academic Press, 1970.
\bibitem{DurenSchuster}         P.~L. Duren and A.~P. Schuster,
                        Bergman spaces, Mathematical Surveys and Monographs, 100, Amer. Math. Soc., Providence, RI, 2004.
\bibitem{GalanoGirelaMerchCesaro}   P. Galanopoulos, D. Girela and N. Merch\'an, Ces\`aro-like operators acting on spaces of analytic functions, Anal. Math. Phys. 12 (2022), no.~2, Paper No. 51, 29 pp.
\bibitem{GalanoGirelaMerchCesaro2}  P. Galanopoulos, D. Girela and N. Merch\'an, Ces\`aro-type operators associated with Borel measures on the unit disc acting on some Hilbert spaces of analytic functions, J. Math. Anal. Appl. 526 (2023), no.~2, Paper No. 127287, 13 pp.
\bibitem{GalaSisZhao}          P. Galanopoulos, A.~G. Siskakis and R. Zhao, Weighted Ces\`aro type operators between weighted Bergman spaces, Bull. Sci. Math. 202 (2025), Paper No. 103622, 17 pp.
\bibitem{Garnett}           J.~B. Garnett, Bounded analytic functions, revised first edition,
Graduate Texts in Mathematics, 236, Springer, New York, 2007.
\bibitem{GirelaBMOA}        D. Girela, Analytic functions of bounded mean oscillation, in Complex function spaces (Mekrij\"arvi, 1999), 61--170, Univ. Joensuu Dept. Math. Rep. Ser., 4, Univ. Joensuu, Joensuu.
\bibitem{Grafakos}          L. Grafakos,                                Classical Fourier analysis, third edition,
Graduate Texts in Mathematics, 249, Springer, New York, 2014.
\bibitem{Hardy}             G.~H. Hardy, Note on a theorem of Hilbert, Math. Z. 6 (1920), no.~3-4, 314--317.
\bibitem{HL1}               G.~H. Hardy and J.~E. Littlewood,
                            Some properties of fractional integrals. II, Math. Z. 34 (1932), no.~1, 403--439.
\bibitem{Heden}         H. Hedenmalm, B. Korenblum and K. Zhu,
                        Theory of Bergman spaces, Graduate Texts in Mathematics, 199, Springer, New York, 2000.
\bibitem{Koosis}            P.~J. Koosis, Introduction to $H_p$ spaces, second edition,
Cambridge Tracts in Mathematics, 115, Cambridge Univ. Press, Cambridge, 1998.
\bibitem{Landau}            E. Landau, I. Schur and G.~H. Hardy, A Note on a Theorem Concerning Series of Positive Terms: Extract from a Letter, J. London Math. Soc. 1 (1926), no.~1, 38--39.
\bibitem{LaPavlovic}        B. \L anucha, M.~T. Nowak and M. Pavlovi\'c, Hilbert matrix operator on spaces of analytic functions, Ann. Acad. Sci. Fenn. Math. 37 (2012), no.~1, 161--174.
\bibitem{Lin}          Q.~Z. Lin and H. Xie, Ces\`aro-type operators on derivative-type Hilbert spaces of analytic functions: the proof of a conjecture, J. Funct. Anal. 288 (2025), no.~6, Paper No. 110813, 22 pp.
\bibitem{Liouville}     J. Liouville,
                        M\'emoire sur quelques questions de g\'eom\'etrie et de m\'ecanique, et sur un nouveau genre de calcul pour r\'esoudre ces questions, Journal de l'\'Ecole Polytechnique, (1832), 13, Paris: 1-69.
\bibitem{LueckingProc}  D.~H. Luecking, Embedding derivatives of Hardy spaces into Lebesgue spaces, Proc. London Math. Soc. (3) 63 (1991), no.~3, 595--619.
\bibitem{MasMerchanRosa}    A. Mas, N. Merch\'an and E. de~la~Rosa, Generalized Ces\`aro operator acting on Hilbert spaces of analytic functions, Ann. Funct. Anal. 15 (2024), no.~3, Paper No. 56, 17 pp.
\bibitem{PavlovicLp}        M.~S. Mateljevi\'c{} and M. Pavlovi\'c, $L\sp{p}$-behaviour of the integral means of analytic functions, Studia Math. 77 (1984), no.~3, 219--237.
\bibitem{Miao}  J. Miao, The Ces\`aro operator is bounded on $H^p$ for $0< p<1$, Proc. Amer. Math. Soc. 116 (1992), no.~4, 1077--1079.
\bibitem{Moreno}        \'A.~M. Moreno, J.~\'A. Pel\'aez and E. de~la~Rosa, Fractional derivative description of the Bloch space, Potential Anal. 61 (2024), no.~3, 555--571.
\bibitem{Pan}       J. Pan, C. Tong, Z. Yang, Weighted  Ces\`aro type operators on the weighted Bergman spaces in the unit ball, https://arxiv.org/abs/2609.03243.
\bibitem{Pavlovic2004}M. Pavlovi\'c, Introduction to function spaces on the disk, Posebna Izdanja, 20, Mat. Inst. SANU, Belgrade, 2004.
\bibitem{Pavlovic2014}  M. Pavlovi\'c, Function classes on the unit disc, De Gruyter Studies in Mathematics, 52, De Gruyter, Berlin, 2014.
\bibitem{SS2014}    J.~\'A. Pel\'aez, Small weighted Bergman spaces, in  Proceedings of the Summer School in Complex and Harmonic Analysis, and Related Topics, 29--98, Publ. Univ. East. Finl. Rep. Stud. For. Nat. Sci., 22, Univ. East. Finl., Fac. Sci. For., Joensuu.
\bibitem{PelaezHp}  J.~\'A. Pel\'aez, Compact embedding derivatives of Hardy spaces into Lebesgue spaces, Proc. Amer. Math. Soc. 144 (2016), no.~3, 1095--1107.
\bibitem{PelaezDeLaRosa}    J.~\'A. Pel\'aez and E. de~la~Rosa,                                     Littlewood-Paley inequalities for fractional derivative on Bergman spaces, Ann. Fenn. Math. 47 (2022), no.~2, 1109--1130.
\bibitem{PR2014}    J.~\'A. Pel\'aez and J. R\"atty\"a, Weighted Bergman spaces induced by rapidly increasing weights, Mem. Amer. Math. Soc. 227 (2014), no.~1066, vi+124 pp.
\bibitem{PR2015}            J.~\'A. Pel\'aez and J. R\"atty\"a,                                     Embedding theorems for Bergman spaces via harmonic analysis, Math. Ann. 362 (2015), no.~1-2, 205--239.
\bibitem{PR2016}        J.~\'A. Pel\'aez and J. R\"atty\"a,
							Two weight inequality for Bergman projection, J. Math. Pures Appl. 105 (2016), 102--130.
\bibitem{PR2021}        J.~\'A. Pel\'aez and J. R\"atty\"a, Bergman projection induced by radial weight, Adv. Math. 391 (2021), Paper No. 107950, 70 pp.
\bibitem{PR2025} J.~\'A. Pel\'aez and J. R\"atty\"a, Small Hankel operators induced by measurable symbol acting on weighted Bergman spaces, https://arxiv.org/pdf/2407.04645.
\bibitem{PRS}               J.~\'A. Pel\'aez, J. R\"atty\"a and K. Sierra, Atomic decomposition and Carleson measures for weighted mixed norm spaces, J. Geom. Anal. 31 (2021), no.~1, 715--747.
\bibitem{PRWW}              A. Pennanen, J. R\"atty\"a, S. Wang and F. Wu,
                            Optimal off-diagonal upper estimates for Bergman reproducing
kernels, https://arxiv.org/pdf/2607.19959.
\bibitem{PeralaJGEA}        A. Per\"al\"a, Vanishing Bergman kernels on the disk, J. Geom. Anal. 28 (2018), no.~2, 1716--1727.
\bibitem{PeralaFrac}        A. Per\"al\"a, General fractional derivatives and the Bergman projection, Ann. Acad. Sci. Fenn. Math. 45 (2020), no.~2, 903--913.
\bibitem{PRW}               A. Per\"al\"a, J. R\"atty\"a{} and S. Wang,
                            Two-weight fractional derivative on Bloch and Bergman spaces, J. Geom. Anal. 35 (2025), no.~7, Paper No. 209, 18 pp.
\bibitem{Ransford}          T.~J. Ransford, Potential theory in the complex plane, London Mathematical Society Student Texts, 28, Cambridge Univ. Press, Cambridge, 1995.
\bibitem{Ross}          W.~T. Ross, The Ces\`aro operator, in Recent progress in function theory and operator theory, 185--215, Contemp. Math., 799, Amer. Math. Soc., RI.
\bibitem{Siskakis1}     A.~G. Siskakis, Composition semigroups and the Ces\`aro operator on $H^p$, J. London Math. Soc. (2) 36 (1987), no.~1, 153--164.
\bibitem{Siskakis2}     A.~G. Siskakis, The Ces\`aro operator is bounded on $H^1$, Proc. Amer. Math. Soc. 110 (1990), no.~2, 461--462.
\bibitem{Xie}              H. Xie, J.~M. Liu and Q.~Z. Lin, Ces\`aro-type operators on derivative-type Hilbert spaces of analytic functions II, J. Funct. Anal. 290 (2026), no.~5, Paper No. 111287, 27 pp.
\bibitem{ZhuFrac}       K. Zhu,
                        Bergman and Hardy spaces with small exponents, Pacific J. Math. 162 (1994), no.~1, 189--199.
\bibitem{Zhu}           K. Zhu, Operator theory in function spaces, second edition,
Mathematical Surveys and Monographs, 138, Amer. Math. Soc., Providence, RI, 2007.
\end{thebibliography}
\end{document}